\documentclass{amsart}
\usepackage{amsmath,amssymb,amsthm,amsfonts,amscd,mathrsfs,mathtools}
\usepackage[utf8]{inputenc}
\usepackage[hyphens]{url}
\usepackage[pagebackref=true]{hyperref}
\usepackage{tikz-cd}
\usepackage{tikz}
\usepackage{overpic}
\usepackage{caption}
\usepgflibrary{arrows}
\usetikzlibrary{shapes.arrows,trees,backgrounds,decorations.markings,shadows,calc, patterns, positioning}
\usepackage{enumitem}
\usepackage{multicol}

\setlist[itemize,description]{leftmargin=*}

\theoremstyle{plain}
    \newtheorem{thm}{Theorem}[section]
    \newtheorem{prop}[thm]  {Proposition}
    \newtheorem{lem}[thm]   {Lemma}
    \newtheorem{cor}[thm]   {Corollary}
    
    \newtheorem{athm}{Theorem}

\theoremstyle{definition}
    \newtheorem{rem}[thm]   {Remark}
    \newtheorem{defn}[thm]  {Definition}
    \newtheorem{nota}[thm]  {Notation}

\newcommand{\cA}{\mathcal{A}} 
\newcommand{\bB}{\mathbb{B}} 
\newcommand{\C}{\mathbb{C}} 
\newcommand{\fc}{\mathfrak{c}} 
\newcommand{\cD}{\mathcal{D}} 
\newcommand{\scD}{\mathscr{D}} 
\newcommand{\ud}{\underline{d}} 
\newcommand{\cF}{\mathcal{F}} 
\newcommand{\fri}{\mathfrak{i}} 
\newcommand{\cK}{\mathcal{K}} 
\newcommand{\N}{\mathbb{N}} 
\newcommand{\bQ}{\mathbb{Q}} 
\newcommand{\uQ}{\underline{Q}} 
\newcommand{\fm}{\mathfrak{m}} 
\newcommand{\scM}{\mathscr{M}} 
\newcommand{\R}{\mathbb{R}} 
\newcommand{\bs}{\mathbf{s}} 
\newcommand{\cS}{\mathcal{S}} 
\newcommand{\fS}{\mathfrak{S}} 
\newcommand{\scU}{\mathscr{U}} 
\newcommand{\uX}{\underline{X}} 
\newcommand{\cX}{\mathcal{X}} 
\newcommand{\cY}{\mathcal{Y}} 
\newcommand{\fw}{\mathfrak{w}} 
\newcommand{\Z}{\mathbb{Z}} 

\newcommand{\fM}{\mathfrak{M}} 
\newcommand{\Riem}{\mathfrak{Riem}} 
\newcommand{\fr}{\mathfrak{r}} 
\newcommand{\frst}{\fr_{\mathrm{st}}} 
\newcommand{\scF}{\mathscr{F}} 
\newcommand{\mscF}{\mathring{\scF}} 
\newcommand{\pr}{p} 
\newcommand{\mpr}{\mathring{\pr}} 
\newcommand{\cO}{\mathcal{O}} 
\newcommand{\ccO}{\bar{\cO}} 
\newcommand{\ccOud}{\ccO_{g,n}[\ud]} 
\newcommand{\tccOud}{\tilde{\ccO}_{g,n}[\ud]} 
\newcommand{\Harm}{\mathfrak{H}} 
\newcommand{\Harmud}{\Harm_{g,n}[\ud]} 
\newcommand{\cHarm}{\check{\Harm}} 
\newcommand{\cHarmud}{\cHarm_{g,n}[\ud]} 
\newcommand{\tHarm}{\tilde{\Harm}}
\newcommand{\tHarmud}{\tHarm_{g,n}[\ud]}
\newcommand{\forg}{\theta} 
\newcommand{\MTSO}{\mathrm{MTSO}} 
\newcommand{\Teich}{\mathfrak{T}} 
\newcommand{\sgn}{\Sigma_{g,n}}
\newcommand{\crit}{\mathcal{P}} 
\newcommand{\MonPol}{\mathfrak{MonPol}}
\newcommand{\NMonPol}{\mathfrak{NMonPol}}
\newcommand{\cv}{\mathfrak{cv}} 
\newcommand{\ccv}{\check{\mathfrak{cv}}} 
\newcommand{\bc}{\mathfrak{bc}} 
\newcommand{\uzeta}{\underline{\zeta}} 
\newcommand{\stab}{\mathrm{stab}} 

\newcommand{\SP}{\mathrm{SP}}

\usepackage{mathbbol}
\DeclareSymbolFontAlphabet{\mathbb}{AMSb}
\DeclareSymbolFontAlphabet{\mathbbl}{bbold}
\newcommand{\one}{\mathbbl{1}} 
\newcommand{\geo}{^{\mathrm{geo}}} 
\newcommand{\cG}{\mathcal{G}} 
\newcommand{\Q}{\mathcal{Q}} 
\newcommand{\lc}{\mathrm{lc}} 
\newcommand{\fQ}{\mathfrak{Q}} 
\newcommand{\tr}{\mathbf{t}} 
\newcommand{\bl}{\mathbf{l}} 
\newcommand{\br}{\mathbf{r}} 
\newcommand{\klud}{(\!(\ud)\!)} 
\newcommand{\kld}[1]{(\!(#1)\!)} 
\newcommand{\ulambda}{\underline{\lambda}} 
\newcommand{\fP}{\mathfrak{P}} 

\newcommand{\Hur}{\mathrm{Hur}} 
\newcommand{\CHur}{\mathrm{CHur}} 
\newcommand{\mHurm}{\mathring{\mathrm{HM}}} 
\newcommand{\bHurm}{\breve{\mathrm{HM}}} 
\newcommand{\CmP}{\C\setminus P} 
\newcommand{\diamo}{\diamond} 
\newcommand{\bdel}{\breve{\partial}} 
\newcommand{\bdiamolr}{\breve{\diamo}^{\mathrm{lr}}} 
\newcommand{\zdiamleft}{z_{\diamond}^{\mathrm{l}}} 
\newcommand{\zdiamright}{z_{\diamond}^{\mathrm{r}}} 
\newcommand{\uU}{\underline{U}} 
\newcommand{\totmon}{\omega} 
\newcommand{\fU}{\mathfrak{U}} 
\newcommand{\fQext}{\fQ^{\mathrm{ext}}} 
\newcommand{\braiding}{\mathfrak{br}} 

\newcommand{\cR}{\mathcal{R}} 
\newcommand{\mcR}{\mathring{\cR}} 
\newcommand{\bH}{\mathbb{H}} 

\newcommand{\Arr}{\mathrm{Arr}} 
\newcommand{\del}{\partial} 

\newcommand{\cpt}{^{\infty}} 

\newcommand{\pa}[1]{\left(#1\right)}

\newcommand{\abs}[1]{\left|#1\right|}
\newcommand{\set}[1]{\left\{#1\right\}}
\newcommand{\sca}[1]{[\! [#1]\! ]}

\renewcommand{\phi}{\varphi}
\renewcommand{\epsilon}{\varepsilon}

\DeclareMathOperator{\Diff}{Diff}
\DeclareMathOperator{\Id}{Id}
\DeclareMathOperator{\Tel}{Tel}
\DeclareMathOperator{\colim}{colim}

\title[Moduli spaces as Hurwitz spaces]{Moduli spaces of Riemann surfaces as Hurwitz spaces}
\author{Andrea Bianchi}
\thanks{
This work was partially supported by the \emph{Deutsche
  Forschungsgemeinschaft} (DFG, German Research Foundation) under Germany’s
Excellence Strategy (\texttt{EXC-2047/1}, \texttt{390685813}),
by the
\emph{European Research Council} under the
European Union’s Horizon 2020 research and innovation programme (grant agreement No. \texttt{772960}),
and by the
\emph{Danish National Research Foundation} through the \emph{Copenhagen Centre for
Geometry and Topology} (\texttt{DNRF151}).
}
\email{anbi@math.ku.dk}
\address{Department of Mathematical Sciences, University of Copenhagen \newline
Universitetsparken 5, Copenhagen, 2100, Denmark}  
\date{\today}
\dedicatory{To Carl-Friedrich B\"odigheimer, on the occasion of his 65\textsuperscript{th} birthday}

\keywords{Hurwitz space, moduli space, Riemann-Roch, group completion}
\subjclass[2020]{
55P35,  
55P62,   
55R80, 
57T25.  
}

\begin{document}
\begin{abstract}
We consider the moduli space $\mathfrak{M}_{g,n}$ of Riemann surfaces of genus $g\ge0$ with $n\ge1$
ordered and directed marked points. For $d\ge 2g+n-1$ we show that $\mathfrak{M}_{g,n}$ is homotopy equivalent to a component
of the simplicial Hurwitz space $\mathrm{Hur}^{\Delta}(\mathfrak{S}_d^{\mathrm{geo}})$
associated with the partially multiplicative quandle $\mathfrak{S}_d^{\mathrm{geo}}$.
As an application, we give a new proof of the Mumford conjecture on the stable rational cohomology of moduli spaces of Riemann surfaces. We also provide a combinatorial model for the infinite loop space $\Omega^{\infty-2}\mathrm{MTSO}(2)$ of Hurwitz flavour.
\end{abstract}

\maketitle
\section{Introduction}
The aim of this article is to use the theory of generalised Hurwitz spaces from \cite{Bianchi:Hur1,Bianchi:Hur2,Bianchi:Hur3} to give
a combinatorial model for the moduli space $\fM_{g,n}$ of closed, connected Riemann surfaces of genus $g\ge0$ with $n\ge1$ ordered and directed
marked points; a \emph{directed} marked point is endowed with a non-zero tangent vector.
As an application, we reprove
the Mumford conjecture on the stable rational cohomology ring of moduli spaces \cite{Mumford}, originally proved by Madsen and Weiss \cite{MadsenWeiss}, which can be formulated as follows: there is an isomorphism of graded commutative $\bQ$-algebras
\[
 \lim_{g\to\infty}H^*(\fM_{g,1};\bQ)\cong\bQ[x_1,x_2,\dots]
\]
between the ring of stable rational cohomology classes of moduli spaces $\fM_{g,1}$, and a polynomial ring in infinitely many variables $x_i$, one in each even degree $2i$.

Our argument provides also a ``Hurwitz model'' for the space $\Omega^{\infty-2}\MTSO(2)$, where $\MTSO(2)$ is the spectrum used by Madsen and Weiss to prove the Mumford conjecture.

\subsection{Classical Hurwitz spaces and moduli spaces}
The connection between Hurwitz spaces and moduli spaces goes back to Hurwitz himself \cite{Hurwitz}: for $d\ge2$ and $k\ge d-1$, the classical Hurwitz space $\CHur_{\fS_d,k}^c$ parametrises equivalence classes $[\cS,f,\mathfrak{u}]$, where
\begin{itemize}
 \item $\cS$ is a connected, punctured Riemann surface;
 \item $f\colon\cS\to\C$ is a branched cover of degree $d$, admitting exactly  $k$ distinct branch values in $\C$, all of which are \emph{simple};
 \item $\mathfrak{u}\colon f^{-1}(\bH_{\le\upsilon})\cong\set{1,\dots,d}\times\bH_{\le\upsilon}$ is a trivialisation of $f$ over some lower 
half-plane $\bH_{\le\upsilon}:=\set{z\in\C\,|\,\Im(z)\le\upsilon}$ that is disjoint from all branch values, for a suitable $\upsilon\in\R$.
\end{itemize}
The equivalence relation is given by declaring $(\cS,f,\mathfrak{u})$ and $(\cS',f',\mathfrak{u}')$ equivalent if there is a biholomorphism $\chi\colon\cS\cong\cS'$ such that $f'=f\circ\chi$ and $\mathfrak{u}\equiv\mathfrak{u}'\circ\chi$ on $f^{-1}(\bH_{\le\min(\upsilon,\upsilon')})$.

Our notation here is the same used in \cite{EVW:homstabhur}: the ``$\mathrm{C}$'' indicates that we restrict to \emph{connected} branched covers, i.e. we require $\cS$ to be connected; the ``$c$'' denotes the conjugacy class of transpositions in the symmetric group $\fS_d$, where local monodromies take value: this ensures that branch values are simple.

Given $[\cS,f,\mathfrak{u}]\in \CHur_{\fS_d,k}^c$, represented by $(\cS,f,\mathfrak{u})$, we let $P=\set{z_1,\dots,z_k}\subset\C$ be the set of simple branch values of $f$, and restrict $f$ to a genuine $d$-fold cover
\[
f\colon f^{-1}(\CmP)\to \CmP.
\]
Choosing a basepoint $*_P\in\CmP$ ``below $P$'', i.e. inside the lower halfplane $\bH_{\le\upsilon}$ over which the covering is trivialised, we can consider the action of $\pi_1(\CmP,*_P)$ on the fibre $f^{-1}(*_P)\cong\set{1,\dots,d}$, and thus obtain a homomorphism of groups $\phi\colon\pi_1(\CmP,*_P)\to\fS_d$. In particular we can evaluate $\phi$ at the class of a simple loop $\gamma$ in $\CmP$ spinning clockwise around all points of $P$, obtaining the \emph{total monodromy of $[\cS,f,\mathfrak{u}]$}, which is a permutation in $\fS_d$: by \emph{loop} we mean a continuous map $\gamma\colon[0,1]\to\CmP$ sending both $0$ and $1$ to $*_P$, and by \emph{simple} we mean that $\gamma$ induces an injective map $[0,1]/\set{0,1}\to\CmP$.

Hurwitz proved that a complete invariant for connected components of $\CHur_{\fS_d,k}^c$ is given by the total monodromy. More precisely, for $\sigma\in\fS_d$
denote by $\CHur_{\fS_d,k,\sigma}^c$ the subspace of $\CHur_{\fS_d,k}^c$ of configurations with total monodromy equal to $\sigma$. 
Moreover, let $N(\sigma)$ be the word length norm of $\sigma$ with respect to all transpositions in $\fS_d$, i.e. the minimal $r\ge0$ such that $\sigma$ can be written as a product of $r$ transpositions. Then $\CHur_{\fS_d,k,\sigma}^c$ is connected and non-empty if $k-(2d-2-N(\sigma))$ is even and non-negative, and $\CHur_{\fS_d,k,\sigma}^c$ is empty otherwise.

In the following discussion we restrict for simplicity to the case $n=1$, i.e. we focus on the connection between Hurwitz spaces and moduli spaces of Riemann surfaces with a single marked point. Let $\lc_d$ denote the permutation $(1,2,\dots,d)\in\fS_d$ consisting of a unique \emph{long cycle}, and note that $N(\lc_d)=d-1$. There is a continuous, forgetful map
\[
\forg_{\CHur}\colon \CHur_{\fS_d,d-1+2g,\lc_d}^c\to\fM_{g,1},
\]
sending the class $[\cS,f,\mathfrak{u}]$ of a branched cover $f\colon\cS\to\C$ to the total space $\bar\cS$ of its \emph{completion} $f\colon \bar\cS\to\C P^1$; the smooth, closed Riemann surface $\bar\cS$ can be naturally equipped with a directed marked point $Q$, namely the (unique) preimage along $f$ of $\infty\in\C P^1$, or in other words, the unique point in $\bar\cS\setminus\cS$; a tangent vector $X$ at $Q\in\bar\cS$ can also be chosen in a canonical way, pointing in the direction of $\set{1}\times\bH_{\le\upsilon}\subset\cS\subset\bar\cS$, where we use $\mathfrak{u}$, the trivialisation of $f$ over some lower half-plane in $\C$, to embed $\set{1,\dots,d}\times\bH_{\le\upsilon}$ into $\cS$. The fact that both the source and the target of the map $\forg$ are connected spaces implies that $\forg$ is a $0$-connected map.\footnote{In fact, 
as mentioned in \cite{EVW:homstabhur}, Severi first proved that $\fM_g$ is connected by constructing a surjective map similar to $\forg$ with source a connected component of a Hurwitz space. At the time of Severi, Teichm\"uller theory, and in particular the contractibility of $\Teich_g$, was not available yet!}. A natural question arises: \emph{how highly connected is the map $\forg$?}

For $d=2$ and $g=1$, the map $\forg_{\CHur}\colon \CHur_{\fS_2,3,\lc_2}^c\to\fM_{1,1}$ happens to be a homotopy equivalence: this follows from the well-known fact that every closed Riemann surface $\bar\cS$ of genus $1$ with a marked point $Q$ admits a unique involution fixing the point $Q$: in fact $(\bar\cS,Q)$ can be considered as an elliptic curve, and the involution is the inversion in the group structure. However, it turns out that the previous is a very special case, and in fact $\forg_{\CHur}$ is not even a $1$-connected map for $g\ge2$:
note for instance that $H_1(\CHur_{\fS_d,d-1+2g,\lc_d}^c)$ is infinite, as it admits a surjection onto $\Z$, regarded as $H_1$ of the unordered configuration space of $d-1+2g$ points in $\C$; on the other hand
$H_1(\fM_{g,1})$ is finite for $g\ge2$.

In fact $\forg_{\CHur}\colon \CHur_{\fS_d,d-1+2g,\lc_d}^c\to\fM_{g,1}$ is the restriction of another map, whose target is still the moduli space $\fM_{g,1}$, but whose domain of definition is a larger space, which we can think of as a \emph{completion} of $\CHur_{\fS_d,d-1+2g,\lc_d}^c$. In this article, for $g\ge0$ and $d\ge1$, we introduce the space $\ccO_{g,1}[d]$: it contains
equivalence classes $[\cS,f,\mathfrak{u}]$ of $d$-fold branched covers $f\colon\cS\to\C$ equipped with a trivialisation $\mathfrak{u}$ over some lower half-plane $\bH_{\le\upsilon}$,
such that the total monodromy is equal to $\lc_d\in\fS_d$, and such that the total space $\bar\cS$ of the compactified covering $f\colon\bar\cS\to\C P^1$ is a smooth Riemann surface of genus $g$. Comparing with the definition of $\CHur_{\fS_d,d-1+2g,\lc_d}^c$, there is only one requirement which we drop: we no longer require that all branch values of $f$ in $\C$ be \emph{simple}.
As we will see, it is possible to identify $\CHur_{\fS_d,d-1+2g,\lc_d}^c$ with an open dense subspace of $\ccO_{g,1}[d]$, and there is a forgetful
map $\forg_{\ccO}\colon \ccO_{g,1}[d]\to\fM_{g,1}$ extending $\forg_{\CHur}$.

The space $\ccO_{g,1}[d]$ is introduced, in a similar way as in \cite{Bianchi:PhD},
as a closed subspace of a space $\cHarm_{g,1}[d]$, which is a mild variation of the \emph{slit configuration space}
$\Harm_{g,1}[d]$. The latter space $\Harm_{g,1}[d]$ was introduced by B\"odigheimer in \cite{Boedigheimer90}, in the case $d=1$, to give a combinatorial model of $\fM_{g,1}$; the construction was generalised in \cite{BH} to higher values of $d$
(see also \cite[Chapter 6]{Bianchi:PhD} for a short discussion).

The space $\cHarm_{g,1}[d]$ is also endowed with a natural forgetful map $\check\forg_{\Harm}$ towards $\fM_{g,1}$, which happens to be a homotopy equivalence for all $d\ge1$; the restriction of this map to $\ccO_{g,1}[d]$ is precisely the map $\forg_{\ccO}$ considered above, which further restricts to $\forg_{\CHur}$ on $\CHur_{\fS_d,d-1+2g,\lc_d}^c$.

\subsection{Statement of results}
We highlight five results from this article, focusing for simplicity on the case $n=1$.
\begin{athm}[Theorem \ref{thm:main1} for $n=1$]
 \label{thm:main1intro}
Let $d\ge 2g\ge0$; then the canonical map
\[
\forg_{\ccO}\colon \ccO_{g,1}[d]\to \fM_{g,1}
\]
is a homotopy equivalence: it starts in the completion $\ccO_{g,1}[d]$ of the classical Hurwitz space $\CHur_{\fS_d,d-1+2g,\lc_d}^c$ considered above; and it ends in the moduli space $\fM_{g,1}$.
\end{athm}
Before stating the second result, we recall the notion of partially multiplicative quandle (PMQ), introduced in \cite[Section 2]{Bianchi:Hur1}: roughly speaking, a PMQ is a set with well-behaved binary operations of conjugation and of partial product, satisfying conditions similar to those that the usual conjugation and product operations in a group satisfy.
For an \emph{augmented} PMQ $\Q$ (see \cite[Definition 4.9]{Bianchi:Hur1}), a \emph{simplicial Hurwitz space} $\Hur^{\Delta}(\Q)$ was constructed
in \cite[Definition 6.13]{Bianchi:Hur1}. The augmented PMQ $\fS_d\geo$ was introduced in \cite[Definition 7.1]{Bianchi:Hur1}, and plays a central role in this article: its underlying set is the symmetric group $\fS_d$, the operation of conjugation coincides with group conjugation, the partial product is only defined for certain couples of permutations, satisfying a \emph{geodesic} condition, and coincides with the usual product of permutations when it is defined.
\begin{athm}[Theorem \ref{thm:main2}]
\label{thm:main2intro}
Let $d\ge1$; then the PMQ $\fS_d\geo$ is a \emph{Poincar\'e} PMQ. In other words, all components of the simplicial Hurwitz space $\Hur^{\Delta}(\fS_d\geo)$ are topological manifolds.
\end{athm}
The connection between simplicial Hurwitz spaces and moduli spaces is established combining Theorem \ref{thm:main1intro} and the following theorem.
\begin{athm}[Theorem \ref{thm:main3} for $n=1$]
\label{thm:main3intro}
The space $\ccO_{g,1}[d]$ is homeomorphic to the
connected component $\Hur^{\Delta}(\fS_d\geo)(\kld{d}_g)$
of the space $\Hur^{\Delta}(\fS_d\geo)$. 
\end{athm}
Here $\kld{d}_g$ is a certain element of the completion $\widehat{\fS_d\geo}$ of the PMQ $\fS_d\geo$. Recall that for a generic augmented PMQ $\Q$ we have a completion $\hat\Q$, and the connected components of $\Hur^{\Delta}(\Q)$ are classified by the set $\hat\Q$, along the $\hat\Q$-valued total monodromy \cite[Theorem 6.14]{Bianchi:Hur1}. In our case, we are considering the connected component of $\Hur^{\Delta}(\fS_d\geo)$
of configurations whose $\widehat{\fS_d\geo}$-valued total monodromy is equal to $\kld{d}_g$.
The element $\kld{d}_g$ can for instance be expressed as the product
\[
\kld{d}_g=\widehat{(1,2)}\widehat{(1,2)}\dots\widehat{(1,2)}\widehat{(2,3)}\widehat{(3,4)}\dots\cdot\widehat{(d-1,d)}\in\widehat{\fS_d\geo},
\]
where at the beginning the factor $\widehat{(1,2)}$ is repeated in total $2g+1$ times. The element
$\kld{d}_g$ can also be expressed in the notation of \cite[Proposition 7.13]{Bianchi:Hur1} as
\[
\kld{d}_g=\pa{\lc_d\,;\,\set{1,\dots,d}\,;\,d-1+2g}.
\]
We remark that the single element $\kld{d}_g\in\widehat{\fS_d\geo}$ contains three pieces of information:
\begin{itemize}
 \item the classical $\fS_d$-valued total monodromy: in our case, the permutation $\lc_d$;
 \item a partition of the set $\set{1,\dots,d}$ into subsets: in our case, the trivial partition with a single subset,
 corresponding to the requirement that total spaces of branched covers be connected surfaces;
 \item the number of simple branch values seen in the generic case, split as a sum of numbers $\ge0$ corresponding to the pieces of the partition: in our case, the single number $d-1+2g$.
\end{itemize}

We can use the machinery of Hurwitz spaces to identify the stable homology of moduli spaces with that of an explicit double loop space: the following is the main result of the article.
\begin{athm}[Theorem \ref{thm:main4}]
 \label{thm:main4intro}
The stable homology groups $\mathrm{colim}_{g\to\infty}H_*(\fM_{g,1})$ are isomorphic to $H_*(\Omega_0^2\bB_\infty)$,
where $\bB_\infty$ is the homotopy colimit of a certain sequence of \emph{Hurwitz-Ran} spaces, and $\Omega^2_0$ is a component of the double loop space of $\bB_\infty$.
\end{athm}
We summarise here the argument, providing also the definition of $\bB_\infty$. We replace each of the spaces $\ccO_{g,1}[d]\cong\Hur^{\Delta}(\fS_d\geo)_{\kld{d}_g}$ with the homotopy equivalent space
$\mHurm(\fS_d\geo)_{\kld{d}_g}$; here, following \cite[Definition 2.4]{Bianchi:Hur3}, we consider the space $\mHurm(\fS_d\geo)$, which is a strict topological monoid counterpart of the simplicial Hurwitz space $\Hur^\Delta(\fS_d\geo)$, obtained through the theory of Hurwitz-Ran spaces from \cite{Bianchi:Hur2}.
We consider then a double-indexed diagram of spaces $\mHurm(\fS_d\geo)_{\kld{d}_g}$, for $d\ge2 $ and $g\ge0$, with stabilising maps increasing either $g$ or $d$ by 1. We identify the following colimits of homology groups
\[
 \mathrm{colim}_{g\to\infty} H_*(\fM_{g,1})\cong\mathrm{colim}_{g,d\to\infty}H_*\pa{\mHurm(\fS_d\geo)_{\kld{d}_g}}
\]
using Theorem \ref{thm:main1intro} and a diagonal argument. We then use the group-completion theorem \cite{McDuffSegal, FM94}
together with \cite[Theorem 4.19]{Bianchi:Hur3}, and we rewrite the second colimit as
\[
\begin{split}
 \mathrm{colim}_{g,d\to\infty}H_*\pa{\mHurm(\fS_d\geo)_{\kld{d}_g}}&=\mathrm{colim}_{d\to\infty}
 \pa{\mathrm{colim}_{g\to\infty}H_*\pa{\mHurm(\fS_d\geo)_{\kld{d}_g}}}\\
 &\cong\mathrm{colim}_{d\to\infty} H_*\pa{\Omega_0^2\Hur_+(\cR,\del\cR;\fS_d\geo,\fS_d)_{\one}}\\
 &\cong H_*(\Omega^2_0\bB_\infty).
\end{split}
\]
Here $\Hur_+(\cR,\del\cR;\fS_d\geo,\fS_d)_{\one}$ is an instance of a \emph{relative} Hurwitz-Ran space, see Section \ref{sec:mumford}
for a brief explanation, and \cite{Bianchi:Hur2} for the precise construction of Hurwitz-Ran spaces.
The space $\bB_\infty$ is defined as the homotopy colimit
\[
\bB_\infty:=\mathrm{hocolim}_{d\to\infty}\Hur_+(\cR,\del\cR;\fS_d\geo,\fS_d)_{\one},
\]
and the last isomorphism follows immediately.

We exploit Theorem \ref{thm:main4intro} to make a further computation, passing to rational coefficients and to cohomology. We use \cite[Theorem 6.1]{Bianchi:Hur3} (which is applicable in our situation thanks to Theorem \ref{thm:main3intro}) and a standard argument in rational homotopy theory to
compute the cohomology ring
\[
 H^*(\Omega^2_0\bB_\infty;\bQ)\cong\lim_{d\to\infty} H^*(\Omega_0^2\Hur_+(\cR,\del\cR;\fS_d\geo,\fS_d)_{\one} ;\bQ)\cong\bQ[x_1,x_2,\dots],
\]
where $x_i$ has degree $2i$.
For the first isomorphism we check that the Mittag-Leffler condition is satisfied, so that no $\lim^1$ terms occur.
Thus $\lim_{g\to\infty} H^*(\fM_{g,1};\bQ)$ is isomorphic
to the polynomial ring $\bQ[x_1,x_2,\dots]$, which is a formulation of the Mumford conjecture;
see Corollary \ref{cor:mumford}.

We conclude the article with the following theorem, which compares our proof of the Mumford conjecture with the original one, due to Madsen and Weiss \cite{MadsenWeiss}, and relying on the spectrum $\MTSO(2)$.
\begin{athm}[Theorem \ref{thm:main5}]
 \label{thm:main5intro}
The space $\bB_\infty$ is homotopy equivalent to the infinite loop space $\Omega^{\infty-2}\MTSO(2)$.
where $\MTSO(2)$ is the spectrum used by Madsen and Weiss in \cite{MadsenWeiss}.
\end{athm}
Thus  $\bB_\infty$ provides a ``Hurwitz model'' for the infinite loop space of the double suspension of $\MTSO(2)$.

\subsection{Motivation}
This is the fourth and final article in a series about Hurwitz spaces. We apply the machinery of simplicial Hurwitz spaces $\Hur^{\Delta}(\fS_d\geo)$ with monodromies in the PMQ $\fS_d\geo$ to give a new combinatorial model of moduli spaces $\fM_{g,n}$ of Riemann surfaces. The model is handy enough to allow a new proof of the Mumford conjecture.

Our hope is to exploit further the simplicial Hurwitz spaces $\Hur^{\Delta}(\fS_d\geo)$ to obtain information about the \emph{unstable} homology of moduli spaces. As discussed in \cite[Section 6]{Bianchi:Hur1}, leveraging on the fact that $\fS_d\geo$ is a Poincar\'e PMQ, one can in principle compute the homology of any component of $\Hur^{\Delta}(\fS_d\geo)$ by using an explicit, finite chain complex of \emph{arrays}, which is in spirit similar to the Fadell-Neuwirth-Fuchs chain complex computing the homology of classical configuration spaces. We hope that some new, interesting homological information about moduli spaces of Riemann surfaces can be extracted in the future from this chain model.

\subsection{Acknowledgments}
This series of articles is a generalisation and a further development of my PhD thesis \cite{Bianchi:PhD}. I would like to dedicate this article to my PhD supervisor Carl-Friedrich B\"odigheimer for his advice and encouragement during and also after my PhD, and for introducing me to the world of slit configurations.

I am also grateful to Felix Boes for several explanations about slit configurations, and to Jesper Grodal, Ib Madsen and Nathalie Wahl for the suggestion of deriving Theorem \ref{thm:main5intro} from the argument of Theorem \ref{thm:main4intro}. Finally I would also like to thank Nathalie Wahl and the anonymous referee for some helpful comments on a first draft of the article.

\tableofcontents
\section{Preliminaries on moduli spaces}
\label{sec:slit}
\subsection{Moduli spaces of surfaces with directed marked points}
\label{sec:intro}
We fix $g\geq 0$ and $n\geq 1$ throughout the section; this includes the special case
$g=0$ and $n=1$, which will be slightly different and for which a few extra remarks will be needed.

We fix a closed, smooth, connected, orientable surface $\Sigma_g$ of genus $g$.
We fix $n$ distinct points $Q_1,\dots, Q_n\in\Sigma_g$, called \emph{marked points}, 
and for each $Q_i$ we fix a non-zero tangent vector $X_i\in T_{Q_i}\Sigma_g$ (see Figure \ref{fig:moduli_1}). We denote by $\Sigma_{g,n}$ the surface $\Sigma_g$ equipped with these $n$ directed marked points.
We abbreviate $\uQ=(Q_1,\dots,Q_n)$ and $\uX=(X_1,\dots,X_n)$. The open surface
$\sgn\setminus\set{Q_1,\dots,Q_n}$ is usually denoted by $\sgn\setminus\uQ$.
A sequence $\ud=(d_1,\dots,d_n)$ of numbers $d_i\geq 1$ is fixed throughout the section, and we denote by $d$ the sum $\sum_{i=1}^nd_i$. Finally, we set $h:=2g+n+d-2$ throughout the section.
\begin{figure}[ht]
 \centering
\begin{overpic}[scale=0.7]{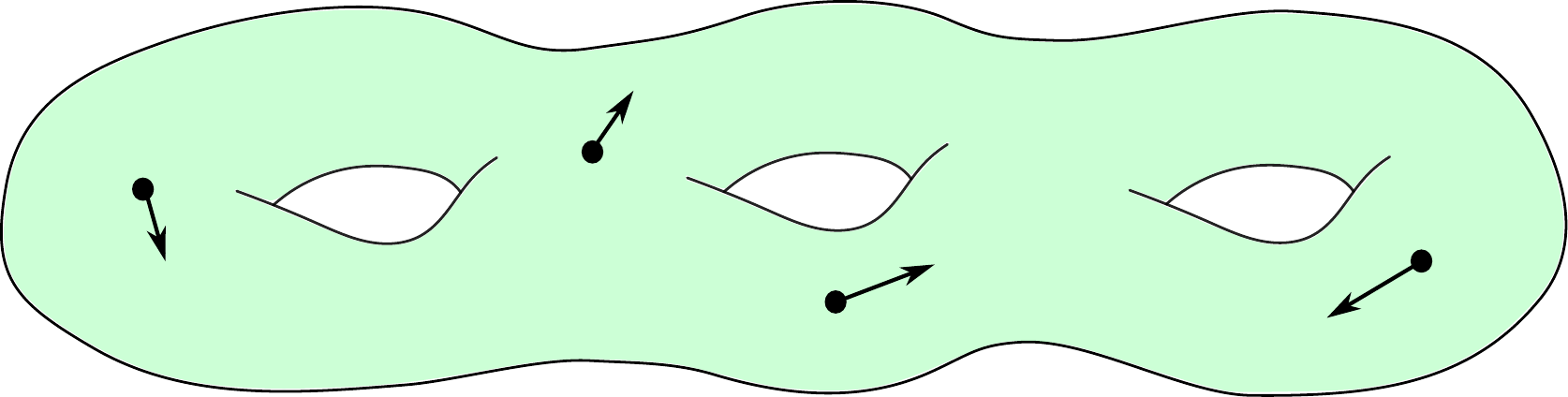}
\put(7,15){$Q_1$}
\put(11,7){$X_1$}
\put(39,20){$X_2$}
\put(34,13){$Q_2$}
\put(48,6){$Q_3$}
\put(58,9){$X_3$}
\put(92,7){$Q_4$}
\put(84,2){$X_4$}
\end{overpic}
 \caption{A surface of genus 3 with 4 marked points.}
 \label{fig:moduli_1}
\end{figure}
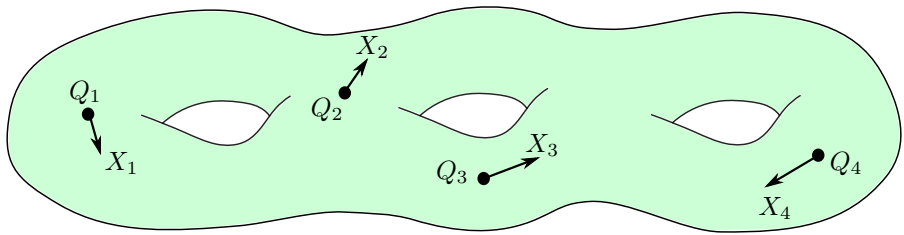
\begin{defn}
 \label{defn:mcg}
We denote by $\Diff_{g,n}$ the topological group of diffeomorphisms of $\Sigma_g$ that preserve the orientation and fix all points $Q_i$ and all vectors $X_i$.
The mapping class group $\Gamma_{g,n}$ is the group $\pi_0(\Diff_{g,n})$ of isotopy classes of such diffeomorphisms. 
\end{defn}

\begin{defn}
 \label{defn:fM}
The moduli space $\fM_{g,n}$ contains equivalence classes of Riemann structures $\fr$ on $\Sigma_{g,n}$: two Riemann structures $\fr$ and $\fr'$
on $\Sigma_{g,n}$ are \emph{equivalent} if there exists a diffeomorphism $\psi\in\Diff_{g,n}$ which is a biholomorphism from $(\sgn,\fr)$ to $(\sgn,\fr')$.

More formally, we consider the action of the topological group $\Diff_{g,n}$ on the space
$\Riem(\sgn)$ of Riemann structures on $\sgn$;
the moduli space $\fM_{g,n}$ is the orbit space $\Riem(\sgn)/\Diff_{g,n}$. 
\end{defn}
In a two-step process, one usually first defines the Teichm\"uller space $\Teich_{g,n}=\mathfrak{Riem}(\sgn)/\Diff^{\sim\Id}_{g,n}$ as the quotient by the subgroup of $\Diff_{g,n}$ of diffeomorphisms isotopic to the identity.

For $g>0$ or $n>1$, the space $\Teich_{g,n}$ is homeomorphic to $\R^{6g-6+4n}$, and one then defines $\fM_{g,n}$ as the quotient of $\Teich_{g,n}$ by the free, properly discontinuous and orientation-preserving action of $\Gamma_{g,n}$; thus $\fM_{g,n}$ is a classifying space for the group $\Gamma_{g,n}$ and it is an orientable manifold of real dimension $6g-6+4n$.

The case $g=0$ and $n=1$ is exceptional: $\Teich_{0,1}$ is a point, $\Gamma_{0,1}$ is the trivial group, and $\fM_{0,1}=\Teich_{0,1}/\Gamma_{0,1}$ is also just a point. The reader acquainted with the theory of topological stacks may prefer to consider $\fM_{0,1}$ as the classifying stack of the topological group $(\C,+,0)$, and treat here (and in the following) all cases $g\ge0$ and $n\ge1$ at the same time.

\begin{nota}
 \label{nota:modulus}
We usually denote by $\fm=[\fr]\in\fM_{g,n}$ a \emph{modulus} on $\sgn$, i.e. the equivalence class of a Riemann structure $\fr$ on $\sgn$.
\end{nota}

\begin{defn}
 \label{defn:normalchart}
A \emph{normal chart} at $(Q_i,X_i)$ for a Riemann structure $\fr$ on $\sgn$
is a holomorphic chart $w_i\colon U_i\to\C$ defined on a neighbourhood $Q_i\in U_i\subset\sgn$ with the following properties:
\begin{itemize}
 \item $w_i\colon Q_i\mapsto 0\in\C$;
 \item $Dw_i\colon X_i\mapsto \partial/\partial x$, where $x=\Re(z)$ and $y=\Im(z)$ are the standard real coordinates on $\C\simeq\R^2$, and $\partial/\partial x$ is the dual vector of $dx$ (informally, it is the horizontal unit vector pointing to right).
\end{itemize}
\end{defn}
Note that the composition of a normal chart with the inverse of another normal chart
is a holomorphic function $\fw=\fw(z)$ defined on some neighbourhood of $0\in\C$;
its Taylor expansion has the form $\fw(z)=z+\mathrm{h.o.t.},$
where the higher order terms correspond to powers of $z$ with exponent higher than $1$. In particular
the differential of a normal chart $Dw_i\colon T_{Q_i}\sgn\to T_0\C$ does not depend on the normal chart.
Note also that $X_i$ can be recovered from a normal chart $w_i$ as $Dw_i^{-1}(\partial/\partial x)$.

\begin{nota}
 \label{nota:CP1}
Our preferred model of surface $\Sigma_{0,1}$ is $\C P^1$:
the unique marked point $Q$ is the point at infinity $\infty=[1:0]$;
the vector $X\in T_{\infty} \C P^1$ is the unique vector such that the
chart $[a:b]\mapsto b/a\in\C$, defined on a neighbourhood of $\infty$,
is normal. The difference $\C P^1\setminus \set{Q}$ is identified with $\C$ through the usual map $[a:b]\mapsto a/b$. The standard Riemann structure on $\C P^1$ is denoted $\frst$.
\end{nota}
The action of $\Diff_{0,1}$ on $\mathfrak{Riem}(\Sigma_{0,1})$ is transitive but not free. The isotropy group of $\frst$ is the group of biholomorphisms $f\colon \C P^1\to\C P^1$ that fix $\infty$ and induce the identity on $T_{\infty}\C P^1$: this group is isomorphic to the additive group $\C$, where $\lambda\in\C$ corresponds to the translation $z\mapsto z+\lambda$ on $\C$, extended by $\infty\mapsto\infty$.

In the case $g> 0$ or $n> 1$, instead, the action of $\Diff_{g,n}$ on 
$\mathfrak{Riem}(\Sigma_{g,n})$ is free. See \cite{EarleSchatz} and \cite{Gardiner} for these classical results in Teichm\"uller theory.

In order to study the homotopy type of $\fM_{g,n}$
it is convenient to have a combinatorial model for this space;
by \emph{combinatorial model}
we mean in great generality a pair of finite CW-complexes $(\cX,\cX')$ whose definition (cells and
attaching maps) is combinatorial in flavour, and such that $\cX\setminus\cX'$ is homotopy equivalent to $\fM_{g,n}$.

In Subsection \ref{subsec:harm} we briefly recall the model $\Harmud$ of slit configurations, and in Section \ref{sec:ccO} we will describe the Hurwitz model $\ccOud$.

\subsection{The theorem of Riemann-Roch}
\label{subsec:RiemannRoch}
\begin{defn}
\label{defn:stddivisor} 
A divisor $\mathcal{D}$ on $\sgn$ is an element of the free abelian group generated by points of $\sgn$.
The degree of $\mathcal{D}=\sum_{i=1}^k\lambda_iP_i$ is $\deg(\cD)=\sum_{i=1}^k\lambda_i$.

We say that $\mathcal{D}$ is $\uQ$-supported if it is of the form $\sum_{i=1}^n\lambda_i Q_i$, with coefficients $\lambda_i\geq 1$.
We denote by $D=D(\ud)=\sum_{i=1}^n d_i Q_i$
the $\uQ$-supported divisor of $\sgn$ associated with the sequence $\ud$; note that
$\deg(D)=d=\sum_{i=1}^nd_i$. Similarly, we denote by $D+\uQ$ the $\uQ$-supported divisor $\sum_{i=1}^n (d_i+1)Q_i$.
\end{defn}
Let $\fr$ be a Riemann structure on $\sgn$, let $\mathcal{D}=\sum_{i=1}^k\lambda_iP_i$ be a divisor and consider the complex vector space
$\cO(\fr,\mathcal{D})$ containing all holomorphic sections of the holomorphic line bundle associated with $\mathcal{D}$. Concretely, an element of $\cO(\fr,\mathcal{D})$
is a meromorphic function $f\colon(\Sigma_g,\fr)\to\C P^1$ such that
$f$ is holomorphic away from $\set{P_1,\dots,P_k}$ and,
for all $1\leq i\leq k$,
$f$ has a pole of order at most $\lambda_i$ at $P_i$; this means that if $\lambda_i<0$, then $f$ must have a zero of
order at least $-\lambda_i$ at $P_i$.
The complex dimension of $\cO(\fr,\mathcal{D})$ depends in general on $\fr$ and $\mathcal{D}$.

Recall that we can associate with every Riemann surface $(\sgn,\fr)$
a \emph{canonical divisor} $K(\fr)$ of degree $2g-2$, whose corresponding holomorphic line bundle is isomorphic to the holomorphic cotangent bundle $T^*\sgn$. The divisor $K(\fr)$ is uniquely determined up to principal divisors.
\begin{thm}[Riemann-Roch]
\label{thm:RiemannRoch}
Let $\fr$ be a Riemann structure on $\sgn$, let $\mathcal{D}$ be any divisor on $\sgn$, and let $K(\fr)$ be a canonical divisor associated with the Riemann surface $(\sgn,\fr)$.
Then
\[
 \dim_{\C}\cO(\fr,\mathcal{D})-\dim_{\C}\cO(\fr,K(\fr)-\mathcal{D})=\deg(\cD)-g+1.
\]
\end{thm}
Note that the inequality $\dim_{\C}\cO(\fr,D)\geq \deg(\cD)-g+1$ always holds;
if moreover $\deg(\cD)\geq 2g-1$, then
$K(\fr)-\deg(\cD)$ has degree $\leq -1$, so that $\cO(\fr,K(\fr)-D)=0$ and we obtain the
equality $ \dim_{\C}\cO(\fr,\cD)= \deg(\cD)-g+1$.

\subsection{The parallel slit complex}
\label{subsec:harm}
The model $\Harmud$ is constructed by considering Riemann surfaces $(\sgn,\fr)$ equipped with a harmonic
function $u\colon\sgn\setminus\uQ\to\R$ having prescribed behaviour near the marked points.
One can cut the surface along the critical lines of the flow associated
with $-\nabla u$, the negative gradient of $u$; the cut surface consists of $d$ contractible components, each of which can be biholomorphically embedded in $\C$
as the complement of a finite collection of slits. A \emph{slit} is a horizontal halfline in $\C$ running to the left,
i.e. of the form $\set{z_0-t\,|\, t\geq 0}$ for some $z_0\in\C$ (see Figure \ref{fig:moduli_2}).

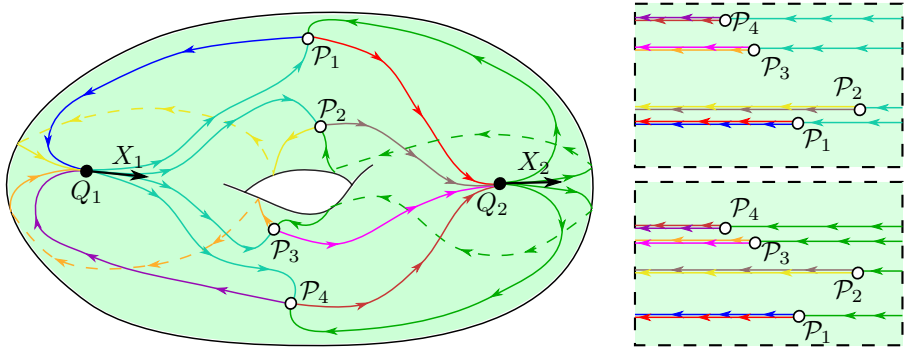
\begin{figure}[ht]
 \centering
\begin{overpic}[scale=0.7]{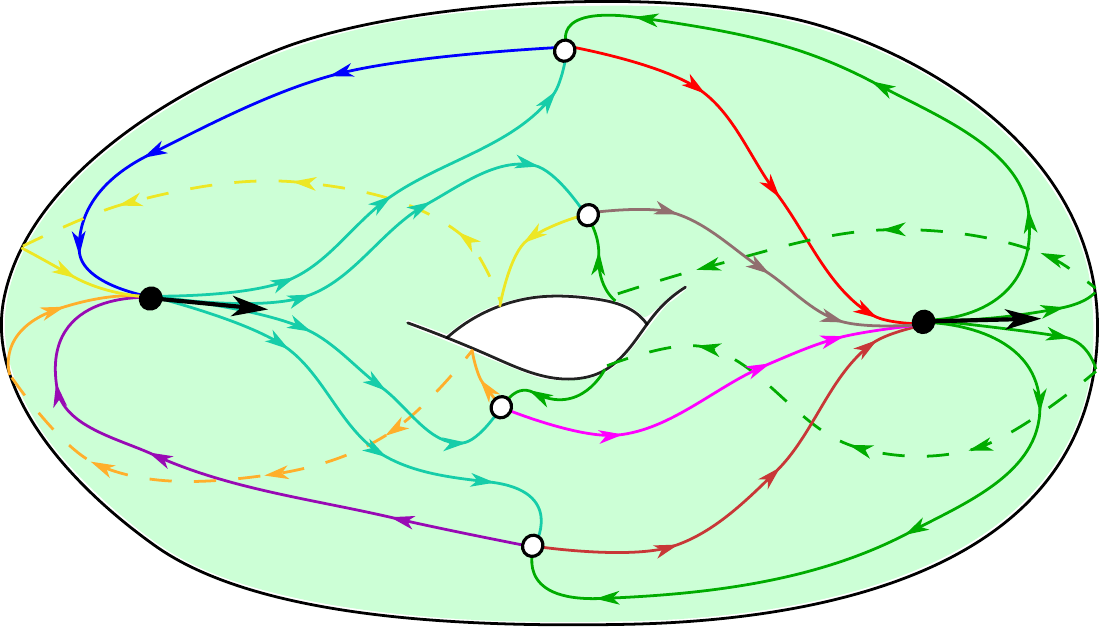}
\put(11,25){$Q_1$}
\put(18,31){$X_1$}
\put(45,15){$\crit_3$}
\put(52,48){$\crit_1$}
\put(50,8){$\crit_4$}
\put(53,39){$\crit_2$}
\put(87,30){$X_2$}
\put(80,23){$Q_2$}
\end{overpic}
\hspace{0.3cm}
\begin{overpic}[scale=0.7]{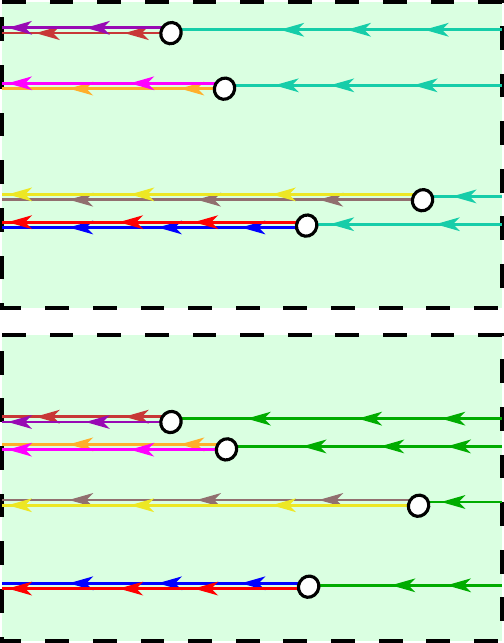}
\put(58,14){$\crit_2$}
\put(37,25){$\crit_3$}
\put(49,2){$\crit_1$}
\put(28,38){$\crit_4$}
\put(58,73){$\crit_2$}
\put(37,79){$\crit_3$}
\put(49,58){$\crit_1$}
\put(28,92){$\crit_4$}
\end{overpic}
 \caption{On left, the Riemann surface $(\Sigma_{1,2},\fr)$, together with a harmonic function
 $u\colon\Sigma_{1,2}\setminus\set{Q_1,Q_2}\to\R$ having directed poles of order 1 at $Q_1$ and $Q_2$: the critical points of $u$ are
 $\crit_1,\crit_2,\crit_3,\crit_4$, and the flow lines of $-\nabla u$ entering
 or exiting from the critical points are shown.
 On right, the associated slit picture. Here $d=2$ and $d_1=d_2=1$.}
 \label{fig:moduli_2}
\end{figure}

The process of dissecting a Riemann surface as described above, in order to represent its parts as tame open subsets of the plane, is called \emph{Hilbert uniformisation} \cite{Hilbert}. B\"{o}digheimer \cite{Boedigheimer90} considers all moduli of Riemann surfaces at the same time and describes the model $\Harm_{g,1}[1]$ for the moduli space $\fM_{g,1}$;
the same technique works for surfaces with more than one marked point and with higher values for the numbers $d_i$. The formal definition of $\Harmud$ requires some preparation. 

\begin{defn}
 \label{defn:udharm}
Let $\fr$ be a Riemann structure on $\sgn$. A harmonic function $u\colon\sgn\setminus\uQ\to\R$ is \emph{$\ud$-directed} if for all $1\leq i\leq n$ it has a \emph{directed} pole of order $d_i$ at $Q_i$ in the following sense: in a normal chart $w_i\colon U_i\to\C$ around $Q_i$ we can write
  \[
   u(w_i)=\Re\pa{\frac{1}{w_i^{d_i}}+\mathrm{h.o.t.}}-B_i\log\abs{w_i}
  \]
for some real number $B_i$. The higher order terms correspond to powers of $w_i$ with exponent higher than $-d_i$. The word ``directed'' refers to the fact that the coefficient of $w_i^{-d_i}$ is required to be $1$, and not another number in $\C^*$.
\end{defn}
\begin{defn}
\label{defn:critgraph}
The critical graph of a harmonic function $u\colon(\sgn,\fr)\setminus\uQ\to\R$, denoted $\cK_0(u)\subset\sgn$, is defined by considering the flow lines of $-\nabla u$, the negative gradient of $u$ on $\sgn\setminus\uQ$. The vector field $-\nabla u$, and hence the parametrisation of the flow lines, depend on a choice of Riemannian metric in the conformal class $\fr$; but the decomposition of $\sgn\setminus\uQ$ into maximal flow lines is independent of this choice.
The graph $\cK_0$ is the union of all critical points of $u$ (points where $du=0$), all flow lines having a critical point of $u$ as negative limit, and all marked points $Q_i$. See Figure \ref{fig:moduli_5}.
\end{defn}
 \begin{figure}[ht]
 \centering
\begin{overpic}[scale=0.7]{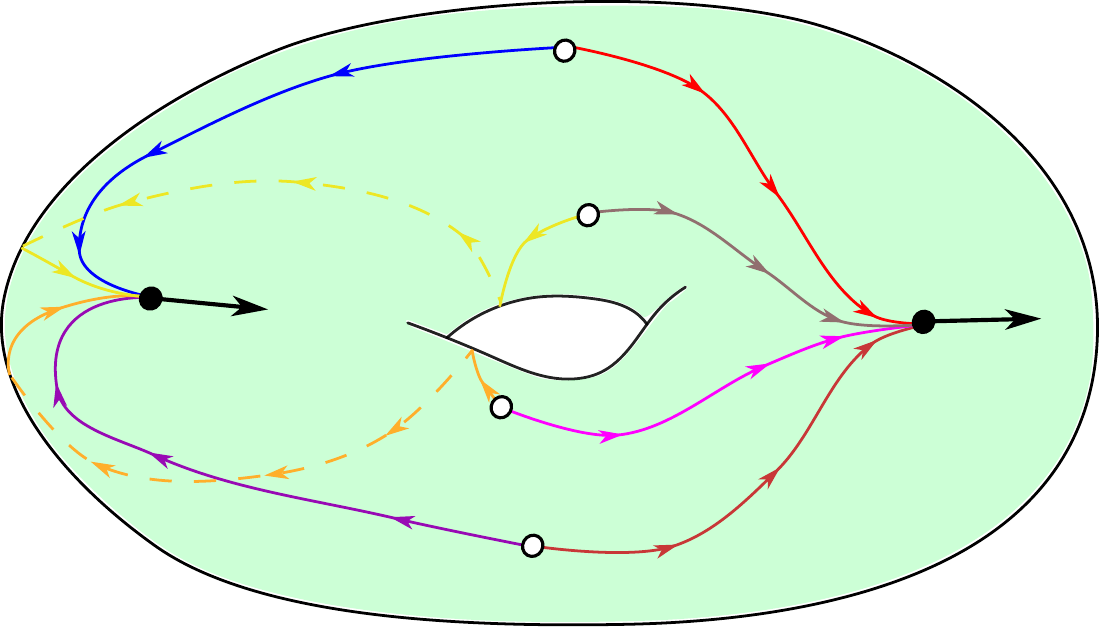}
\put(11,25){$Q_1$}
\put(18,31){$X_1$}
\put(45,15){$\crit_3$}
\put(52,48){$\crit_1$}
\put(50,8){$\crit_4$}
\put(53,39){$\crit_2$}
\put(87,30){$X_2$}
\put(80,23){$Q_2$}
\end{overpic}
\hspace{0.3cm}
\begin{overpic}[scale=0.7]{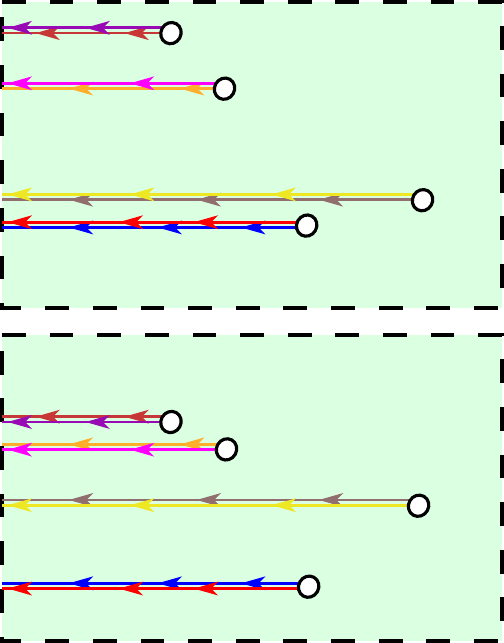}
\put(58,14){$\crit_2$}
\put(37,25){$\crit_3$}
\put(49,2){$\crit_1$}
\put(28,38){$\crit_4$}
\put(58,73){$\crit_2$}
\put(37,79){$\crit_3$}
\put(49,58){$\crit_1$}
\put(28,92){$\crit_4$}
\end{overpic}
 \caption{On left, the critical graph $\cK_0\subset\Sigma_{1,2}$ in the example of Figure \ref{fig:moduli_2}.
 It is an oriented graph with $6$ vertices and $8$ edges embedded in the surface $\Sigma_{1,2}$.
 The complement $\Sigma_{1,2}\setminus \cK_0$ is endowed with a holomorphic embedding into $\C_1\sqcup \C_2$;
 the image of this embedding is the complement of the infinite halflines, called \emph{slits}, shown on right.}
 \label{fig:moduli_5}
 \end{figure}
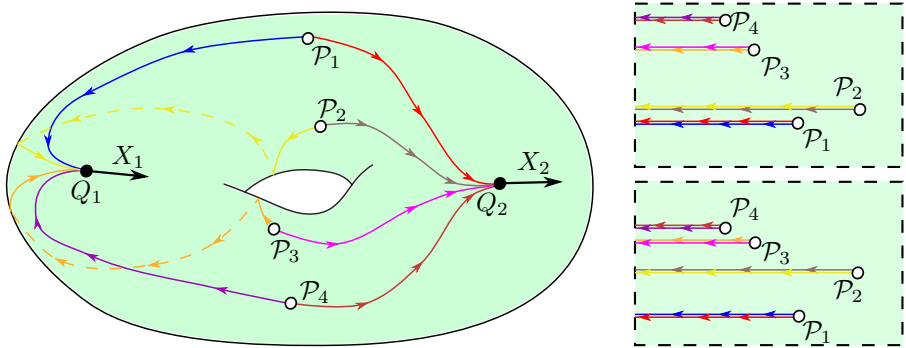
For a $\ud$-directed harmonic function $u\colon(\sgn,\fr)\setminus\uQ\to \R$, the complement of the critical graph $\sgn\setminus\cK_0(u)$ consists of $d$ contractible, open regions. Let
$L_{i,j}\colon[0,\epsilon)\to\sgn$, for $1\le i\le n$ and $0\le j\le d_i-1$, be a short, smooth path exiting from $Q_i$ with velocity 
$e^{2\pi j\sqrt{-1}/d_i}X_i$, where we use the Riemann structure $\fr$ to rotate $X_i$ by an angle $2\pi j/d_i$. Each arc $L_{i,j}$ is contained, for small positive times, in a single region of $\sgn\setminus\cK_0(u)$; viceversa, all $d$ regions of $\sgn\setminus\cK_0(u)$
contain the germ of precisely one arc $L_{i,j}$. See Figure \ref{fig:moduli_3}.
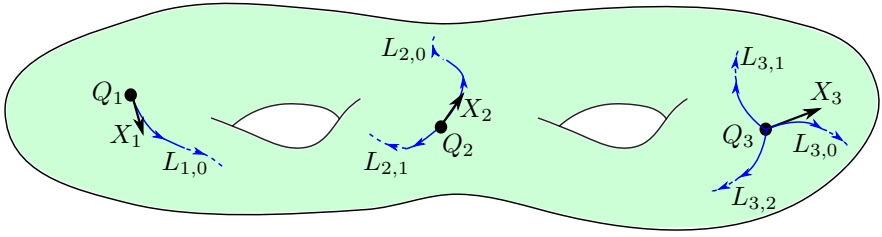
\begin{figure}[ht]
\centering
\begin{overpic}[scale=0.7]{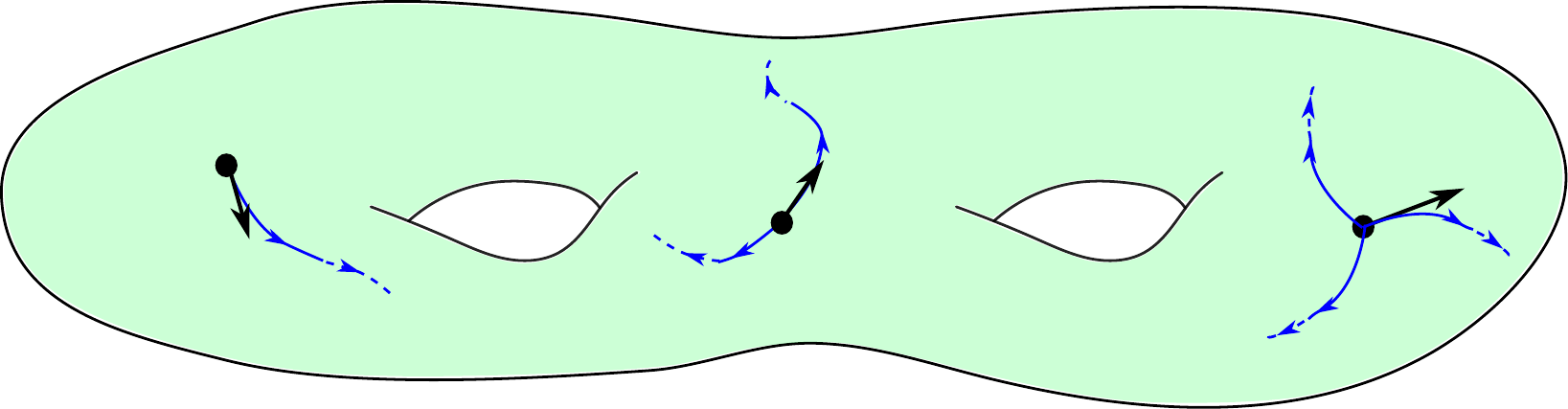}
\put(10,15){$Q_1$}
\put(12,10){$X_1$}
\put(18,7){$L_{1,0}$}
\put(50,9){$Q_2$}
\put(52,13){$X_2$}
\put(43,20){$L_{2,0}$}
\put(41,7){$L_{2,1}$}
\put(82,10){$Q_3$}
\put(92,15){$X_3$}
\put(90,9){$L_{3,0}$}
\put(84,19){$L_{3,1}$}
\put(83,3){$L_{3,2}$}
\end{overpic}
 \caption{An example of choices of the germs $L_{i,j}$ for a Riemann surface $(\Sigma_{2,3},\fr)$ and $\ud=(1,2,3)$.}
 \label{fig:moduli_3}
\end{figure}
\begin{defn}
 \label{defn:Harmudfr}
 Let $\fr$ be a Riemann structure on $\sgn$. The space $\Harm(\fr,\ud)$ contains couples $(u,v)$, where $u\colon(\sgn,\fr)\setminus\uQ\to\R$ is a $\ud$-directed harmonic function, and $v\colon\sgn\setminus\cK_0(u)$ is a conjugate harmonic function to $u$ defined on the complement of the critical graph of $u$, i.e. $u+\sqrt{-1}v$ is a holomorphic function on $\sgn\setminus\cK_0(u)$.
\end{defn}
We note that $\Harm(\fr,\ud)$ has the structure of a real affine space. If $(u,v)$ and $(u',v')$ are in $\Harm(\fr,\ud)$ and $t\in\R$, then $u'':=tu+(1-t)u'$ is also a $\ud$-directed harmonic function on $\sgn$; the complement
$\sgn\setminus\cK_0(u'')$ consists of $d$ connected components, each containing the germ of one of the paths $L_{i,j}$; the function $tv+(1-t)v'$ is well-defined on the image of $L_{i,j}$, at least for small positive times, and can be extended uniquely to a conjugate harmonic function to $u''$ on the component of $\sgn\setminus\cK_0(u'')$ containing $L_{i,j}$; these $d$ extensions give a harmonic function $v''\colon(\sgn,\fr)\setminus\cK_0(u'')$, and we may define the convex combination $t(u,v)+(1-t)(u',v')$ to be $(u'',v'')$.

To prove that $\Harm(\fr,\ud)$ is non-empty, and to compute its real dimension,
we first give a definition.
\begin{defn}
 \label{defn:uddf}
 Let $df$ be a meromorphic differential on $\sgn$. We say that $df$ is $\ud$-directed if $df$ is holomorphic away from $\uQ$, and if for all $1\le i\le n$ the Laurent expansion of $df$ in a normal chart $w_i$ around $Q_i$ has the form 
 \[
 df(w_i)=\frac{-d_i}{w_i^{d_i+1}}dw_i+\mathrm{h.o.t.}.
 \]
\end{defn}
Note the different exponents in comparison with Definition \ref{defn:udharm}.
We then argue as follows: see \cite{Boedigheimer90} for more details.
\begin{itemize}
\item Given $(u,v)\in\Harm(\fr,\ud)$, the complex differential $df:=du+\sqrt{-1}dv$ can be continuously extended to a $\ud$-directed meromorphic differential on $\sgn$;
moreover, for all loops $\gamma$ in $\sgn\setminus\uQ$ we have $\Re(\int_\gamma df)=0$: we call this the \emph{$\Re\!\int\!0$-condition}.
\item Conversely, given a $\ud$-directed meromorphic differential $df$ on $\sgn$
satisfying the $\Re\!\int\!0$-condition, we can choose
 an integral $u=\int \Re(df)\colon\sgn\setminus\uQ\to \R$; $u$ is a $\ud$-directed harmonic function, uniquely determined up to one real integration constant;
 on $\sgn\setminus\cK_0(u)$ we can choose an integral $v=\int\Im(df)$, which is uniquely determined up to $d$ real integration constants.
\item The space of \emph{all} $\ud$-directed meromorphic differentials on $(\sgn,\fr)$ is an affine translation of $\cO(\fr,K(\fr)+D)$ inside $\cO(\fr,K(\fr)+D+\uQ)$; since $\deg(K(\fr)+D)\ge 2g-1$, the complex dimension of $\cO(\fr,K(\fr)+D)$ is $g-1+d$.
\item The space of $\ud$-directed meromorphic differentials $df$ on $(\sgn,\fr)$ 
satisfying the $\Re\!\int\!0$-condition is non-empty by an application of the Dirichlet principle; it has real dimension $2g-2+2d-(2g+n-1)=2d-n-1$
because it is obtained from the affine space of all $\ud$-directed meromorphic differentials by imposing the vanishing of $2g+n-1$ real linear functionals.
\item Summing the real dimension of the space of $\ud$-directed meromorphic differentials $df$ satisfying the $\Re\!\int\!0$-condition with the number $d+1$ of integration constants, we obtain that the real dimension of $\Harm(\fr,\ud)$ is
$3d-n=3h-(6g+4n-6)$. This is equal to $3h-\dim\fM_{g,n}$ unless $g=0$ and $n=1$,
in which case we get $3h+2$.
\end{itemize}

The spaces $\Harm(\fr,\ud)$, for varying $\fr\in\mathfrak{Riem}(\sgn)$, assemble into a space $\tHarmud$; the forgetful map 
$\tilde\forg_{\Harm}\colon\tHarmud\to\mathfrak{Riem}(\sgn)$, sending $(\fr,u,v)$ to $\fr$, is a bundle map.
\begin{defn}
 \label{defn:Harmud}
 The group $\Diff_{g,n}$ acts freely on $\tHarmud$ by pulling back both Riemann structures and harmonic functions. The orbit space is denoted
 \[
 \Harmud:=\tHarmud/\Diff_{g,n};
 \]
 the forgetful map $\tilde\forg_{\Harm}$ gives rise to a bundle map $\forg_{\Harm}\colon\Harmud\to\fM_{g,n}$.
\end{defn}
The space $\Harmud$ is an orientable manifold of dimension $3h$, and the map $\forg_{\Harm}$ is a bundle map. If $g>0$ or $n>1$, the fibre $\forg_{\Harm}^{-1}(\fm)$ is canonically identified with $\Harm(\fr,\ud)$, for any $\fm=[\fr]\in\fM_{g,n}$. In the case $g=0$ and $n=1$, both $\forg_{\Harm}$ and $\tilde\forg_{\Harm}$ have only one fibre, $\tHarm(\ud)$ is itself a real affine space of dimension $3h+2$, and
$\Harmud$ is the quotient of $\tHarm(\ud)$ by a free, affine action of $\C$ by translations. In all cases $\forg_{\Harm}$ has contractible fibres and is therefore a homotopy equivalence.

Given an class $[\fr,u,v]\in\tHarmud$, we can consider the holomorphic function $f=u+\sqrt{-1}v\colon(\sgn\setminus\cK_0(u),\fr)\to\C$.
The map $f$ restricts to an open embedding on each of the $d$ connected components of $\sgn\setminus\cK_0(u)$, with image given by the complement of a finite configuration of horizontal, left-oriented halflines, called \emph{slits}.

The graph $\cK_0(u)$ has the following decorations:
\begin{itemize}
 \item it is a \emph{fat graph}, i.e. at each vertex there is a cyclic order of the incident edges; the associated surface, obtained by fattening edges to strips, has genus $g$ and $d$ boundary components (it is indeed a small neighbourhood of $\cK_0(u)$ in $\sgn$);
 \item all edges are oriented (by $-\nabla u$);
 \item $n$ of the vertices are labelled $Q_1,\dots, Q_n$; let $\crit_1,\dots,\crit_k$ be the other vertices, for some $k\ge0$;
 \item all edges incident to $Q_i$ are incoming; there is a partition of the
 edges incident to $Q_i$ into $d_i$ parts; each part has an internal order, and the set of parts is also totally ordered; the cyclic order at $Q_i$ is induced by the lexicographic order given by this partition (use the vectors $e^{2\pi j\sqrt{-1}/d_i}X_i$ to split the parts of this partition);
 \item no two edges incident to $\crit_i$ can be consecutive in the cyclic order and both incoming;
 \item the vertices $\crit_1,\dots,\crit_k$ are endowed with a preorder, i.e.
 an equivalence relation together with a total order on the set of equivalence classes; $\crit_i$ strictly preceeds $\crit_j$ in the preorder as soon as there is an edge oriented from $\crit_j$ to $\crit_i$ (evaluate the function $u$ to obtain the preorder).
\end{itemize}

If we fix a combinatorial type for $\cK_0$ as decorated graph, the subspace of $\Harmud$ of classes $[\fr,u,v]$ realising this combinatorial type is homeomorphic to an open multisimplex, whose coordinates correspond to the horizontal and vertical positions of the slits in the associated slit configurations. The one-point compactification of $\Harmud$ admits a cell decomposition with cells in bijection with combinatorial types for $\cK_0$ (together with the 0-cell given by the point at infinity).
The cellular structure of $\Harmud\cpt$ is described in detail in \cite{BH} and \cite{Bianchi:PhD}.

\section{The moduli space of branched covers of \texorpdfstring{$\C P^1$}{CP1}}
\label{sec:ccO}
We fix $g\geq 0$, $n\geq 1$ and a sequence $\ud=(d_1,\dots,d_n)$ as in Section \ref{sec:slit};
we moreover assume $d\ge2$, as in this section the case $d=1$ would be of little interest.
We still denote $d=\sum_{i=1}^nd_i$ and $h=2g+n+d-2$.

In this section we introduce $\ccOud$, the moduli space of Riemann surfaces $(\sgn,\fr)$ endowed with a $\ud$-directed meromorphic function. In fact $\ccOud$ can be considered as a closed subspace of $\Harmud$, and the restricted map
\[
\forg_{\ccO}=\forg_{\Harm}|_{\ccOud}\colon\ccOud\to\fM_{g,n}
\]
is a homotopy equivalence whenever $d\ge 2g+n-1$. It will however be convenient to embed $\ccOud$ directly
into a certain \emph{quotient} $\cHarmud$ of $\Harmud$: roughly speaking, if elements of $\Harmud$ carry the information of $d$ integration constants for the harmonic form $dv$, elements of $\cHarmud$ only carry the information of a single integration constant.

\subsection{Definition of \texorpdfstring{$\ccOud$}{Ogn[ud]} as a subspace of \texorpdfstring{$\cHarmud$}{cHgn[ud]}}
\begin{defn}
 \label{defn:cHarmr}
Recall Definition \ref{defn:udharm} and \ref{defn:Harmudfr}. For a Riemann structure $\fr$ on $\sgn$ we denote by $\cHarm(\fr,\ud)$ the space of triples $(u,dv,v_{1,0})$, where $u\colon(\sgn,\fr)\setminus\uQ\to\R$ is a $\ud$-directed harmonic function, $dv$ is the conjugate harmonic 1-form to $du$ on $\sgn\setminus\uQ$ (i.e. $dv$ is obtained from $du$ by applying the Hodge star operator), and $v_{1,0}$ is an integral of $dv$ on the component
of $\sgn\setminus\cK_0(u)$ containing the germ of a path $L_{1,0}$ exiting from $Q_1$ with velocity $X_1$, compare with Figure \ref{fig:moduli_3}.
\end{defn}
Note that $dv$ is completely determined by $u$ (or by $v_{1,0}$); the main reason to put $dv$ among the data is to have a ready notation for the harmonic conjugate of $du$.
The space $\cHarm(\fr,\ud)$ is an affine quotient of $\Harm(\fr,\ud)$, by considering the surjective map of affine spaces sending $(u,v)$ to $(u,dv,v_{1,0})$, where $v_{1,0}$ is the restriction of $v$ to the component
of $\sgn\setminus\cK_0(u)$ containing the germ of $L_{1,0}$. The real dimension of $\cHarm(\fr,\ud)$ is
$3d-n-(d-1)=2d-n+1$.

\begin{defn}
 \label{defn:cHarmud}
 For $g>0$ or $n>0$ we define $\cHarmud$ as the real affine bundle of dimension $2d-n+1$ over $\fM_{g,n}$
 whose fibre over $[\fr]=\fm\in\fM_{g,n}$ is $\cHarm(\fr,\ud)$. Formally, we first consider the affine bundle over $\Riem(\Sigma_{g,n})$ with fibre $\cHarm(\fr,\ud)$ over $\fr$, and then we take the quotient of the total space of this affine bundle by the group $\Diff_{g,n}$, which acts by pulling back simultaneously Riemann structures, harmonic functions and harmonic forms.
 We denote by $\check\forg_{\Harm}\colon\cHarmud\to\fM_{g,n}$ the associated bundle map.
\end{defn}
For $g=0$ and $n=0$ we will not need to consider a space $\cHarmud$. The following definition should be compared with Definitions \ref{defn:udharm} and \ref{defn:uddf}.
\begin{defn}
\label{defn:ccOr}
Let $\fr$ be a Riemann structure on $\sgn$.
A meromorphic function $f\colon(\sgn,\fr)\to \C P^1$ is \emph{$\ud$-directed} if it is holomorphic on $\sgn\setminus\uQ$ and if
for all $1\leq i\leq n$ and for any normal chart $w_i$ around $Q_i$, the Laurent expansion of $f$ has the form
\[
 f(w_i)=\frac{1}{w_i^{d_i}}+\mathrm{h.o.t.},
\]
where the higher order terms correspond powers of $w_i$ with exponent higher than $-d_i$:
we say that $f$ has a \emph{directed pole} of order $d_i$ at $Q_i$.

We let $\ccO(\fr,\ud)$ be the space of $\ud$-directed meromorphic functions on $(\sgn,\fr)$.
\end{defn}
Using normal charts, one can note that if $f,g\in\ccO(\fr,\ud)$ and $\lambda\in\C$, then
also $\lambda f+(1-\lambda)g\in\ccO(\fr,\ud)$; therefore $\ccO(\fr,\ud)$ is a \emph{complex affine} subspace
of $\cO(\fr,D)$ (see Definition \ref{defn:stddivisor}).

Given $f\in\ccO(\fr,\ud)$, we can define a harmonic
function $u=\Re(f)$ on $(\sgn,\fr)\setminus\uQ$, which admits a conjugate harmonic function $v=\Im(f)$ on the entire subspace $\sgn\setminus\uQ$, and in particular on the component of $\sgn\setminus\cK_0(u)$ containing the germ of a path $L_{1,0}$ as in Definition \ref{defn:cHarmr}: thus we determine a triple $(u,dv,v_{1,0})\in\cHarm(\fr,\ud)$.
We obtain an inclusion of real affine spaces $\ccO(\fr,\ud)\subset\cHarm(\fr,\ud)$;
the image of this inclusion consists of all triples $(u,dv,v_{1,0})\in\Harm(\fr,\ud)$ such that $v_{1,0}$ can be continuously extended over $\sgn\setminus\uQ$ to an integral of $dv$.
\begin{defn}
 \label{defn:ccOud}
We define $\ccOud$ as the moduli space of Riemann surfaces $(\sgn,\fr)$ endowed with a $\ud$-directed meromorphic function
$f\in\ccO(\fr,\ud)$. Two couples $(\fr,f)$ and $(\fr',f')$ are equivalent if there is a diffeomorphism $\psi\in\Diff_{g,n}$
pulling back $\fr'$ to $\fr$ and $f'$ to $f$.

Formally, we first consider the space $\tilde\ccO_{g,n}[\ud]$ of couples $(\fr,f)$,
with $\fr\in\Riem(\sgn)$ and $f\colon(\sgn,\fr)\to\C P^1$ a $\ud$-directed meromorphic function; the group
$\Diff_{g,n}$ acts freely on this space by pulling back Riemann structures and meromorphic functions, and
we define $\ccOud$ as the quotient $\tilde\ccO_{g,n}[\ud]/\Diff_{g,n}$.
We denote by $\forg_{\ccO}\colon\ccOud\to\fM_{g,n}$ the forgetful map $\forg_{\ccO}\colon [\fr,f]\mapsto[\fr]$.
\end{defn}
We used the same notation ``$\forg$'' in Definitions \ref{defn:Harmud}, \ref{defn:cHarmud} and \ref{defn:ccOud}: in fact, for $g>0$ or $n>1$, the space
$\ccOud$ can be embedded into $\cHarmud$ as the closed subspace of classes of quadruples $[\fr,u,dv,v_{1,0}]$ such that $v_{1,0}$ can be continuously extended to an integral $v\colon(\sgn,\fr)\setminus\uQ\to\R$ of $dv$.
Under this embedding, the map $\forg_{\ccO}$ from Definition \ref{defn:ccOud} is the restriction of the map
$\check\forg_{\Harm}$ from Definition \ref{defn:Harmud}. In the following we reformulate the characterisation of $\ccOud$ as subspace of $\cHarmud$.

\begin{defn}
 \label{defn:cFgn}
For $g>0$ or $n>1$ we denote by $\pr\colon\scF_{g,n}\to\fM_{g,n}$ the tautological \emph{closed} surface bundle over $\fM_{g,n}$, with fibre $\sgn$. Formally, $\scF_{g,n}=(\sgn\times\Riem(\sgn))/\Diff_{g,n}$. Removing marked points fibrewise yields a tautological \emph{open} surface bundle $\mpr\colon\mscF_{g,n}\to\fM_{g,n}$, with fibre $\sgn\setminus\uQ$.

We can then compute fibrewise the first real cohomology and obtain a vector bundle
$H^1(\mpr)\colon H^1(\mscF_{g,n})\to\fM_{g,n}$ with fibre $H^1(\sgn\setminus\uQ)\cong\R^{2g+n-1}$.
We define a map of real affine bundles over $\fM_{g,n}$
\[
\int\colon\cHarmud\to H^1(\mscF_{g,n})
\]
by sending the class $[\fr,u,dv,v_{1,0}]$ to the cohomology class on $\mpr^{-1}([\fr])\cong(\sgn,\fr)\setminus\uQ$ given by integrating the closed real 1-form $dv$ along 1-cycles.
\end{defn}
After Definition \ref{defn:cFgn}, for $g>0$ or $n>1$ we can then characterise $\ccOud\subset\cHarmud$ as the kernel of the map $\int$, i.e. the preimage of the zero section of $H^1(\mscF_{g,n})$. Comparing with the discussion after Definition \ref{defn:uddf}, we may say that for $f=(u,v)\in\ccO(\fr,\ud)$ the $\ud$-directed meromorphic form $df=du+\sqrt{-1}dv$ not only satisfies the ``$\Re\!\int\!0$-condition'', but also the analogous ``$\Im\!\int\!0$-condition''.

If $g> 0$ or $n> 1$, for all $\fm=[\fr]\in\fM_{g,n}$
there is a natural identification between $\forg^{-1}_{\ccO}(\fm)$ and $\ccO(\fr,\ud)$.
In the case $g=0$ and $n=1$, the space $\ccO(\frst,d)$
coincides with the space $\mathfrak{MonPol}_d$ of monic polynomials $z^d+a_{d-1}z^{d-1}+\dots+a_0$ of degree $d$ with complex coefficients. The space
$\ccO_{0,1}[d]$ is then the quotient of $\ccO(\frst,d)$ by the free action of $\C$ given by precomposing polynomials with translations of the complex plane: the element $\lambda\in \C$ acts on the polynomial $z^d+a_{d-1}z^{d-1}+\dots+a_0$ by sending it to the polynomial $(z+\lambda)^d+a_{d-1}(z+\lambda)^{d-1}+\dots+a_0$, which is again a monic polynomial of degree $d$.
\begin{defn}
 \label{defn:NMonPol}
 We denote by $\mathfrak{NMonPol}_d$ the space of \emph{normalised} monic polynomials of degree $d$, i.e. the quotient
 of $\mathfrak{MonPol}_d$ by the free action of $\C$ described above.
The projection $\mathfrak{MonPol}_d\to\mathfrak{NMonPol}_d$ is a surjective map,
admitting a section given by taking polynomials of the form $z^d+a_{d-2}z^{d-2}+\dots+a_0$, i.e. with vanishing coefficient of $z^{d-1}$. We will therefore regard $\mathfrak{NMonPol}_d$ also as an affine subspace of $\mathfrak{MonPol}_d$.
\end{defn}
In general the projection $\mathfrak{MonPol}_d\to\mathfrak{NMonPol}_d$ is not a map of affine spaces.
 
\subsection{Moduli spaces as classifying spaces}
\label{subsec:classifyingspaces}
Let $g>0$ or $n>1$. The tautological fibre bundle $\pr\colon\scF_{g,n}\to\fM_{g,n}$ is a fibre bundle with fibres Riemann surfaces of type $\Sigma_{g,n}$: that is, each fibre of $\pr$ is a Riemann surface of genus $g$ endowed with $n$ ordered and directed marked points.

In fact, $\pr$ is \emph{universal} among fibre bundles with this property, in the following strong sense. Let $\cX$ be a paracompact space and let $\pr_{\cY}\colon\cY\to\cX$ be a fibre bundle over $\cX$ with fibres smooth surfaces of genus $g$ endowed with the following additional structure:
\begin{itemize}
 \item $n$ disjoint, continuous sections $Q_1,\dots,Q_n\colon \cX\to\cY$ of $\pr_{\cY}$ are chosen, yielding $n$ marked points on each fibre;
 \item for each $1\le i\le n$, a nowhere-vanishing section $X_i\colon Q_i(\cX)\to VT(\cY)$ is chosen, where $VT(\cY)$ denotes the vertical tangent bundle of $\cY$, making the $n$ marked points of each fibre into directed marked points;
 \item each fibre of $\pr_{\cY}$ is endowed with a Riemann structure; Riemann structures change continuously
 in $\Riem_{g,n}$, i.e. for any fibre-wise smooth trivialisation $\pr_{\cY}^{-1}(U)\cong U\times\Sigma_{g,n}$ the natural map of sets $U\to\Riem_{g,n}$ is continuous.
\end{itemize}
Then there is a unique, continuous map $\kappa\colon\cX\to\fM_{g,n}$ such that $\cY$ is the pull back of $\scF_{g,n}$ along $\kappa$.

Consider now the projection $\check\forg\colon\cHarmud\to\fM_{g,n}$, and let $\check\pr_{\Harm}\colon\check\scF_{g,n}^{\Harm}\to\cHarmud$ be the pull back of the tautological surface bundle.
Similarly, let $\pr_{\ccO}\colon\scF^{\ccO}_{g,n}\to\ccOud$ be the restriction of $\check\pr_{\Harm}$
over $\ccOud\subset\cHarmud$.

Then $\check\pr_{\Harm}$ is universal among fibre bundles $\pr_{\cY}\colon\cY\to\cX$ over paracompact spaces with fibres Riemann surfaces of type $\Sigma_{g,n}$ and endowed with the following:
\begin{itemize}
 \item there is a continuous function $u\colon\cY\setminus(Q_1(\cX)\sqcup\dots\sqcup Q_n(\cX))\to\R$ restricting to a $\ud$-directed harmonic function on each fibre;
 \item there is a continuous function $v\colon \cY_{1,0}\to\R$, where $\cY_{1,0}\to\R$ is the union, over the fibres of $\pr_{\cY}$, of the connected components of the complements of the critical graphs of $u$ containing germs of paths exiting from the section $Q_1$ with velocity the vector field $X_1$; $v$ is moreover a harmonic conjugate to $u$ on each portion of fibre where it is defined.
\end{itemize}
Indeed the datum of such a fibre bundle is equivalent to a map $\kappa\colon\cX\to\fM_{g,n}$ together with a section of the pull back affine bundle $\kappa^*(\check\forg_{\Harm})$; in other words, the datum of such a fibre bundle is equivalent to a map $\check\kappa_{\Harm}\colon\cX\to\cHarmud$.

Finally, by restriction of the previous case, $\pr_{\ccO}$ is universal among fibre bundles over paracompact spaces with fibres Riemann surfaces $\Sigma_{g,n}$ endowed with a $\ud$-directed meromorphic function.

\subsection{Homotopy equivalence between \texorpdfstring{$\ccOud$}{Ogn[ud]} and \texorpdfstring{$\fM_{g,n}$}{Mgn}}
\label{subsec:ccOhtequivmgn}
In this subsection we prove the following theorem, which is the first main result of the article.
\begin{thm}
 \label{thm:main1}
Suppose $d\geq 2g+n-1$. Then the map $\forg\colon\ccOud\to\fM_{g,n}$ is a fibre bundle map with contractible fibres,
hence a homotopy equivalence; moreover the space $\ccOud$ is an orientable manifold of dimension $2h=2(2g+n+d-2)$.
\end{thm}
The condition $d\geq 2g+n-1$ is needed in the following lemma, in which we apply Theorem \ref{thm:RiemannRoch} to prove that $\ccO(\fr,\ud)$ is non-empty, for all Riemann structures $\fr$ on $\sgn$.
\begin{lem}
 \label{lem:ccOnonempty}
Suppose $d\geq 2g+n-1$. Then for every Riemann structure $\fr$ on $\sgn$,
$\ccO(\fr,\ud)$ is non-empty and is a complex affine space of complex dimension $d-g+1-n$.
\end{lem}
\begin{proof}
We consider the divisor $\cD:=D-\sum_{i=1}^nQ_i$, whose degree is $\geq 2g-1$: by Theorem \ref{thm:RiemannRoch} we have
\[
 \dim_{\C}\cO\pa{\fr,\cD}=d-g+1-n.
\]
We note that $\cO\pa{\fr,\cD}$ is a sub-vector space of $\cO(\fr,D)$, and that 
$\ccO(\fr,\ud)$ is either empty or a translate of $\cO\pa{\fr,\cD}$ in $\cO(\fr,D)$:
indeed the difference of two meromorphic functions in $\cO(\fr,D)$ lies in $\cO\pa{\fr,\cD}$.

To prove that $\ccO(\fr,\ud)$ is non-empty, consider for all $1\leq j\leq n$ the divisor $\cD+Q_j$, whose
degree is $\geq 2g$: again by Theorem \ref{thm:RiemannRoch} we have
\[
 \dim_{\C}\cO\pa{\fr,\cD+Q_j}=d-g+2-n,
\]
so that, by comparing dimensions, we can find a function
\[
f_j\in\cO\pa{\fr,\cD+Q_j}\setminus \cO\pa{\fr,\cD}.
\]
Note that $f_j$ has a pole of order exactly $d_j$ at $Q_j$;
up to multiplying by a suitable constant in $\C^*$ we may assume that the Laurent expansion of $f_j$ around $Q_j$, read in a normal
chart $w_j$, has the form $1/w_j^{d_j}+\mathrm{h.o.t.}$.
We can now consider the sum $f=\sum_{j=1}^nf_j$, which belongs to $\ccO(\fr,\ud)$ and witnesses that the latter is non-empty.
\end{proof}
Using Lemma \ref{lem:ccOnonempty}, we are ready to prove Theorem \ref{thm:main1}.
\begin{proof}[Proof of Theorem \ref{thm:main1}]
Assume first $g> 0$ or $n>1$.
Recall that $\check\forg_{\Harm}\colon\cHarmud\to\fM_{g,n}$ is a real affine bundle of dimension $2d-n+1$, whereas $H^1(\mpr)\colon H^1(\mscF_{g,n})\to\fM_{g,n}$ is a vector bundle of dimension $2g+n-1$. Lemma \ref{lem:ccOnonempty}
implies that $\int\colon\Harmud\to H^1(\mscF_{g,n})$ is surjective on fibres:
indeed the difference of dimensions $2d-n+1-(2g+n-1)=2d-2g+2-2n$ coincides with the dimension of the fibres of $\forg_{\ccO}=\ker\int$. Since $\ccOud$ is the kernel of a surjective map of a real affine bundle onto a real vector bundle, we conclude that
$\forg_{\ccO}\colon\ccOud\to\fM_{g,n}$ is also an affine bundle map, with fibres of real dimension $2(d-g+1-n)$.
Since fibres of $\forg_{\ccO}$ have a complex structure, they are equipped with a canonical orientation.

Since $\forg_{\ccO}\colon\ccOud\to\fM_{g,n}$ is a fibre bundle map with contractible
fibres, we obtain the first statement of Theorem \ref{thm:main1}. Moreover $\ccOud$ is the total space of a fibre bundle over an orientable manifold of real dimension $6g-6+4n$, with fibres being oriented manifolds of real dimension $2(d-g+n-1)$: it follows that $\ccOud$ is an orientable manifold of real dimension $6g-6+4n + 2(d-g+1-n)=2h$.

In the case $g=0$ and $n=1$, the map $\forg_{\ccO}\colon\ccOud\to\fM_{g,n}$ is also a homotopy equivalence, as
$\ccO_{0,1}[d]\cong\mathfrak{NMonPol}_d$ is contractible; the real dimension of
the unique fibre of $\forg_{\ccO}$ is $2(d-1)=2h$; note again that this is $2$ \emph{less} than the real dimension of
$\ccO(\frst,d)\cong\mathfrak{MonPol}_d$, which is $2d$.
\end{proof}
In the case $d<2g+n-1$ we fail in proving the surjectivity of $\forg_{\ccO}\colon\ccOud\to\fM_{g,n}$, but it is still true that fibres $\forg_{\ccO}^{-1}(\fm)$ are either empty or contractible (as they have a structure of affine complex space), so one could expect $\ccOud$ to capture the homotopy type of a subspace of $\fM_{g,n}$. For example, when $n=1$ and $d=2$, the image of 
$\forg_{\ccO}\colon\ccO_{g,1}[d]\to\fM_{g,1}$ is precisely the \emph{hyperelliptic moduli space}, i.e. the subspace of $\fM_{g,1}$ containing moduli $\fm=[\fr]$ such that $(\Sigma_{g,1},\fr)$ admits a hyperelliptic involution fixing its marked point $Q$ (and sending $X$ to $-X$). It would be interesting to generalise this observation to higher values of $d$; similar stratifications of the moduli space $\fM_g$ of \emph{closed} Riemann surfaces have been considered, among other, in \cite{Rauch, Farkas, Arbarello, ArbarelloII, Lax, Diaz, Looijenga95}.

\section{The Poincar\'e property for the PMQ \texorpdfstring{$\fS_d\geo$}{Sdgeo}}
\label{sec:Poincare}
Let $d\ge1$ be fixed throughout the section and let $\fS_d$ denote the group of permutations of the set $\set{1,\dots,d}$. In \cite{Bianchi:Hur1}
the word length norm on $\fS_d$ with respect to the generating set of all transpositions was used to define a \emph{partially multiplicative quandle} (shortly PMQ) called $\fS_d\geo$. The underlying set of $\fS_d\geo$ is $\fS_d$. Conjugation is defined as the conjugation of the group $\fS_d\geo$. The partial multiplication coincides, whenever defined, with the multiplication in $\fS_d$: a couple of permutations $(\sigma,\tau)$ admits a product in $\fS_d\geo$ if and only if it is \emph{geodesic}, i.e. the permutation $\sigma\tau\in\fS_d$ has norm equal to the sum of the norms of $\sigma$ and $\tau$. By construction, $\fS_d\geo$ is a \emph{normed} PMQ, i.e. it is endowed with a morphism of PMQs $N\colon\fS_d\geo\to\N$, sending 
$(\fS_d\geo)_+:=\fS_d\geo\setminus\set{\one}$ to positive natural numbers. We remark that $\fS_d\geo$
is augmented (a product of elements $\neq\one$ is $\neq \one$) and locally finite (every element can be written in finitely many ways as a finite product of elements $\neq\one$). In this section we show the following theorem.
\begin{thm}
  \label{thm:main2}
For $d\ge1$ the partially multiplicative quandle $\fS_d\geo$ is Poincar\'e, i.e.
all components of $\Hur^{\Delta}(\fS_d\geo)$ are topological manifolds.
\end{thm}
We immediately remark that for $d=1$ the PMQ $\fS_1\geo$ is isomorphic to the trivial PMQ $\set{\one}$, so that $\Hur^{\Delta}(\fS_d\geo)$ is a point, and in particular a manifold of dimension 0. We will henceforth focus on the case $d\ge2$.
\begin{nota}
\label{nota:cR}
 For $t\ge0$ we denote by $\cR_t$ the closed rectangle $[0,t]\times[0,1]\subset\C$, which for $t=0$ is a vertical segment; we denote by $\mcR_t$ the interior of $\cR_t$, which for $t=0$ is empty. For $t=1$ we abbreviate $\cR_1=\cR$ and $\mcR_1=\mcR$.
\end{nota}
The simplicial Hurwitz space $\Hur^{\Delta}(\fS_d\geo)$ was introduced in \cite{Bianchi:Hur1}; since $\fS_d\geo$ is a locally finite PMQ, by \cite[Theorem 9.1]{Bianchi:Hur2} the space $\Hur^{\Delta}(\fS_d\geo)$ is homeomorphic to the \emph{Hurwitz-Ran space} $\Hur(\mcR,(\fS_d\geo)_+)$: elements in the latter space are couples $(P,\psi)$ such that $P\subset\mcR$ is finite, and $\psi\colon\fQ(P)\to\fS_d\geo$ is an augmented map of PMQs, assigning in a consistent way an element in $(\fS_d\geo)_+$ to each based loop in $\pi_1(\C\setminus P)$ that spins clockwise around exactly one point of $P$. By \cite[Theorem 6.14]{Bianchi:Hur1} the total monodromy $\hat\totmon\colon\Hur^{\Delta}(\fS_d\geo)\to\widehat{\fS_d\geo}$ induces a bijection $\pi_0\pa{\Hur^{\Delta}(\fS_d\geo)}\to\widehat{\fS_d\geo}$.

The complete PMQ $\widehat{\fS_d\geo}$ was described in \cite[Proposition 7.13]{Bianchi:Hur1}: its elements are given by sequences $(\sigma;\fP_1,\dots,\fP_\ell;r_1,\dots,r_\ell)$, where $\sigma\in\fS_d$, the sets $\fP_1,\dots,\fP_\ell$ form an unordered partition of $\set{1,\dots,d}$, and $r_1,\dots,r_\ell$ are non-negative integers, each corresponding to one piece of the partition; moreover, certain mild combinatorial conditions must be satisfied. By \cite[Lemma 7.12]{Bianchi:Hur1}, each element in $\widehat{\fS_d\geo}$ can be also represented as a product $\hat\tr_1\dots\hat\tr_r$, where $\hat\tr_i\in\widehat{\fS_d\geo}$ is the image of the transposition $\tr_i\in\fS_d$ under the canonical inclusion of PMQs $\fS_d\geo\hookrightarrow\widehat{\fS_d\geo}$; the decomposition of an element in $\widehat{\fS_d\geo}$ into transpositions is unique up to standard moves \cite[Definition 3.6]{Bianchi:Hur1}.

In order to prove Theorem \ref{thm:main2} we thus have to show that for all $a\in\widehat{\fS_d\geo}$ the space
$\Hur^{\Delta}(\fS_d\geo)(a)$, or equivalently the homeomorphic space $\Hur(\mcR,(\fS_d\geo)_+)_a$, is a topological manifold. By \cite[Theorem 9.3]{Bianchi:Hur2} we can restrict to the case $a=\hat\sigma$, i.e. $a$ is the image of an element $\sigma\in\fS_d\geo$ along the inclusion $\fS_d\geo \hookrightarrow\widehat{\fS_d\geo}$; in the notation of \cite[Proposition 7.13]{Bianchi:Hur1} we have
$\hat\sigma=(\sigma;c_1,\dots,c_n;0,\dots,0)$, where $c_1,\dots,c_n$ are the cycles of a permutation $\sigma\in\fS_d$, considered as subsets of $\set{1,\dots,d}$. Using the action by global conjugation \cite[Subsection 6.2]{Bianchi:Hur2}, it suffices to consider one representative $\sigma$ in each conjugacy class of $\fS_d\geo$.

We will first address the case $\sigma=\lc_d\in\fS_d\geo$, where $\lc_d=(1,\dots,d)$ is the \emph{long cycle}, and then we will deduce the case of a generic permutation $\sigma\in\fS_d\geo$.

\subsection{The case \texorpdfstring{$\sigma=\lc_d$}{sigma=lcd}}
\label{subsec:monic}
Recall Definition \ref{defn:NMonPol}, and let $f(z)=z^d+a_{d-2}z^{d-2}+\dots+a_0$ be a normalised monic polynomial of degree $d$.
The derivative of $f$ is the polynomial $f'(z)=dz^{d-1}+(d-2)a_{d-2}z^{d-3}+\dots+a_1$; denote the $d-1$ roots of $f'$ by $\zeta_1,\dots,\zeta_{d-1}$. The images $f(\zeta_1),\dots,f(\zeta_{d-1})$ are the \emph{critical values} of $f$ when considered as a branched cover $f\colon\C\to\C$, and they form an element of the $(d-1)$-fold symmetric power of $\C$, denoted $\SP^{d-1}(\C)=\C^{d-1}/\fS_{d-1}$. We obtain a continuous and proper map $\cv\colon\NMonPol_d\to\SP^{d-1}(\C)$, with finite fibres.
\begin{defn}
 We denote by $\NMonPol_d(\cR)$ (respectively $\NMonPol_d(\mcR)$) the preimage under $\cv$ of $\SP^{d-1}(\cR)$ (respectively of $\SP^{d-1}(\mcR)$).
\end{defn}
Note that the space $\NMonPol_d(\cR)$ is compact.
\begin{prop}
 \label{prop:tildecv}
There exists a continuous bijection
\[
\check\cv^{\cR}\colon\NMonPol_d(\cR)\to\Hur(\cR;(\fS_d\geo)_+)_{\hat\lc_d},
\]
restricting to a bijection $\NMonPol_d(\mcR)\to\Hur(\mcR;(\fS_d\geo)_+)_{\hat\lc_d}$. 
\end{prop}
Let us assume for a moment that Proposition \ref{prop:tildecv} holds. Both $\NMonPol_d(\cR)$
and $\Hur(\cR;\fS_d\geo)_{\hat\lc_d}$ are compact Hausdorff spaces: in particular the second is homeomorphic
to the realisation of a finite semi-bisimplicial set $\Arr^{\mathrm{ndeg}}(\fS_d\geo)(\hat\lc_d)$.
It follows that $\check\cv^{\cR}$ is a homeomorphism, and by restriction $\Hur(\mcR;(\fS_d\geo)_+)_{\hat\lc_d}$
is homeomorphic to the subspace $\NMonPol_d(\mcR)$ of $\NMonPol_d$, which is a manifold.

The map $\check\cv^{\cR}$ will be a lift of the map $\cv\colon\NMonPol_d(\cR)\to\Hur(\cR;\N_+)_{d-1}\cong\SP^{d-1}(\cR)$
along the map $N_*\colon \Hur(\cR;(\fS_d\geo)_+)_{\hat\lc_d}\to\Hur(\cR;\N_+)_{d-1}$ induced by the augmented map of PMQs $N\colon\fS_d\geo\to \N$; this explains the notation ``$\check\cv$''. The superscript ``$\cR$'' distinguishes this map from a similar one considered in Section \ref{sec:homeo}.

\begin{defn}
 \label{defn:f0}
 We denote by $f_0\in\NMonPol_d(\cR)$ the polynomial $f_0(z)=z^d$. We define a deformation retraction $\rho\colon\NMonPol_d(\cR)\times[0,1]\to\NMonPol_d(\cR)$ onto $f_0$ by the formula $\rho(f,t)=f_0$ for $t=0$, and
 $\rho(f,t)=t^d f(z/t)$ for $t>1$. Note that the critical values of $t^d f(z/t)$ are obtained by multiplying the critical values of $f(z)$ by $t^d$, and hence they are contained in $\cR$ as well.
\end{defn}
Let $f(z)\in\NMonPol_d(\cR)$, and let $P=\set{f(\zeta_1),\dots,f(\zeta_{d-1})}$, considered as a finite subset of $\cR$ of cardinality at most $d-1$, where $\zeta_1,\dots,\zeta_{d-1}$ are the roots of $f'(z)$, as above. The branched cover $f\colon\C\to\C$ restricts to a genuine $d$-fold cover $f\colon f^{-1}(\CmP)\to\CmP$. In particular, the restriction $f\colon f^{-1}(\C\setminus\cR)\to\C\setminus\cR$ is a $d$-fold cyclic covering of the annulus
$\C\setminus\cR$. Varying $f$ in $\NMonPol_d(\cR)$, we obtain a family of $d$-fold cyclic coverings of $\C\setminus\cR$.

We fix once and for all a trivialisation $f_0^{-1}(\bH_{<0})\cong\set{1,\dots,d}\times \bH_{<0}$, where $\bH_{<0}=\set{z\,|\,\Im(z)<0}$; later, we will add a mild combinatorial assumption on this fixed trivialisation.
Through the homotopy $\rho$, we obtain for all $f\in\NMonPol_d(\cR)$ a trivialisation
\[
f^{-1}(\bH_{<0})\cong\set{1,\dots,d}\times \bH_{<0}.
\]
Taking all $f\in\NMonPol_d(\cR)$ at the same time, we obtain an embedding
\[
 \iota\colon \NMonPol_d(\cR)\times\set{1,\dots,d}\times\bH_{<0}\hookrightarrow \NMonPol_d(\cR)\times\C
\]
that is compatible with the projection onto $\NMonPol_d(\cR)$, and that satisfies the equality
$f(\pi_{\C}(\iota(f,i,z)))=z$ for all $f\in\NMonPol_d(\cR)$, $i\in\set{1,\dots,d}$ and $z\in\bH_{<0}$. Here $\pi_{\C}$ denotes the projection onto the factor $\C$.

Let $*=-\sqrt{-1}$ be the basepoint of $\CmP$ and of $\C\setminus\cR$. For $f\in\NMonPol_d(\cR)$, the $d$-fold cyclic covering $f\colon f^{-1}(\C\setminus\cR)\to\C\setminus\cR$ of the annulus $\C\setminus\cR$ is endowed with a trivialisation over $\bH_{<0}$, and in particular over $*$; we can thus compute the monodromy $\totmon(f)\in\fS_d$ as the action on $f^{-1}(*)\cong\set{1,\dots,d}$ induced by a based loop in $\C\setminus\cR$ spinning once clockwise around $\cR$. The monodromy $\totmon(f)\in\fS_d$ is a permutation consisting of a single cycle of length $d$, and is independent of $f\in\NMonPol_d(\cR)$; up to changing the above ``once and for all fixed'' trivialisation $f_0^{-1}(\bH_{<0})\cong\set{1,\dots,d}\times \bH_{<0}$, we can assume $\totmon(f)=\lc_d$ for all $f\in\NMonPol_d(\cR)$.

More generally, the covering $f\colon f^{-1}(\CmP)\to\CmP$ gives rise to an action
of the group $\pi_1(\CmP,*)$ on the set $f^{-1}(*)\cong\set{1,\dots,d}$. We thus obtain a map of groups $\phi\colon\pi_1(\CmP,*)\to\fS_d$, which is the $\fS_d$-valued monodromy associated with $f$. Restricting to the fundamental PMQ $\fQ(P)$, and interpreting elements of $\fS_d$ as the corresponding elements of the PMQ $\fS_d\geo$, we obtain a map of PMQs $\psi\colon\fQ(P)\to\fS_d\geo$, which is the $\fS_d\geo$-valued monodromy of $f$. We thus obtain a configuration $(P,\psi)\in\Hur(\cR;\fS_d\geo)$.
By construction, if $f\in\NMonPol_d(\mcR)$, then the set $P$ of critical values of $f$ is contained in $\mcR$, so that $(P,\psi)\in\Hur(\mcR;\fS_d\geo)$.

Our next goal is to show the following two statements:
\begin{itemize}
 \item the $\widehat{\fS_d\geo}$-valued total monodromy $\hat\totmon(P,\psi)$ is precisely $\hat\lc_d$, i.e. the image of $\lc_d\in\fS_d\geo$ under the inclusion $\fS_d\geo\hookrightarrow\widehat{\fS_d\geo}$;
 \item the map $\psi$ sends non-trivial loops of $\fQ(P)$ to elements of $(\fS_d\geo)_+$.
\end{itemize}
Let $|P|=k$ and write $P=\set{z_1,\dots,z_k}$. Let $\alpha_1,\dots,\alpha_k\colon[0,1]\to\CmP$ be loops based at $*$ representing an admissible generating set for $\pi_1(\CmP,*)$: the loops are embedded and disjoint away from the endpoints, and $\alpha_i$ spins once clockwise around $z_i$. Let $D_i\subset\C$ be the open disc bounded by $\alpha_i$. Let $\sigma_i=\phi(\alpha_i)\in\fS_d$, and let $c_{i,1},\dots,c_{i,\lambda_i}$ be the cycles in the cycle decomposition of $\sigma_i$, where $\lambda_i= d-N(\sigma_i)$ by \cite[Lemma 7.3]{Bianchi:Hur1}. Then $f^{-1}(D_i)$ is the union of $\lambda_i$ open discs in $\C$, which are in canonical bijection with the cycles $c_{i,1},\dots,c_{i,\lambda_i}$; moreover the disc corresponding to $c_{i,j}$ covers $D_i$ as a cyclic branched cover of degree $|c_{i,j}|$ with a unique branch point $\zeta_{i,j}$ over $z_i$. The point $\zeta_{i,j}\in\C$ is a zero of $f'$ of order $|c_{i,j}|-1=N(c_{i,j})$, and all zeroes of $f'$ arise in this way. It follows that $N(\sigma_i)=\sum_{j=1}^{\lambda_i}N(c_{i,j})$ is also equal to the number of zeroes of $f'$ lying over the critical value $z_i$ of $f$, counted with multiplicity. Since $z_i$ is assumed to be a critical value of $f$, we have $N(\sigma_i)\ge1$, i.e. $\sigma_i\in(\fS_d\geo)_+$. Moreover, summing over $1\le i\le k$, we obtain the equality $N(\sigma_1)+\dots+N(\sigma_k)=d-1$, which is the degree of the polynomial $f'$, but also the norm of $\lc_d$. Up to rearranging the indices from $1$ to $k$, we can assume that the product $\alpha_1\dots\alpha_k$ represents the same element of $\pi_1(\CmP,*)$ as an embedded loop spinning once clockwise around $\cR$. It follows that the product $\sigma_1\dots\sigma_k$ is defined in $\fS_d\geo$, and is equal to $\lc_d\in\fS_d\geo$ (a priori we only know that $\sigma_1\dots\sigma_k$ is an element of $\widehat{\fS_d\geo}$ mapping to the group element $\lc_d\in\fS_d$ under the natural map of PMQs $\widehat{\fS_d\geo}\to\fS_d$ having the group $\fS_d$ as target).
In fact a bit more is true, and will be used later: each partial product $\sigma_i\dots\sigma_{i'}$, for $i<i'$, is defined in $\fS_d\geo$.
\begin{defn}
 \label{defn:tildecv}
 The above assignment $f\mapsto(P,\psi)$ gives a map of sets
 \[
 \check\cv^{\cR}\colon\NMonPol_d(\cR)\to\Hur(\cR;(\fS_d\geo)_+)_{\hat\lc_d}.
 \]
\end{defn}

\begin{proof}[Proof of Proposition \ref{prop:tildecv}]
We first prove continuity of $\check\cv^{\cR}$. Let $f\in\NMonPol_d(\cR)$ and denote $\check\cv^{\cR}(f)=(P,\psi)\in \Hur(\cR;(\fS_d\geo)_+)_{\hat\lc_d}$; we write $P=\set{z_1,\dots,z_k}$ as above, and fix loops $\alpha_1,\dots,\alpha_k$ as above. Let $U_1,\dots,U_k$ be disjoint, convex open sets contained in $\C\setminus\pa{\bigcup_{i=1}^k\alpha_i}$, such that $U_i$ intersects $P$ in $z_i$; the sets $U_i$ give an \emph{adapted covering} $\uU=(U_1,\dots,U_k)$ of $P$, as in \cite[Definition 2.5]{Bianchi:Hur2}. Let $V\subseteq\NMonPol_d(\cR)$ be a connected neighbourhood of $f$ such that $\cv(V)\subseteq\SP^{d-1}(\uU\cap\cR)\subset\SP^{d-1}(\cR)$: such a neighbourhood exists because $\cv$ is continuous. Let $\bar f\in V$ and denote $(\bar P,\bar\psi)=\check\cv^{\cR}(\bar f)$.
Note that for all $1\le i\le k$ the map $\bar\psi\colon\fQ(\bar P)\to\fS_d\geo$ can be extended over $\alpha_i\in\pi_1(\C\setminus \bar P,*)$ in the sense of \cite[Definition 2.13]{Bianchi:Hur2}: we can factor $\alpha_i$ as a product of admissible generators $\bar\alpha_{i,1}\dots\bar\alpha_{i,k_i}$ and use the remark before Definition \ref{defn:tildecv}.
Then $\bar P\subset\uU$ by the condition $\cv(\bar f)\in \SP^{d-1}(\uU\cap\cR)$; moreover, for any path of polynomials $(f_s)_{0\le s\le 1}$ in $V$ joining $f$ with $\bar f$,
the $\fS_d\geo$-valued monodromy of $f_s$ along $\alpha_i$ is well-defined and locally constant in $s$, since the $\fS_d\geo$-valued monodromy is reflected by the $\fS_d$-valued monodromy, which is locally constant on the path. It follows that $\bar\psi(\alpha_i)=\psi(\alpha_i)$ as elements of $\fS_d\geo$. This means precisely that $(\bar P,\bar \psi)$ belongs to the normal neighbourhood $\fU(P,\psi;\uU)$ of $(P,\psi)$ in $\Hur(\cR;\fS_d\geo)_{\hat\lc_d}$, see Definition \cite[Definition 3.7]{Bianchi:Hur2}; we have thus proved that $\check\cv^{\cR}$ is continuous.

We next prove that $\check\cv^{\cR}$ is a bijection of sets. For this we will define a map of sets
$\bc^{\cR}\colon\Hur(\cR;\fS_d\geo)_{\lc_d}\to\NMonPol_d(\cR)$ which is both a left and right inverse to $\check\cv^{\cR}$. The notation ``$\bc$'' stands for ``branched cover''; the superscript ``$\cR$'' is to distinguish this map from a similar one used in Section \ref{sec:homeo}.

Let $(P,\psi)\in\Hur(\cR,(\fS_d\geo)_+)_{\hat\lc_d}$. The map of PMQs $\psi\colon\fQ(P)\to\fS_d\geo$ gives rise to a map of groups $\phi\colon\pi_1(\CmP,*)\to\fS_d$ using \cite[Theorem 3.3]{Bianchi:Hur1}; the action of $\fS_d$ on the set $\set{1,\dots,d}$ gives rise to a covering $f\colon \cF\to\CmP$ of degree $d$, with fibre $f^{-1}(*)$ canonically identified with the set $\set{1,\dots,d}$. The total space $\cF$ of the covering is connected, as the image of $\phi$ contains $\lc_d$ and thus acts transitively on $\set{1,\dots,d}$. The Euler characteristic of $\cF$ is computed as
\[
 \chi(\cF)=d\chi(\CmP)=d(1-k),
\]
using again the notation $P=\set{z_1,\dots,z_k}$, hence $|P|=k$. We can compactify $\cF$ to a smooth closed Riemann surface $\bar\cF$ endowed with a branched cover map $f\colon\bar\cF\to\C P^1$ by adjoining one point $Q$ lying over $\infty\in\C P^1$, and several other points, lying over the points of $P$. More precisely, fix loops $\alpha_1,\dots,\alpha_k$ as above, and use the same notation used above in the following: then for each $1\le i\le k$, the preimage $f^{-1}(D_i\setminus z_i)$ is a disjoint union of cyclic covers of the punctured disc $D_i\setminus z_i$, one of degree $|c_{i,j}|$ for every $1\le j\le \lambda_i=d- N(\sigma_i)$; we adjoin a point $\zeta_{i,j}$ to compactify each of these cyclic covers over $z_i$.
The Euler characteristic of the resulting surface is $\chi(\cF)$ plus the number of adjoined points, i.e.
\[
 \chi(\bar\cF)=d\chi(\CmP)+1+\sum_{i=1}^k (d-N(\sigma_i))=d-kd+1+kd-{d-1}=2.
\]
This implies that $\bar\cF$ is biholomorphic to the Riemann sphere; we note that $\infty\in \C P^1$ has a unique preimage $Q\in\bar\cF$. Moreover the meromorphic map
$f\colon\bar\cF\to\C P^1$ restricts to a trivial $d$-fold cover over $\bH_{<0}\subset\C P^1$, and we can uniquely extend the trivialisation $f^{-1}(*)\cong\set{1,\dots,d}$ to a trivialisation $f^{-1}(\bH_{<0})\cong\set{1,\dots,d}\times\bH_{<0}$.
We let then $X\in T_Q\bar\cF$ be the unique non-zero vector making $Q$ into a directed pole of $f$ of order $d$, and such that a germ of path $L$ exiting from $Q$ with velocity $X$ lies in $\set{1}\times\bH_{<0}\subset\C P^1$ for small positive times.

We obtain therefore a surface $(\bar\cF,Q,X)$ of type $\Sigma_{0,1}$, that can be identified with $\C P^1$; the map $f\colon\bar\cF\to\C P^1$ corresponds to the map $f\colon \C P^1\to\C P^1$ represented by a polynomial $f(z)$ of degree $d$. By construction, the polynomial $f(z)$ is monic, i.e. it belongs to $\MonPol_d$, and its projection
to $\NMonPol_d$ is independent of the identification $(\bar\cF,Q,X)\cong\C P^1$. Moreover the critical values of $f$ are all points of $P$ (all of them, because $\psi$ is an augmented map of PMQs, hence each $\sigma_i$ is not the identity permutation), and in particular lie in $\cR$.
The above assignment $(P,\psi)\mapsto f$ gives a map of sets
 \[
 \bc^{\cR}\colon\to\Hur(\cR;(\fS_d\geo)_+)_{\hat\lc_d}\to\NMonPol_d(\cR).
 \]
The fact that $\check\cv^{\cR}$ and $\bc^{\cR}$ are inverse maps of sets is straightforawrd.
\end{proof}

\subsection{The generic case \texorpdfstring{$\sigma\in\fS_d\geo$}{sigma in Sdgeo}}
\label{subsec:hNs}
Let $\sigma\in\fS_d\geo$ be generic, and let $\sigma=c_1\dots c_{\lambda}$ be the cycle decomposition of $\sigma$, with $\lambda=d-N(\sigma)$. Let $\fS_{c_i}$ be the symmetric group on the set $c_i\subset\set{1,\dots,d}$. There is an inclusion of PMQs (see \cite[Definition 2.16]{Bianchi:Hur1} for the product of PMQs)
\[
 \mu\colon \fS_{\underline{c}}\geo:=\fS_{c_1}\geo\times\dots\times\fS_{c_\lambda}\geo\hookrightarrow\fS_d\geo
\]
with image those permutations $\tau\in\fS_d\geo$ such that each cycle of $\tau$, considered as a subset of $\set{1,\dots,d}$, is contained in a single cycle $c_i$. By \cite[Corollary 7.5]{Bianchi:Hur1}, the image of $\mu$ is closed under factorisations in $\fS_d\geo$, or in other words, the complement
$\fS_d\geo\setminus\mu\pa{\fS_{\underline{c}}\geo}$ is an ideal
of $\fS_d\geo$ (see \cite[Definition 2.20]{Bianchi:Hur1} for the notion of ideal). Moreover, the image of $\mu$ contains $\sigma$.

It follows that $\Hur(\mcR;(\fS_d\geo)_+)_{\hat\sigma}$ is homeomorphic to $\Hur(\mcR;(\fS_{\underline{c}}\geo)_+)_{\hat\sigma}$.
Moreover $\Hur(\mcR;(\fS_{\underline{c}}\geo)_+)_{\hat\sigma}$ is homeomorphic to the product of spaces
\[
 \Hur(\mcR;(\fS_{\underline{c}}\geo)_+)_{\hat\sigma}\cong \prod_{i=1}^\lambda \Hur(\mcR;(\fS_{c_i}\geo)_+)_{\hat c_i}
\]
as can be directly seen using continuity of the external product \cite[Definition 5.7]{Bianchi:Hur2} and functoriality of Hurwitz-Ran spaces in the PMQ \cite[Subsection 4.1]{Bianchi:Hur2}.

Each space $\Hur(\mcR;(\fS_{c_i}\geo)_+)_{\hat c_i}$ is a manifold, since $c_i$ is a \emph{long cycle} (a permutation with a single cycle) in $\fS_{c_i}$; we conclude that the product of manifolds
$\Hur(\mcR;(\fS_d\geo)_+)_{\hat\sigma}$ is again a manifold.

\section{Homeomorphism between \texorpdfstring{$\ccOud$}{Ogn[d]} and simplicial Hurwitz spaces}
\label{sec:homeo}
In this section we prove Theorem \ref{thm:main3}, which is the third main result of the article and establishes a connection between moduli spaces $\ccOud$ and simplicial Hurwitz spaces $\Hur^{\Delta}(\fS_d\geo)$.
Some of the arguments of this section are similar to the ones from Section \ref{sec:Poincare}.
We fix $g\ge0$, $n\ge1$ and $\ud$ as in the previous sections; we further assume $d=\sum_{i=1}^nd_i\ge2$ throughout this section. We still denote $h=2g+n+d-2$.

\begin{defn}
 \label{defn:klud}
For $1\le i\le n$ we set $\bl_i=1+\sum_{j=1}^{i-1}d_j$; in particular $\bl_1=1<\bl_2<\dots<\bl_n$. We let $\klud\in\fS_d$ be the permutation with cycle decomposition
\[
 \klud=(\bl_1,\dots,\bl_2-1)(\bl_2,\dots,\bl_3-1),\dots,(\bl_n,\dots,d).
\]
Note that $\klud$ is a permutation on $n$ cycles of lengths $d_1,\dots,d_n$. We consider $\klud$ also as an element of $\fS_d\geo$.

We denote by $\widehat{\klud}$ the image of $\klud\in\fS_d\geo$ under the inclusion
$\fS_d\geo\hookrightarrow\widehat{\fS_d\geo}$, and for a transposition $\tr=(i,j)\in\fS_d$, similarly, we denote by $\hat\tr=\widehat{(i,j)}$ the corresponding element in 
$\widehat{\fS_d\geo}$. For $2\le i\le n$, we denote by $\tr^{\ud}_i$ the transposition $(\bl_i-1,\bl_i)$. We define $\klud_g\in\widehat{\fS_d\geo}$ as the following product, in which $\widehat{(1,2)}$ is repeated $2g$ times at the beginning, and each $\hat\tr^{\ud}_i$ is repeated twice at the end:
\[
 \klud_g=\widehat{(1,2)}\cdot\dots\cdot\widehat{(1,2)}\cdot\widehat{\klud}
 \cdot\hat\tr^{\ud}_2\cdot\hat\tr^{\ud}_2\cdot\hat\tr^{\ud}_3\cdot\hat\tr^{\ud}_3\cdot\dots\cdot\hat\tr^{\ud}_n\cdot\hat\tr^{\ud}_n.
\]
\end{defn}
Note that $N(\klud_g)=h$, where we consider the extension $N\colon\widehat{\fS_d\geo}\to \N$ of the norm $N\colon\fS_d\geo\to\N$.
Using the notation of \cite[Proposition 7.13]{Bianchi:Hur1}, we can write
\[
 \klud_g=(\klud;\set{1,\dots,d};h),
\]
where ``$\set{1,\dots,d}$'' denotes the trivial partition of the set $\set{1,\dots,d}$,
and ``$h$'' denotes a splitting of $h=N(\klud_g)$ into several summands, one for each
piece of the partition (in this case a single summand).

\begin{thm}
\label{thm:main3}
There is a homeomorphism
\[
 \bc\colon\Hur^{\Delta}(\fS_d\geo)(\klud_g)\overset{\cong}{\to}\ccOud.
\]
\end{thm}

In the case $g=0$ and $n=1$, Theorem \ref{thm:main3} reduces to the statement that $\Hur^{\Delta}(\fS_d\geo)(\hat\lc_d)$ is homeomorphic to $\NMonPol_d\cong\C^{d-1}$; this follows directly from Proposition \ref{prop:tildecv}, implying that 
$\Hur^{\Delta}(\fS_d\geo)(\hat\lc_d)$ is homeomorphic to the open subsset $\NMonPol_d(\mcR)$ of $\NMonPol_d$: since $\NMonPol_d(\mcR)$ is the interior of a contractible, semi-algebraic $2(d-1)$-dimensional submanifold with boundary
of $\NMonPol_d$, namely the interior of $\NMonPol_d(\cR)$, we conclude that $\NMonPol_d(\mcR)$ is homeomorphic to an open $2(d-1)$-dimensional ball, i.e. to $\C^{d-1}$. Therefore, from now on we will focus on the case $g>0$ or $n>1$.

To prove Theorem \ref{thm:main3} we will construct the map $\bc$ and another map
\[
\check\cv\colon\ccOud\to\Hur^{\Delta}(\fS_d\geo)(\klud_g),
\]
such that $\check\cv$ and $\bc$ are inverse bijections, and prove continuity of the two maps. Roughly, $\bc$ is the map given by constructing a \emph{branched cover}, whereas $\check\cv$ is the map given by taking \emph{critical values} (a finite configurations in $\C$) and retaining also the monodromy information.

It will in fact be convenient to replace $\Hur^{\Delta}(\fS_d\geo)(\klud_g)$ by the homeomorphic space $\Hur(\mcR,(\fS_d\geo)_+)_{\klud_g}$, using \cite[Theorem 9.1]{Bianchi:Hur2}, and to replace further
$\Hur(\mcR,(\fS_d\geo)_+)_{\klud_g}$ with a homeomorphic space $\Hur(\C,(\fS_d\geo)_+)_{\klud_g}$ that we will introduce in Subsection \ref{subsec:HurC}; we will then construct the maps $\bc$ and $\check\cv$ directly between $\ccOud$
and the new space $\Hur(\C,(\fS_d\geo)_+)_{\klud_g}$.

\subsection{Hurwitz-Ran spaces with \texorpdfstring{$\C$}{C} as background}
\label{subsec:HurC}
Let $\Q$ be a PMQ throughout the subsection.
For each subspace $\cX\subset\bH$ of the closed upper half plane,
a \emph{Hurwitz-Ran space} $\Hur(\cX,\Q)$ was introduced in \cite{Bianchi:Hur2}; a point in 
$\Hur(\cX,\Q)$ can be described as a couple $(P,\psi)$, where $P\subset\cX$ is a finite subset, and
$\psi\colon\fQ(P)\to\Q$ is a map of PMQs. The fundamental PMQ $\fQ(P)$ is defined as a subset of the fundamental group $\pi_1(\CmP,*)$, which in turn is defined after choosing a basepoint $*\in\CmP$. The choice $*=-\sqrt{-1}$, together with the hypothesis $\cX\subset\bH$, ensures that $*\not\in P$ for all $P$, and gives thus a ``constant'' choice of basepoint for the spaces $\CmP$, for $P$ varying among finite subsets of $\cX$.

\begin{defn}
For $\cK\subset\C$ with compact closure we define the point $*_{\cK}$ as
\[
 *_{\cK}:=\pa{\mathrm{inf}\set{\Im(z)\,|\, z\in \cK\cup\set{0}}-1}\cdot\sqrt{-1}\in\C\setminus\cK.
\]
\end{defn}
If $\cK\subset\bH$, then $*_{\cK}=-\sqrt{-1}$. We can now generalise the definition of the fundamental PMQ $\fQ(P)$ to all finite subsets $P\subset\C$; in the next definition we treat the more general case of a finite union of disjoint convex subsets of $\C$.
\begin{defn}
\label{defn:fQP}
Let $\cK$ be a finite union of convex subsets of $\C$ with compact and disjoint closures. We denote by $\fQ(\cK)\subset\pi_1(\C\setminus\cK,*_{\cK})$ the union of conjugacy classes corresponding to simple closed curves in $\C\setminus\cK$ spinning clockwise around at most one component of $\cK$, and by $\fQext(\cK)$ the union of all conjugacy classes corresponding to simple closed curves in $\C\setminus\cK$ spinning clockwise. Both $\fQ(\cK)$ and $\fQext(\cK)$ inherit a PMQ structure from $\pi_1(\C\setminus\cK,*_{\cK})$ as in \cite[Definition 2.8]{Bianchi:Hur1}.
\end{defn}
We note that $\fQ(\cK)\subseteq\fQext(\cK)$ are augmented PMQs
as can be easily checked by considering the abelianisation of $\pi_1(\C\setminus\cK,*_{\cK})$.
See the proof of \cite[Lemma 2.16]{Bianchi:Hur2} for an analogous argument.
\begin{defn}
 \label{defn:HurC}
An element of the space $\Hur(\C,\Q)$ has the form $(P,\psi)$, where $P\subset\C$ is finite and $\psi\colon\fQ(P)\to\Q$ is a map of PMQs.

Given $(P,\psi)\in\Hur(\C,\Q)$, write $P=\set{z_1,\dots,z_k}$ and let $\uU=(U_1,\dots,U_k)$
be a collection of disjoint, covex open sets with compact, disjoint closures, such that $z_i\in U_i$ (an \emph{adapted covering}). For any finite set $P'\subset\uU$ with $P'$ intersecting all components of $\uU$,
translating basepoints along straight segments in $\C$ gives a natural, injective group homomorphism
$\pi_1(\C\setminus\uU,*_{\uU})\hookrightarrow\pi_1(\CmP',*_{P'})$ restricting to an injection
$\fQ(\uU)\hookrightarrow\fQext(P')$; in the special case $P'=P$ we get in fact identifications $\pi_1(\C\setminus\uU,*_{\uU})\cong\pi_1(\CmP,*_{P})$ and $\fQ(\uU)\cong\fQ(P)$.

We define the \emph{normal neighbourhood} $\fU(P,\psi;\uU)\subset\Hur(\C,\Q)$ as the subset of configurations $(P',\psi')$ such that $P'\subset \uU$, $P'$ intersects all components of $\uU$, $\psi'$ can be extended to a sub-PMQ of $\fQext(P')$ containing $\fQ(\uU)$, and the restriction of $(\psi')^{\mathrm{ext}}$ on $\fQ(\uU)\cong\fQ(P)$ coincides with $\psi$.

For an element $a$ in the completion $\hat\Q$ of $\Q$, we denote by $\Hur(\C,\Q)_a$ the closed subspace of configurations $(P,\psi)$ such that the natural extension $\fQext(P)\to\hat\Q$ of $\psi$ sends the class of a simple loop $\gamma$ spinning clockwise around $P$ to the element $a$: we say that $a$ is the $\hat\Q$-valued total monodromy of $(P,\psi)$, and write $a=\hat\totmon(P,\psi)$.

If $\Q$ is augmented, we denote by $\Hur(\C;\Q_+)\subset\Hur(\C,\Q)$ the subspace of configurations $(P,\psi)$ such that $\psi\colon\fQ(P)\to\Q$ is an \emph{augmented} map of augmented PMQs,
i.e. the preimage of $\one\in\Q$ is $\set{\one}\subseteq\fQ(P)$.
\end{defn}
As in \cite[Sections 2 and 3]{Bianchi:Hur2}, the sets $\fU(P,\psi;\uU)$ define a Hausdorff topology on $\Hur(\C,\Q)$.
Our next goal is to identify the spaces
$\Hur(\mcR;\Q)$ and $\Hur(\C,\Q)$. Let $\xi\colon\C=\R^2\overset{\cong}{\to}\mcR=(0,1)^2$ be the homeomorphism given by
\[
\xi(x,y)=(e^x/(e^x+1),e^y/(e^y+1)).
\]
Given $(P,\psi)\in\Hur(\C,\Q)$, let $P'=\xi(P)\subset\mcR$. The fundamental groups $\pi_1(\CmP',*)$ and $\pi_1(\CmP',\xi(*_P))$ can be identified by translating $*=-\sqrt{-1}$ to $\xi(*_P)$ along a straight segment in $\CmP'$. Thus $\xi$ yields an identification of groups
\[
\pi_1(\CmP,*_P)\cong \pi_1(\mcR\setminus P',\xi(*_P))\cong  \pi_1(\CmP',\xi(*_P))\cong
\pi_1(\CmP',*)
\]
and this identification restricts to an identification $\fQ(P)\cong\fQ(P')$. We can then define $\psi'\colon \fQ(P')\to\Q$ as the map of PMQs corresponding to $\psi$. The assignment $\xi_*\colon(P,\psi)\mapsto(P',\psi')$ is a bijection between $\Hur(\C,\Q)$ and $\Hur(\mcR,\Q)$. Moreover $\xi_*$ induces a bijection between the bases of the two topologies given by normal neighbourhoods corresponding to coverings $\uU$ consisting of disjoint open rectangles compactly contained in $\C$ and in $\mcR$, respectively. Hence $\xi_*\colon \Hur(\C,\Q)\to\Hur(\mcR,\Q)$ is a homeomorphism.
Finally, $\xi_*$ respects the $\hat\Q$-valued total monodromy; and when $\Q$ is augmented, $\xi_*$ restricts to a homeomorphism $\Hur(\C,\Q_+)\to\Hur(\mcR,\Q_+)$.

In the rest of the section we identify the spaces $\Hur(\C,\fS_d\geo)$ and
$\Hur(\mcR,\fS_d\geo)$ and their corresponding pairs of subspaces without further mention.

\subsection{The map of sets \texorpdfstring{$\ccv$}{tildecv}}
\label{subsec:tildecv}
Denote again $h=2g+n+d-2$. The map of PMQs $N\colon\fS_d\geo\to \N$ given by the norm is augmented, and can be extended to an augmented map of complete PMQs $N\colon\widehat{\fS_d\geo}\to\N$. The element $\klud_g$ considered in the statement
of Theorem \ref{thm:main3} has norm precisely $h$, i.e. $N(\klud_g)=h$.

Similarly as in Subsection \ref{subsec:monic}, we first consider a map of sets
\[
\cv\colon\ccOud\to\SP^h(\C)\cong\Hur(\C,\N_+)_h,
\]
and later construct $\check\cv$ as a lift of $\cv$ along
the map $N_*\colon\Hur(\C,(\fS_d\geo)_+)_{\klud_g}\to \Hur(\C,\N_+)_h$. In this subsection we only define $\cv$ and $\check\cv$ as maps of sets; we will deal in Subsections \ref{subsec:ccvcontinuous} and \ref{subsec:bccontinuous} of the appendix with their continuity.

The map $\cv$ sends an element $[\fr,f]\in\ccOud$ to the collection of critical values of $f$ lying in $\C$, counted with multiplicity. Concretely, recall that $[\fr,f]\in\ccOud$ is the class of a pair $(\fr,f)$, where $\fr$ is a Riemann structure on $\Sigma_{g,n}$ and $f\colon (\sgn,\fr)\to\C P^1$ is an $\ud$-directed meromorphic function. Denote by $P=\set{z_1,\dots,z_k}$ the set of critical values of $f$ lying in $\C$, and for each $1\le i\le k$ let $\zeta_{i,1},\dots,\zeta_{i,\lambda_i}$ be the points in the fibre $f^{-1}(z_i)\subset\Sigma_{g,n}$; denote by $\uzeta$ the set $\set{\zeta_{i,j}\,|\, 1\le i\le k,1\le j\le \lambda_i}$. Since
$f\colon\Sigma_{g,n}\setminus\pa{\uQ\cup\uzeta}$ is a genuine $d$-fold cover of $\CmP$, we have
\[
 2-2g-n-\sum_{i=1}^k\lambda_i=\chi(\Sigma_{g,n}\setminus\pa{\uQ\cup\uzeta})=d\chi(\CmP)=d(1-k),
\]
implying the equality
\[
 \sum_{i=1}^k(d-\lambda_i)=2g+n+d-2=h.
\]
Thus the formal sum $\sum_{i=1}^k(d-\lambda_i)\cdot z_i$ represents an element in $\SP^h(\C)$. If we change representative of $[\fr,f]\in\ccOud$, we obtain the same element in $\SP^h(\C)$.
\begin{defn}
 We define a map of sets $\cv\colon \ccOud\to \SP^h(\C)$ by the assignment $[\fr,f]\mapsto \sum_{i=1}^k(d-\lambda_i)\cdot z_i$.
\end{defn}
We can say a bit more: let $\alpha\subset\CmP$ be a simple closed curve bounding a disc $D\subset\C$; then we can repeat the above argument and compute the Euler characteristic of $f^{-1}(D)$, which is a subsurface of $\Sigma_{g,n}$:
\[
 \chi(f^{-1}(D))=\chi(f^{-1}(D\setminus\uzeta))+\sum_{z_i\in D\cap P}\lambda_i=d-\sum_{z_i\in D\cap P}(d-\lambda_i).
\]
Next, for fixed $[\fr,f]\in\ccOud$ represented by $(\fr,f)$, and using the notation above, consider the closed segment $I_P=\set{*_P+t\,|\, t\in\R_{\ge0}}\cup\set{\infty}\subset\C P^1$, i.e. the closure of the horizontal line in $\C$ starting at $*_P$ and running towards right. Note that $I_P$ is disjoint from $P$. The preimage $f^{-1}(I_P)$ is a union of $d$ segments $I_{P,1},\dots,I_{P,d}\subset\sgn$, where for all $1\le i\le n$ and $0\le j\le d_i-1$ the segment $I_{\bl_i+j}$ has an endpoint at $Q_i$ and is tangent to $e^{2\pi j\sqrt{-1}/d_i}X_i$; compare with Figure \ref{fig:moduli_3}.

We obtain a trivialisation $f^{-1}(*_P)\cong\set{1,\dots,d}$ by declaring the endpoint of $I_i$ lying over $*_P$ to be the $i$\textsuperscript{th} point of $f^{-1}(*_P)$. In fact, using that $f$ restricts to a trivial $d$-fold covering over the lower half-plane $\bH_{\le\Im(*_P)}=\set{z\,|\,\Im(z)\le\Im(*_P)}\subset\C$, we also obtain a trivialisation
$f^{-1}(\bH_{\le\Im(*_P)})\cong \set{1,\dots,d}\times \bH_{\le\Im(*_P)}$.

We can now consider the map of groups $\phi\colon\pi_1(\CmP,*_P)\to\fS_d$ given by the monodromy of the $d$-fold cover $f^{-1}(\CmP)\overset{f}{\to}\CmP$, with trivialised fibre over $*_P$. Restricting $\phi$ over
$\fQ(P)$, and interpreting permutations as elements of $\fS_d\geo$, we obtain a map of PMQs
$\psi\colon\fQ(P)\to\fS_d\geo$: here we also use \cite[Theorem 3.3]{Bianchi:Hur1}, i.e. that $\fQ(P)$ is a free PMQ on $k$ generators. The couple $(P,\psi)$ is an element of $\Hur(\C,\fS_d\geo)$, and it is independent of the representative of $[\fr,f]\in\ccOud$.
\begin{defn}
\label{defn:tcv}
  We define a map of sets $\check\cv\colon \ccOud\to \Hur(\C,\fS_d\geo)$ by the above assignment $[\fr,f]\mapsto (P,\psi)$.
\end{defn}
We note that if $\alpha_i$ is a simple loop in $\CmP$ bounding a disc $D_i\subset\C$ with $D_i\cap P=\set{z_i}$, then $\psi(\alpha_i)$ is a permutation with precisely $\lambda_i$ cycles, one for each preimage $\zeta_{i,j}$ of $z_i$ along $f$. In particular, since $z_i$ is assumed to be a critical value of $f$, we have $\lambda_i<d$ and thus $N(\psi(\alpha_i))>0$. This implies $\psi(\alpha_i)\in(\fS_d\geo)_+$, and therefore
$\psi\colon\fQ(P)\to\fS_d\geo$ is an augmented map of PMQs. Moreover we have
\[
\chi(f^{-1}(D_i))=d-(d-\lambda_i)=d-N(\psi(\alpha_i)).
\]
Given now a simple loop $\alpha$ representing a class in $\fQext(P)$, we can also consider the natural extension $\hat\psi\colon\fQext(P)\to\widehat{\fS_d\geo}$ and evaluate $\hat\psi(\alpha)$: for instance, up to renumbering the points of $P$, we can find an admissible generating set $\alpha_1,\dots,\alpha_k$ of $\pi_1(\CmP,*_P)$ such that $\alpha$ decomposes as product $\alpha_i\dots\alpha_j$, for some $1\le i,j\le k$ with $i\le j+1$ (to allow an empty product), and such that the disc $D$ bounded by $\alpha$ contains the discs $D_i,\dots,D_j$ bounded by $\alpha_i,\dots,\alpha_j$; then we have
\[
 \hat\psi(\alpha)=\widehat{\psi(\alpha_i)}\dots\widehat{\psi(\alpha_j)}\in\widehat{\fS_d\geo},
\]
implying
\[
 N(\hat\psi(\alpha))=N(\psi(\alpha_i))+\dots+N(\psi(\alpha_j));
\]
on the other hand we have
\[
\begin{split}
 \chi(f^{-1}(D)) & =d\chi(f^{-1}(D\setminus\set{z_i,\dots,z_j})+\lambda_i+\dots+\lambda_j\\
 &=d-\pa{N(\psi(\alpha_i))+\dots+N(\psi(\alpha_j))}=d- N(\hat\psi(\alpha)).
\end{split}
\]
This holds in particular in the case in which $\alpha$ is a simple loop in $\CmP$ spinning clockwise around all points of $P$: in this case $\hat\psi(\alpha)=\hat\totmon(P,\psi)$, the $\widehat{\fS_d\geo}$-valued total monodromy of $(P,\psi)$, and we obtain the equality $2-2g-n=d-N(\hat\psi(\alpha))$, i.e. $N(\hat\totmon(P,\psi))=h$.
Moreover the permutations $\psi(\alpha_1),\dots,\psi(\alpha_k)$ must generate a subgroup of $\fS_d$ acting transitively on the set $\set{1,\dots,d}$, as the space $f^{-1}(\CmP)=\sgn\setminus\pa{\uQ\cup\uzeta}$ is connected. Finally, the identity of sets $\fS_d\geo\to\fS_d$ is a map of PMQs, where the second is a group and hence a complete PMQ;
the image of $\hat\totmon(P,\psi)$ under the natural extension $\widehat{\fS_d\geo}\to\fS_d$ is the permutation $\klud$, which is the $\fS_d$-valued total monodromy of the covering $f^{-1}(\CmP)\to\CmP$. Using \cite[Proposition 7.13]{Bianchi:Hur1}, we obtain the equality
$\hat\totmon(P,\psi)=\pa{\klud,   \set{1,\dots,d};h}=\klud_g$, yielding
the following proposition.
\begin{prop}
 The map of sets $\check\cv$ from Definition \ref{defn:tcv} has values in the subspace
 $\Hur(\C,(\fS_d\geo)_+)_{\klud_g}$.
\end{prop}

\subsection{The map of sets \texorpdfstring{$\bc$}{bc}}
We now construct a map of sets 
\[
\bc\colon\Hur(\C;(\fS_d\geo)_+)_{\klud_g}\to\ccOud.
\]
Let $(P,\psi)\in\Hur(\C;(\fS_d\geo)_+)_{\klud_g}$; the map $\psi\colon\fQ(P)\to\fS_d\geo$ gives rise to a map of groups $\phi\colon\pi_1(\CmP,*)\to\fS_d$ by \cite[Theorem 3.3]{Bianchi:Hur1}, and as in Subsection
\ref{subsec:monic} this gives in turn a $d$-fold covering $f\colon\cF\to\CmP$ with trivialised fibre $f^{-1}(*_P)\cong\set{1,\dots,d}$; we extend the trivialisation over $*_P$ to a trivialisation over the lower half-plane $\bH_{\le\Im(*_P)}$. Let $V\subset \bH_{>\Im(*_P)}=\set{z\,|\,\Im(z)>\Im(*_P)}\subset \C$ be a convex, open set containing $P$; then the restriction $f\colon f^{-1}(\C\setminus V)\to\C\setminus V$ is a disjoint union of $n$ cyclic covers of the annulus $\C\setminus V$:
more precisely, for all $1\le i\le n$, there is a component of $f^{-1}(\C\setminus V)$ containing the points labeled $\bl_i,\dots\bl_i+d_i-1$ of the trivialised fibre $f^{-1}(*_P)\cong\set{1,\dots,d}$.
We can compactify $f^{-1}(\C\setminus V)$ by adjoining $n$ points $Q_1^\cF,\dots,Q_n^\cF$: the point $Q_i^\cF$ compactifies the $i$\textsuperscript{th} component of $f^{-1}(\C\setminus V)$. We can also extend $f$ continuously by declaring $f(Q_i^\cF)=\infty\in\C P^1$. Similarly, we adjoin points $\zeta_{i,j}$ to compactify $\cF$ over the points $z_i\in P$, and extend $f$ continuously by declaring $f(\zeta_{i,j})=z_i$. The number $\lambda_i$ of points to adjoin over $z_i$ is equal to the number of cycles of the permutation $\phi(\alpha_i)$, where $\alpha_1,\dots,\alpha_k$ is an admissible generating set for $\pi_1(\CmP,*_P)$; in other words, we have $\lambda_i=d-N(\psi(\alpha_i))$. Since $\psi$ is an augmented map, we have moreover $\lambda_i<d$.

The result of the compactification is a smooth Riemann surface $\bar\cF$ endowed with a branched cover map $f\colon\bar\cF\to \C P^1$, and whose Euler characteristic is
\[
 \chi(\bar\cF)=d\chi(\CmP)+n+\sum_{i=1}^k\lambda_i=d+n-\sum_{i=1}^k N(\psi(\alpha_i))=d+n-h=2-2g;
\]
hence $\bar\cF$ is a closed Riemann surface of genus $g$. Moreover, for each $1\le i\le n$ we can uniquely define a vector $X_i^\cF\in T_{Q_i^\cF}\bar\cF$ satisfying the following properties:
\begin{itemize}
 \item $X_i^\cF$ is tangent to the unique segment in $\bar\cF$ projecting homeomorphically to $I_P$ and having as endpoints $Q_i\in\bar\cF$ and the point of $f^{-1}(*_P)$ labelled $\bl_i$;
 \item $f$ has a $X_i^\cF$-\emph{directed} pole of order $d_j$ at $Q_j^\cF$, i.e.
 for every holomorphic chart $w_i$ defined on a neighbourhood of $Q_i^\cF\in\bar\cF$ such that $w_i(Q_i^\cF)=0\in\C$ and $Dw_i(X_i^\cF)=\partial/\partial x\in T_0\C$,
 the function $f$ has the form
 \[
 f(w_i)=\frac{1}{w_i^{d_i}}+\mathrm{h.o.t.}.
 \]
\end{itemize}
Note that the first property determines $X_j^\cF$ up to real positive multiples, whereas the second determines it up to a rotation by an integral multiple of $2\pi/d_j$. 

We can now identify $(\bar\cF;Q_1^\cF,\dots,Q_n^\cF;X_1^\cF,\dots,X_n^\cF)$ with $\sgn=(\Sigma_g;\uQ;\uX)$ by a diffeomorphism; the Riemann structure on $\bar\cF$ gives a Riemann structure $\fr$ on $\Sigma_{g,n}$, and the function $f$ can be considered as a $\ud$-directed meromorphic function on $\sgn$. We thus obtain a class $[\fr,f]\in\ccOud$.
\begin{defn}
 \label{defn:bc}
  We define a map of sets $\bc\colon \Hur(\C,(\fS_d\geo)_+)_{\klud_g}\to \ccOud$
  by the above assignment $(P,\psi)\mapsto [\fr,f]$.
\end{defn}
The maps of sets $\check\cv$ and $\bc$ are inverse bijections, as is straightforward from the two constructions. It is left to prove that $\check\cv$ and $\bc$ are continuous: this is done in Subsections \ref{subsec:ccvcontinuous} and \ref{subsec:bccontinuous} of the appendix, respectively.

\begin{rem}
 The reader will notice that the statement of Theorem \ref{thm:main3} for $g=0$ and $n=1$ suffices to prove that $\Hur^{\Delta}(\fS_d\geo)(\hat\lc_d)$ is a manifold, since $\hat\lc_d=\klud_0$, and may wonder why one does not first prove Theorem \ref{thm:main3} and then present Theorem \ref{thm:main2}, skipping all of Subsection \ref{subsec:monic}.
 The reason is that we use the arguments of Subsection \ref{subsec:monic} in Subsection \ref{subsec:bccontinuous} of the appendix, which is part of the proof of Theorem \ref{thm:main3}.
 \end{rem}

\section{Stable rational cohomology of moduli spaces}
\label{sec:mumford}
\subsection{Surfaces with boundary}
The moduli space $\fM_{g,n}$ parametrises \emph{closed} Riemann surfaces of genus $g$ with $n$ directed marked points. It is convenient in this section to replace $\fM_{g,n}$ by the homotopy equivalent space $\fM_{g,n}^\del$.
\begin{defn}
 We denote by $\sgn^\del$ a compact, oriented, smooth surface of genus $g$ with $n$ ordered boundary components. Each boundary component $\del_i\sgn^\del$ is endowed with a fixed parametrisation by $S^1\subset\C$ which is compatible with the orientation induced on the boundary by the orientation of $\sgn^\del$. We also fix a germ of parametrisation of a collar neighbourhood of $\del_i\sgn^\del$ by a collar neighbourhood of $S^1\subset\mathbb{D}$, where $\mathbb{D}=\set{z\in \C\,\colon\, |z|\le 1}$ is the unit disc in $\C$; this germ of parametrisation of a collar neighbourhood is automatically orientation-preserving.
 
 We denote by $\Diff(\sgn^\del,\del)$ the topological group of diffeomorphisms of $\sgn^\del$ fixing pointwise some neighbourhood of $\del\sgn^\del$.
 
 We denote by $\Riem(\sgn^\del)$ the space of Riemann structures $\fr$ on $\sgn^\del$ for which the germ of parametrisation of a collar neighbourhood of each boundary curve of $\sgn^\del$ admits a holomorphic representative. The moduli space $\fM_{g,n}^\del$ contains equivalence classes $[\fr]$ of Riemann structures
 in $\Riem(\sgn^\del)$:
two Riemann structures $\fr$ and $\fr'$ are considered equivalent if there is a diffeomorphism $\phi\in\Diff(\sgn^\del,\del)$ that pulls back $\fr$ to $\fr'$. Formally, $\fM_{g,n}^\del$ is the quotient of the space $\Riem(\sgn^\del)$ by the action of the topological group $\Diff(\sgn^\del,\del)$.
\end{defn}
We can sew a copy $\mathbb{D}_i$ of the standard disc $\mathbb{D}=\set{z\in \C\,\colon\, |z|\le 1}$ along each boundary curve of $\sgn^\del$ by the map $S^1\to S^1$ given by $z\mapsto z^{-1}$; the result is a closed surface of genus $g$, endowed with $n$ marked points (the centres $Q_i$ of the discs $\mathbb{D}_i$); moreover, each marked point is endowed with a non-zero tangent vector $X_i:=\frac{d}{dx}\in T_{Q_i}\mathbb{D}_i\cong T_0\mathbb{D}$. We thus can obtain $\sgn$ from $\sgn^\del$ by sewing $n$ standard discs; by using germs of parametrised collar neighbourhoods of boundary components, one can also ensure to have a well-defined smooth structure on the sewing locus. Compare also with Definition \ref{defn:scM} in the next section.

If we have a Riemann structure $\fr\in\Riem(\sgn^\del)$, we can extend $\fr$ to a Riemann structure on $\sgn$ by adjoining the stadard Riemann structure on the discs $\mathbb{D}_i$: this gives a map $\Riem(\sgn^\del)\to\Riem(\sgn)$. Similarly, a diffeomorphism of $\sgn^\del$ fixing pointwise a neighbourhood of the boundary can be extended to a diffeomorphism of $\sgn$ by the identity of each disc $\mathbb{D}_i$: this gives a homomorphism of groups $\Diff(\sgn^\del,\del)\to\Diff_{g,n}$. Finally, the map $\Riem(\sgn^\del)\to\Riem(\sgn)$ is equivariant with respect to the two actions of 
$\Diff(\sgn^\del,\del)$ and $\Diff_{g,n}$ respectively, compared along the given group homomorphism: therefore we obtain a map of spaces $\fM_{g,n}^\del\to\fM_{g,n}$, which is known to be a weak homotopy equivalence.

The advantage of the space $\fM_{g,n}^\del$ is that it allows to define the genus stabilisation map.
 Let $\bar\fr_{1,2}$ be a fixed Riemann structure in $\Riem(\Sigma_{1,2}^\del)$ and let $1\le i\le n$.

 Given a Riemann structure $\fr\in\Riem(\sgn^\del)$, we can obtain a Riemann surface of genus $g+1$ with $n$ boundary curves as follows: for a fixed $1\le i\le n$,
 we sew $(\sgn^\del,\fr)$ and $(\Sigma_{1,2}^\del,\bar\fr_{1,2})$ by identifying $\del_i\sgn^\del$ with $\del_1\Sigma_{1,2}^\del$ along the map $S^1\to S^1$ given by $z\mapsto z^{-1}$.
 The new surface has $n-1$ boundary curves coming from $\sgn^\del$, and one boundary curve 
 coming from $\Sigma_{1,2}$, which is declared to be the $i$\textsuperscript{th} boundary curve of the new surface. We can identify the new surface with $\Sigma_{g+1,n}^\del$.
\begin{defn}
\label{defn:stabi}
 The above construction gives rise, for all $g\ge0$, $n\ge1$ and $1\le i\le n$, to a map
 \[
  \stab_i\colon\fM_{g,n}^\del\to\fM_{g+1,n}^\del.
 \]
\end{defn}
\subsection{A brief historical note}
A classical theorem of Harer \cite{Harer} (with later improvements in the stability ranges, see \cite{Ivanov, Boldsen, ORW:resolutions_homstab})
ensures that $\stab_i$ induces an isomorphism in integral homology $H_*(\fM_{g,n}^\del)\to H_*(\fM_{g+1,n}^\del)$ in the range of degrees $*\le\frac 23g-1$.

For simplicity, in
the rest of the section we shall restrict to surfaces with one boundary component, i.e. we set $n=1$.
It is interesting to compute the homology groups of the homotopy colimit
\[
 \fM_{\infty,1}\colon =\mathrm{hocolim}\pa{\fM_{0,1}^\del\overset{\stab_1}{\to}\fM_{1,1}^\del\overset{\stab_1}{\to}\fM_{2,1}^\del\overset{\stab_1}{\to}\dots},
\]
because these homology groups coincide, in a range depending on $g$, with the actual homology groups of $\fM_{g,1}$.
Mumford \cite{Mumford} conjectured that the cohomology ring $H^*(\fM_{\infty,1};\bQ)$ is isomorphic to the polynomial ring $\bQ[\kappa_1,\kappa_2,\dots]$ generated by variables $\kappa_i$ of degree $2i$, and he also gave a construction of candidates for such free multiplicative generators. Miller \cite{Miller} and Morita \cite{Morita87} proved independently the algebraic independence over $\bQ$ of the classes $\kappa_i\in H^*(\fM_{\infty,1};\bQ)$.

Subsequently, Tillmann \cite{Tillmann97} proved that the integral homology $H_*(\fM_{\infty,1})$ coincides with the homology of a component of an infinite loop space, and Madsen and Weiss \cite{MadsenWeiss} finally proved that $H_*(\fM_{\infty,1})$ is isomorphic to the homology $H_*(\Omega^\infty_0\MTSO(2))$ of a component of the infinite loop space associated with the spectrum $\MTSO(2)$. As a corollary of their theorem, Madsen and Weiss proved the Mumford conjecture.

\subsection{Double loop spaces of Hurwitz-Ran spaces}
In order to state the main result of this section, we need to recollect some definitions. For $d\ge2$ the \emph{relative} Hurwitz-Ran space
$\Hur_+(\cR,\del\cR;\fS_d\geo,\fS_d)_{\one}$, introduced in \cite{Bianchi:Hur2},
is the space of triples $(P,\psi,\phi)$ where
 \begin{itemize}
  \item $P$ is a \emph{non-empty}, finite subset of $\cR=[0,1]^2$ (whence the ``$_+$'');
  \item $\phi\colon \pi_1(\CmP,*_P)\to\fS_d$ is a morphism of groups (the $\fS_d$-valued monodromy);
  \item $\psi\colon \fQ_{(\cR,\del\cR)}(P)\to\fS_d\geo$ is a morphism of PMQs, only defined on $\fQ_{(\cR,\del\cR)}(P)\subset\pi_1(\CmP,*_P)$, the relative fundamental PMQ of $P$ with respect to the \emph{nice couple} $(\cR,\del\cR)$ (see \cite[Definition 2.9]{Bianchi:Hur2}).
 \end{itemize}
A compatibility condition between $\phi$ and $\psi$ is required:
the pair $(\psi,\phi)$ is required to be a morphism of PMQ-group pairs
\[
(\psi,\phi)\colon(\fQ_{(\cR,\del\cR)}(P),\pi_1(\CmP,*_P))\to(\fS_d\geo,\fS_d).
\]
Moreover the total monodromy $\totmon(P,\psi,\phi)$ is required to be equal to $\one\in\fS_d$; in the relative setting, the only well-defined total monodromy is the $\fS_d$-valued total monodromy, and there is no total monodromy with values in $\widehat{\fS_d\geo}$.

For $d\ge2$ we have an inclusion of PMQ-group pairs $(\fS_d\geo,\fS_d)\hookrightarrow(\fS_{d+1}\geo,\fS_{d+1}\geo)$ coming from the standard inclusion of groups $\fS_d\hookrightarrow\fS_{d+1}$; by functoriality of Hurwitz-Ran spaces in the PMQ-group pair (see \cite[Subsection 4.1]{Bianchi:Hur2}), we obtain an inclusion
\[
 \Hur_+(\cR,\del\cR;\fS_d\geo,\fS_d)_{\one}\hookrightarrow\Hur_+(\cR,\del\cR;\fS_{d+1}\geo,\fS_{d+1})_{\one}.
\]
The following is the main result of this section.
\begin{thm}
 \label{thm:main4}
 Let $\bB_\infty=\mathrm{hocolim}_{d\to\infty}\Hur_+(\cR,\del\cR;\fS_d\geo,\fS_d)_{\one}$.
 Then there is a zig-zag of integral homology equivalences of spaces starting from
 $\fM_{\infty,1}$ and ending with $\Omega^2_0 \bB_\infty$; in particular
 \[
  H_*(\fM_{\infty,1})\cong H_*\pa{\Omega^2_0 \bB_\infty}.
 \]
\end{thm}

\subsection{A new proof of the Mumford conjecture}
As a corollary of Theorem \ref{thm:main4}, we obtain a new proof of the Mumford conjecture.
\begin{cor}
 \label{cor:mumford}
The rational cohomology ring $H^*(\fM_{\infty,1};\bQ)$ is isomorphic to a polynomial ring $\bQ[y_1,y_2,\dots]$ generated by a sequence of variables $y_i$, one in each even degree $2i>0$.
\end{cor}
\begin{proof}
By Theorem \ref{thm:main4} it suffices to prove an isomorphism of rings $H^*(\Omega_0^2\bB_\infty;\bQ)\cong\bQ[y_1,y_2,\dots]$.
By Theorem \ref{thm:main2} the PMQ $\fS_d\geo$ is Poincar\'e for all $d\ge2$. We can thus apply \cite[Theorem 6.1]{Bianchi:Hur3} and obtain an isomorphism of graded rings
\[
 H^*\pa{\Hur_+(\cR,\del\cR;\fS_{d+1}\geo,\fS_{d+1})_{\one};\bQ}\cong \cA(\fS_d\geo),
\]
where we denote by $\cA(\fS_d\geo):=\bQ[\fS_d\geo]^{\fS_d}$ the sub-$\bQ$-algebra of conjugation-invariants in the PMQ-algebra $\bQ[\fS_d\geo]$. A basis of $\cA(\fS_d\geo)$ as a $\bQ$-vector space is given by the elements $\sca{S}=\sum_{\sigma\in S}\sca{\sigma}$, for $S$ varying among conjugacy classes of $\fS_d\geo$.
Here $\sca{\sigma}$ denotes the basis element of $\bQ[\fS_d\geo]$ corresponding to $\sigma\in\fS_d\geo$.

A conjugacy class in $\fS_d\geo$ is uniquely determined by a sequence
$\ulambda=(\lambda_i)_{i\ge2}$ of integers $\lambda_i\ge0$, such that $\sum_{i\ge2}i\lambda_i\le d$: we associate with $\ulambda$ the conjugacy class of permutations whose cycle decomposition contains precisely $\lambda_i$ cycles of length $i$, for all $i\ge2$, and a suitable number of fixpoints. We shall denote by $S_{\ulambda,d}\subseteq\fS_d\geo$ the conjugacy class corresponding to $\ulambda$, if $\sum_{i\ge2}i\lambda_i\le d$, and the empty set, if $\sum_{i\ge2}i\lambda_i> d$.
The element $\sca{S_{\ulambda,d}}\in\cA(\fS_d\geo)$ lies in cohomological degree $2N(\ulambda)$, where $N(\ulambda):=\sum_{i\ge2}(i-1)\lambda_i$.

The inclusion $\fS_d\geo\hookrightarrow\fS_{d+1}\geo$ is norm-preserving and induces an injective map betwee sets of conjugacy classes, sending $S_{\ulambda,d}\mapsto S_{\ulambda,d+1}$. Moreover, all conjugacy classes $S_{\ulambda,d+1}$ with $N(\ulambda)\le d/2$ are hit by a (non-empty) conjugacy class $S_{\ulambda,d}\subset\fS_d\geo$.

The argument of the proof of
\cite[Theorem 6.1]{Bianchi:Hur3} ensures that the map of cohomology rings induced by the inclusion
\[
 \Hur_+(\cR,\del\cR;\fS_d\geo,\fS_d)_{\one}\hookrightarrow\Hur_+(\cR,\del\cR;\fS_{d+1}\geo,\fS_{d+1})_{\one}.
\]
corresponds to the map of rings $\cA(\fS_{d+1}\geo)\to\cA(\fS_d\geo)$ sending $\sca{S_{\ulambda,d+1}}\mapsto\sca{S_{\ulambda,d}}$. In particular the map $\cA(\fS_{d+1}\geo)\to\cA(\fS_d\geo)$ is surjective,
and it is an isomorphism in cohomological degree $*\le d$.
We can thus compute
\[
 H^*\pa{\bB_\infty;\bQ}\cong\lim_{d\to\infty}\cA(\fS_d\geo),
\]
leveraging on the fact that the Mittag-Leffler condition is satisfied and no $\lim^1$ terms occur.

We now recall from \cite[Remark 6.3]{LehnSorger} that $\cA(\fS_d\geo)$ is isomorphic in cohomological degrees $*\le d$ to the polynomial algebra $\bQ[x_1,x_2,x_3,\dots,]$ generated by a sequence of variables $x_i$, one in each degree $2i\ge0$; see also \cite[Proposition 3.3]{BCP:Hilb}. We thus obtain an isomorphism of rings
\[
 H^*\pa{\bB_\infty;\bQ}\cong \bQ[x_1,x_2,x_3,\dots,].
\]
To conclude, by \cite[Theorem 4.19]{Bianchi:Hur3} the spaces $\Hur_+(\cR,\del\cR;\fS_d\geo,\fS_d)_{\one}$
are simply connected, and thus also $\bB_{\infty}$ is simply connected. We can then use a standard argument in rational homotopy theory: the map $\bB_\infty\to\prod_{i=1}^\infty K(2i,\bQ)$ classifying the cohomology classes $x_i$, is a rational cohomology equivalence between simply connected spaces, hence a rational equivalence. Its double looping is thus also a rational equivalence, and in particular a rational cohomology equivalence. Taking one component of the double loop spaces, we obtain
\[
\begin{split}
 H^*(\Omega_0^2\bB_\infty;\bQ) & \cong H^*\pa{\prod_{i=1}^\infty \Omega_0^2K(2i,\bQ);\bQ}\\
 &\cong H^*\pa{\prod_{i=1}^\infty K(2i,\bQ);\bQ}\cong
 \bQ[y_1,y_2,y_3,\dots,].
\end{split}
\]
\end{proof}
The rest of the section is devoted to the proof of Theorem \ref{thm:main4}.
\subsection{The main diagram}
For all $d\ge2$ we have a PMQ $\fS_d\geo$, giving rise to a Hurwitz-Moore topological monoid $\mHurm(\fS_d\geo)$, see \cite[Definition 2.4]{Bianchi:Hur3}. It will be convenient to consider in this section the submonoid
$\mHurm((\fS_d\geo)_+)$, containing configurations $(t,(P,\psi))$ with $\psi$ an \emph{augmented} map of PMQs.
The induced map on group completions
$\Omega B\mHurm((\fS_d\geo)_+)\hookrightarrow\Omega B\mHurm(\fS_d\geo)$ is a weak equivalence, so for our purposes the two monoids are interchangeable.

The inclusion of PMQs $\fS_d\geo\hookrightarrow\fS_{d+1}\geo$ gives rise to an inclusion of monoids $\mHurm((\fS_d\geo)_+)\hookrightarrow\mHurm((\fS_{d+1}\geo)_+)$.
\begin{nota}
Recall from \cite[Theorem 2.15]{Bianchi:Hur3} that
$\pi_0(\mHurm((\fS_d\geo)_+))$ is in bijection with $\widehat{\fS_d\geo}$.
For all transposition $\tr\in\widehat{\fS_d\geo}$ we fix once and for all a configuration in $\mHurm((\fS_d\geo)_+)_{\hat\tr}$ of ``width'' $1$, i.e. of the form $(1,\fc_\tr)$ for some $\fc_\tr\in\Hur(\mcR;(\fS_d\geo)_+)_{\hat\tr}$.

We denote by $\bl_\tr\colon\mHurm((\fS_d\geo)_+)\to\mHurm((\fS_d\geo)_+)$ the map $(t,\fc)\mapsto (1,\fc_\tr)\cdot(t,\fc)$.

We denote by $\br_\tr\colon\mHurm((\fS_d\geo)_+)\to\mHurm((\fS_d\geo)_+)$ the map $(t,\fc)\mapsto (t,\fc)\cdot (1,\fc_\tr)$. 
\end{nota}

Recall Definition \ref{defn:klud}, and note that for $n=1$ the element $\kld{d}_g:=\klud_g$ can be factored in $\widehat{\fS_d\geo}$ as
\[
 \kld{d}_g=\widehat{(1,2)}\dots \widehat{(1,2)}\ \cdot\ \hat\lc_d=\widehat{(1,2)}\dots \widehat{(1,2)}\ \cdot\ \widehat{(1,2)}\widehat{(2,3)}\dots\widehat{(d-1,d)},
\]
where $\widehat{(1,2)}$ is repeated $2g$ times at the beginning.

For all $d\ge2$ and $g\ge0$, we consider the following two maps:
\begin{itemize}
 \item the map $\bl_{(1,2)}^2\colon \mHurm((\fS_d\geo)_+)_{\kld{d}_g}\to\mHurm((\fS_d\geo)_+)_{\kld{d}_{g+1}}$,
 obtained by iterating twice the map $\bl_{(1,2)}$;
 \item the map $\mHurm((\fS_d\geo)_+)_{\kld{d}_g}\to\mHurm((\fS_{d+1}\geo)_+)_{\kld{d+1}_g}$ obtained by composing the inclusion
 $\mHurm((\fS_d\geo)_+)_{\kld{d}_g}\hookrightarrow\mHurm((\fS_{d+1}\geo)_+)_{\kld{d}_g}$ and the right multiplication map
 $\br_{(d,d+1)} \colon \mHurm((\fS_{d+1}\geo)_+)_{\kld{d}_g}\to \mHurm((\fS_{d+1}\geo)_+)_{\kld{d+1}_g}$.
\end{itemize}
We obtain a strictly commutative diagram of spaces, which we will refer to as the \emph{main diagram}
\[
 \begin{tikzcd}[column sep=15pt]
  \mHurm((\fS_2\geo)_+)_{\kld{2}_0}\ar[r,"\bl^2_{(1,2)}"]\ar[d,"\br_{(2,3)}"] &
  \pa{\mHurm((\fS_2\geo)_+)_{\kld{2}_1}}\ar[r,"\bl^2_{(1,2)}"]\ar[d,"\br_{(2,3)}"] &
  \mHurm((\fS_2\geo)_+)_{\kld{2}_2}\ar[r,"\bl^2_{(1,2)}"]\ar[d,"\br_{(2,3)}"] & \dots\\
  \mHurm((\fS_3\geo)_+)_{\kld{3}_0}\ar[r,"\bl^2_{(1,2)}"]\ar[d,"\br_{(3,4)}"] &
  \mHurm((\fS_3\geo)_+)_{\kld{3}_1}\ar[r,"\bl^2_{(1,2)}"]\ar[d,"\br_{(3,4)}"] &
  \mHurm((\fS_3\geo)_+)_{\kld{3}_2}\ar[r,"\bl^2_{(1,2)}"]\ar[d,"\br_{(3,4)}"] & \dots\\
  \mHurm((\fS_4\geo)_+)_{\kld{4}_0}\ar[r,"\bl^2_{(1,2)}"]\ar[d,"\br_{(4,5)}"] &
  \mHurm((\fS_4\geo)_+)_{\kld{4}_1}\ar[r,"\bl^2_{(1,2)}"]\ar[d,"\br_{(4,5)}"] &
  \pa{\mHurm((\fS_4\geo)_+)_{\kld{4}_2}}\ar[r,"\bl^2_{(1,2)}"]\ar[d,"\br_{(4,5)}"] & \dots\\
  \vdots &\vdots &\vdots &\ddots
 \end{tikzcd}
\]
\begin{defn}
 We denote by $\mHurm((\fS_\infty\geo)_+)_{\kld{\infty}_\infty}$ the homotopy colimit of the main diagram.
\end{defn}
We have highlighted some of the objects in the main diagram by putting them in parentheses: these are the spaces $\mHurm((\fS_{2g}\geo)_+)_{\kld{2g}_g}$, for varying $g\ge1$, and they form a diagonal of the diagram, which in particular is cofinal. We can thus compute $\mHurm((\fS_\infty\geo)_+)_{\kld{\infty}_\infty}$ also as
\[
 \mHurm((\fS_\infty\geo)_+)_{\kld{\infty}_\infty}\simeq \mathrm{hocolim}_{g\to \infty} \mHurm((\fS_{2g}\geo)_+)_{\kld{2g}_g},
\]
where now we have a sequential colimit, and the stabilisation maps are given by composing two vertical maps and one horizontal map in the previous diagram.

By \cite[Lemma 2.7]{Bianchi:Hur3}, each space $\Hur(\mcR;(\fS_{2g}\geo)_+)_{\kld{2g}_g}$ admits an inclusion into
$\mHurm((\fS_{2g}\geo)_+)_{\kld{2g}_g}$ which is a homotopy equivalence, with inverse an explicit deformation retraction, which in particular is natural in the PMQ. Using this together with Theorem \ref{thm:main3}, we obtain that each space $\mHurm((\fS_{2g}\geo)_+)_{\kld{2g}_g}$ is homotopy equivalent to the corresponding moduli space $\fM_{g,1}$. 

Moreover, for each $g\ge1$, the following diagram commutes up to homotopy:
{\small
\[
\begin{tikzcd}[column sep=7pt]
 \!\!\mHurm((\fS_{2g}\geo)_+)_{\kld{2g}_g}\! \ar[d,"\bl^2_{(1,2)}"] & \!\Hur(\mcR;(\fS_{2g}\geo)_+)_{\kld{2g}_g}\!\ar[l,hook,"\simeq"']
 \ar[r,"\bc","\cong"']  & 
 \ccO_{g,1}[2g] \ar[r,"\forg_{\ccO}","\simeq"'] & \fM_{g,1}\! &\fM_{g,1}^\del \ar[l,"\simeq"']\ar[d,"\stab_1"']\\
 \!\!\mHurm((\fS_{2g}\geo)_+)_{\kld{2g}_{g+1}}\! \ar[d,"\br_{(2g,2g+1)}"] & \!\Hur(\mcR;(\fS_{2g}\geo)_+)_{\kld{2g}_g}\!\ar[l,hook,"\simeq"'] \ar[r,"\bc","\cong"']  & 
 \ccO_{g+1,1}[2g] \ar[r,"\forg_{\ccO}"] & \fM_{g+1,1}\! &\fM_{g+1,1}^\del \ar[l,"\simeq"']\ar[d,equal]\\
 \!\!\mHurm((\fS_{2g+1}\geo)_+)_{\kld{2g+1}_{g+1}}\! \ar[d,"\br_{(2g+1,2g+2)}"]& \!\Hur(\mcR;(\fS_{2g}\geo)_+)_{\kld{2g}_g}\!\ar[l,hook,"\simeq"'] \ar[r,"\bc","\cong"']  & 
 \ccO_{g+1,1}[2g\!+\!1] \ar[r,"\forg_{\ccO}"] & \fM_{g+1,1}\! &\fM_{g+1,1}^\del \ar[l,"\simeq"']\ar[d,equal]\\
 \!\!\mHurm((\fS_{2g+2}\geo)_+)_{\kld{2g+2}_{g+1}}\! & \!\Hur(\mcR;(\fS_{2g}\geo)_+)_{\kld{2g}_g}\ar[l,hook,"\simeq"']\!\ar[r,"\bc","\cong"'] & 
 \ccO_{g+1,1}[2g\!+\!2] \ar[r,"\forg_{\ccO}","\simeq"'] & \fM_{g+1,1}\! &\fM_{g+1,1}^\del \ar[l,"\simeq"'].
\end{tikzcd}
\]
}
The top and bottom rows are zig-zags of weak equivalences: using that $\fM_{g,1}$ is a complex variety, and in particular is homeomorphic to a CW-complex, we could in fact invert up to homotopy the arrows
$\fM_{g,1}^\del\to\fM_{g,1}$. Similarly, we could invert up to homotopy the inclusions 
$\Hur_+(\mcR;\fS_{2g}\geo)_{\kld{2g}_g}\hookrightarrow \mHurm((\fS_{2g}\geo)_+)_{\kld{2g}_g}$ appearing on left, using the deformation retractions provided by \cite[Lemma 2.7]{Bianchi:Hur3}.

It follows that the homotopy colimit given by the left, vertical columns, for varying $g$, is homotopy equivalent to the homotopy colimit given by the right, vertical columns. In other words we have a weak equivalence
\[
\mathrm{hocolim}_{g\to \infty} \mHurm((\fS_{2g}\geo)_+)_{\kld{2g}_g} \simeq
\mathrm{hocolim}_{g\to \infty} \fM_{g,1}^\del.
\]
Putting together this equivalence and the previous one, we obtain the following lemma.
\begin{lem}
\label{lem:zigzag}
 There is a zig-zag of weak equivalences of topological spaces between $\mHurm((\fS_\infty\geo)_+)_{\kld{\infty}_\infty}$ and 
 $\fM_{\infty,1}$.
\end{lem}

Our final goal is to prove the existence of a homology equivalence
\[
 \mHurm((\fS_\infty\geo)_+)_{\kld{\infty}_\infty}\to \Omega^2_0\bB_\infty.
\]
\begin{nota}
For $d\ge2$ we denote by $\mHurm((\fS_d\geo)_+)_{\kld{d}_\infty}$ the homotopy colimit of the $d$\textsuperscript{th} row of the main diagram, i.e.
\[
\mHurm((\fS_d\geo)_+)_{\kld{d}_\infty}:=\mathrm{hocolim}_{g\to \infty}\mHurm((\fS_d\geo)_+)_{\kld{d}_g}.
\]
\end{nota}
We can then write $\mHurm((\fS_\infty\geo)_+)_{\kld{\infty}_\infty}$ as $\mathrm{hocolim}_{d\to\infty}\mHurm((\fS_d\geo)_+)_{\kld{d}_\infty}$.
Similarly,
$\Omega^2_0\bB_\infty$ is homotopy equivalent to the homotopy colimit
\[
\Omega^2_0\bB_\infty\simeq\mathrm{hocolim}_{d\to\infty}\Omega^2_0 \Hur_+(\cR,\del\cR;\fS_d\geo,\fS_d)_{\one},
\]
where we use the maps induced on double loop spaces by the inclusions of spaces
$\Hur_+(\cR,\del\cR;\fS_d\geo,\fS_d)_{\one}\hookrightarrow\Hur_+(\cR,\del\cR;\fS_{d+1}\geo,\fS_{d+1})_{\one}$,
which in turn are induced by the inclusions
of PMQ-group pairs $(\fS_d\geo,\fS_d)\hookrightarrow(\fS_{d+1}\geo,\fS_{d+1})$.

The proof of \cite[Theorem 4.1]{Bianchi:Hur3} is natural in the PMQ-group pair: hence we can replace each term 
$\Omega^2_0 \Hur_+(\cR,\del\cR;\fS_d\geo,\fS_d)_{\one}$ in the last homotopy colimit
by the homotopy equivalent term $\Omega_0B\mHurm((\fS_d\geo)_+)$, and thus obtain a homotopy
equivalence
\[
\Omega^2_0\bB_\infty\simeq \mathrm{hocolim}_{d\to\infty} \Omega^2_0 \Hur_+(\cR,\del\cR;\fS_d\geo,\fS_d)_{\one} \simeq \mathrm{hocolim}_{d\to\infty} \Omega_0B\mHurm((\fS_d\geo)_+)
\]

We would like therefore to prove the existence of a homology equivalence
\[
 \Xi\colon \mHurm((\fS_\infty\geo)_+)_{\kld{\infty}_\infty}\to \mathrm{hocolim}_{d\to\infty} \Omega_0B\mHurm(\fS_d\geo).
\]
In order to do so, we will prove the following proposition.
\begin{prop}
 \label{prop:Xid}
There exist a family of homology equivalences
\[
 \Xi_d\colon \mHurm((\fS_d\geo)_+)_{\kld{d}_\infty}\to \Omega_0 B\mHurm((\fS_d\geo)_+),\quad\quad d\ge2,
\]
such that for all $d\ge2$ the following square commutes up to homotopy
\[
 \begin{tikzcd}
  \mHurm((\fS_d\geo)_+)_{\kld{d}_\infty}\ar[r,"\Xi_d"] \ar[d] &\Omega_0 B\mHurm((\fS_d\geo)_+) \ar[d]\\
  \mHurm((\fS_{d+1}\geo)_+)_{\kld{d+1}_\infty}\ar[r,"\Xi_{d+1}"]& \Omega_0 B\mHurm((\fS_{d+1}\geo)_+).\\
 \end{tikzcd}
\] 
\end{prop}

\subsection{Mapping telescopes}
\label{subsec:maptel}
Let $M$ be a unital topological monoid, and let $\bar m\in M$ be an element. We can consider the mapping telescope
\[
 \Tel(M;\bar m):=\mathrm{hocolim}\pa{M\overset{\bar m\cdot-}{\to} M\overset{\bar m\cdot-}{\to}M \overset{\bar m\cdot-}{\to}\dots},
\]
obtained by iterating left multiplication by $\bar m$;
the argument of the proof of
the group-completion theorem \cite{McDuffSegal,FM94} provides a map $\Xi_{M;\bar m}$
from $\Tel(M;\bar m)$ to $\Omega BM$. First, one defines a map $\bs\colon M\to \Omega BM$,
sending $m\in M$ to the loop in $BM$ represented by the 1-simplex $m$; then one defines
$\Xi_{M;\bar m}$ as the map induced on homotopy colimits of the rows of the following diagram,
where squares can be filled by homotopies in a way that is natural in $M$ and $\bar m$:
\[
\begin{tikzcd}[column sep=40pt]
M\ar[r,"\bar m\cdot-"]\ar[d,"\bs(-)"] & 
M\ar[r,"\bar m\cdot-"]\ar[d,"{\bs(\bar m)^{-1}\cdot\bs(-)}"] & M \ar[r,"\bar m\cdot-"]\ar[d,"{\bs(\bar m^2)^{-1}\cdot\bs(-)}"] & \dots\\
\Omega BM\ar[r,equal] &\Omega BM\ar[r,equal] &\Omega BM\ar[r,equal] &\dots.
\end{tikzcd}
\]
Here, for $m\in M$, the loop $\bs(m)^{-1}\in\Omega BM$ is the inverse of the loop $\bs(m)$, and a representative $-\cdot-\colon\Omega BM\times\Omega BM\to\Omega BM$ is fixed.

We remark that, in order to prove that the map $\Xi_{M;\bar m}$ is a homology
equivalence, one requires additional conditions on $M$ and $\bar m$; but the map $\Xi_{M;\bar m}$ itself can be defined in general.

If we fix an element $a\in M$, representing a path component $\pi_0(a)\in\pi_0(M)$,
we can also restrict the previous to a map
\[
 \Tel(M,a;\bar m):=\mathrm{hocolim}\pa{M_{\pi_0(a)}\overset{\bar m\cdot-}{\to} M_{\pi_0(\bar m a)}\overset{\bar m\cdot-}{\to}M_{\pi_0(\bar m^2a)}\overset{\bar m\cdot-}{\to}\!\dots\!}\!\to\Omega_{\pi_0(a)}BM.
\]
We can then postcompose the previous with the map $\Omega_{\pi_0(a)}BM\overset{\simeq}{\to}\Omega_0BM$
given by right multiplication by $\bs(a)^{-1}$.
We obtain a map
\[
 \Xi_{M,a;\bar m}\colon \Tel(M,a;\bar m)\to \Omega_0BM.
\]
The construction above is natural with respect to triples $(M,a,\bar m)$ consisting of a topological monoid with two chosen elements. Moreover, if $b\in M$ is another element of $M$,
we can consider the map $\br^{\Tel}_b\colon  \Tel(M,a;\bar m)\to  \Tel(M,ab;\bar m)$ induced levelwise by right multiplication by $b$: then the following diagram commutes up to a homotopy
which is again natural in the quadruple $(M,a,\bar m,b)$:
\[
 \begin{tikzcd}
   \Tel(M,a;\bar m)\ar[r,"\br^{\Tel}_b"] \ar[dr,"\Xi_{M,a;\bar m}"'] &  \Tel(M,ab;\bar m)\ar[d,"\Xi_{M,ab;\bar m}"]\\
   & \Omega_0BM.
 \end{tikzcd}
\]
In homology, we have a canonical identification of $H_*(\Tel(M,a;\bar m))$ with the colimit
\[
\colim\pa{H_*(M_{\pi_0(a)})\overset{\bar m\cdot-}{\to} H_*(M_{\pi_0(\bar ma)})\overset{\bar m\cdot-}{\to}H_*(M_{\pi_0(\bar m^2a)})\overset{\bar m\cdot-}{\to}\dots},
\]
where we consider $\bar m\cdot -$ as the left multiplication by the ``ground class''
$\bar m$ in $H_0(M_{\pi_0(\bar m)})$, using the Pontryagin ring structure on $H_*(M)$.

The map $H_*\pa{\Tel(M,a;\bar m)}\to H_*(\Omega_0BM)$ induced by $\Xi_{M,a;\bar m}$ coincides with the map induced on colimits of rows by the following diagram:
\[
\begin{tikzcd}[column sep =45pt]
H_*(M_{\pi_0(a)})\ar[r,"\bar m\cdot-"]\ar[d,"{\bs(a)^{-1}\cdot\bs(-)}"] & H_*(M_{\pi_0(\bar ma) })\ar[r,"\bar m\cdot-"] \ar[d,"{\bs(\bar m a)^{-1}\cdot\bs(-)}"]  &
H_*(M_{\pi_0(\bar m^2a)})\ar[r,"\bar m\cdot-"] \ar[d,"{\bs(\bar m^2a)^{-1}\cdot\bs(-)}"]  &\dots \\
H_*(\Omega_0BM)\ar[r,equal] &H_*(\Omega_0BM)\ar[r,equal]  &H_*(\Omega_0BM)\ar[r,equal] &\dots
\end{tikzcd}
\]

We can now define $\Xi_d$ to be the map $\Xi_{M,a;\bar m}$ corresponding
to the data:
\begin{itemize}
 \item $M=\mHurm((\fS_d\geo)_+)$;\vspace{.1cm}
 \item $a=\pa{1,\fc_{(1,2)}}\pa{1,\fc_{(2,3)}}\dots\pa{1,\fc_{(d-1,d)}}$;\vspace{.1cm}
 \item $\bar m=\pa{1,\fc_{(1,2)}}^2$.
\end{itemize}
Note that $\mHurm((\fS_d\geo)_+)_{\kld{d}_\infty}$ is defined as the same homotopy colimit as $\Tel_{M,a;\bar m}$, for $(M,a,\bar m)$ as above.

The diagram in the statement of Proposition \ref{prop:Xid} commutes by considering the inclusion of monoids $\mHurm((\fS_d\geo)_+)\hookrightarrow\mHurm((\fS_{d+1}\geo)_+)$ together with the map $\br^{\Tel}_b$
induced by the element $b=(1,\fc_{(d,d+1)})$.

In order to prove Proposition \ref{prop:Xid}, and hence Theorem \ref{thm:main4}, it is only left to prove that $\Xi_d$ is a homology equivalence.

\subsection{Propagators}
\label{subsec:propagators}
In this subsection, we fix $d\ge2$ and aim at proving that $\Xi_d$ is a homology equivalence. We abbreviate $\mHurm((\fS_d\geo)_+)$ by $\mHurm_+$ for simplicity (note that this notation is not compatible with a similar notation
used in \cite{Bianchi:Hur3}, for example in \cite[Theorem 2.19]{Bianchi:Hur3}).
We denote $a=\pa{1,\fc_{(1,2)}}\pa{1,\fc_{(2,3)}}\dots\pa{1,\fc_{(d-1,d)}}$ as at the end of the previous subsection.
\begin{nota}
\label{nota:propagators}
 We denote by $e$ the element $\pa{1,\fc_{(1,2)}}\dots \pa{1,\fc_{(1,2)}}\in\mHurm_+$, where $\pa{1,\fc_{(1,2)}}$ is repeated $2d-2$ times; we consider $\pi_0(e)\in\pi_0(\mHurm_+)$,
 and denote by $\hat\totmon(e)$ the $\widehat{\fS_d\geo}$-valued total monodromy of $e$. We denote by $\pi_0(e)$ also the corresponding ``ground'' class
 in $H_0\pa{\mHurm_{+,\hat\totmon(e)}}\cong\Z$. 
 Similarly, we denote by $e'\in\mHurm_+$ the element
 \[
 e'=\pa{1,\fc_{(1,2)}}\pa{1,\fc_{(1,2)}}\pa{1,\fc_{(2,3)}}\pa{1,\fc_{(2,3)}}\dots \pa{1,\fc_{(d-1,d)}}\pa{1,\fc_{(d-1,d)}},
 \]
 and consider $\pi_0(e')$ also as an element in $H_0\pa{\mHurm_{+,\hat\totmon(e')}}\cong\Z$.
\end{nota}
Notice that $\hat\totmon(e)=(\one;\set{1,2},\set{3},\dots,\set{d};2d-2,0,\dots,0)$, using the notation of \cite[Proposition 7.13]{Bianchi:Hur1}, whereas $\hat\totmon(e')=(\one;\set{1,\dots,d};2d-2)$: these are different elements in $\widehat{\fS_d\geo}$.
Recall from \cite[Lemma 7.2]{Bianchi:Hur1} that the enveloping group $\cG(\fS_d\geo)$ of $\fS_d\geo$ is the index two subgroup $\tilde\fS_d\subset\Z\times\fS_d$ containing pairs $(r,\sigma)$ with $r$ and $\sigma$ of the same parity. We remark that $\hat\totmon(e)$ and $\hat\totmon(e')$
are sent to the same element $(2d-2,\one)\in\tilde\fS_d\cong \cG(\fS_d)$ under the natural map
$\widehat{\fS_d\geo}\to\cG(\fS_d\geo)$. We thus have
$\totmon(e)=\totmon(e')=(2d-2,\one)$, i.e. $\tilde\fS_d$-valued total monodromies of $e$ and $e'$ are equal.

There is yet another, more relevant difference between $e$ and $e'$: the element $e'$ is a \emph{propagator} of the topological monoid
$\mHurm_+$, whereas $e$ is not, in the following sense.
\begin{defn}
Let $M$ be a weakly braided topological monoid (see \cite[Definition 3.5]{Bianchi:Hur3}) and let $m\in M$. We say that $m$ is a \emph{propagator} if for every $x\in M$ there exist $y\in M$ and $k\ge1$ such that $x\cdot y$ and $m^k$ belong to the same path component of $M$. 
\end{defn}
The fact that $e'$ is a propagator follows from the fact that $\pi_0(\mHurm_+)\cong\widehat{\fS_d\geo}$
is generated by the elements $\widehat\tr$ corresponding to transpositions $\tr\in\fS_d\geo$; for any fixed transposition $\tr$, the element $\pi_0(e')\in\pi_0(\mHurm_+)$ admits a factorisation $\hat\tr_1\dots\hat\tr_{2d-2}$ in $\widehat{\fS_d\geo}$, with $\tr_1=\tr$, see \cite[Proposition 7.11]{Bianchi:Hur1}.
The group completion theorem ensures that the map
\[
  \Xi_{\mHurm_+,a;e'}\colon \Tel(\mHurm_+,a;e') \to \Omega_0B\mHurm_+
\]
is a homology equivalence. We then note the following fact.
\begin{lem}
\label{lem:braiding}
 Let $g\ge0$ and consider the two maps
 \[
 e\cdot-\,,\,e'\cdot-\colon \mHurm_{+,\kld{d}_g}\to\mHurm_{+,\kld{d}_{g+d-1}};
 \]
 given by left multiplication by $e$ and $e'$ respectively. Then $e\cdot-$ and $e'\cdot-$ are homotopic maps.
\end{lem}
\begin{proof}
 We can write $e\cdot-$ as the composition $\bl_{(1,2)}^{2d-2}$, whereas $e'\cdot-$ can be written as composition $\bl_{(1,2)}^2\bl_{(2,3)}^2\dots\bl_{(d-1,d)}^2$. It suffices therefore to prove, for all $g\ge0$, that for all $1\le j\le d-2$ the maps
 $\bl_{(j,j+1)}^2$ and $\bl_{(j+1,j+2)}^2$ are homotopic maps $\mHurm_{+,\kld{d}_g}\to\mHurm_{+,\kld{d}_{g+1}}$.
 
 For this, recall from the proof of \cite[Lemma 3.6]{Bianchi:Hur3} that the map 
\[
 \braiding\colon\mHurm\times\mHurm\to\mHurm\times\mHurm,\quad\quad \pa{(t,\fc),(t',\fc')}\mapsto\pa{(t',\fc'),(t,\fc^{\totmon(\fc')})}
 \]
 is a weak braiding, i.e. the composite map
 \[
 \begin{tikzcd}
  \mHurm\times \mHurm\ar[r,"\braiding"] & \mHurm\times \mHurm\ar[r,"-\cdot-"] &\mHurm
 \end{tikzcd}
 \]
is homotopic to the multiplication map $-\cdot-\colon \mHurm\times\mHurm\to \mHurm$. It follows that also the composite
\[
 \begin{tikzcd}
  \mHurm\times \mHurm\ar[r,"\braiding"] &\mHurm\times \mHurm\ar[r,"\braiding"] & \mHurm\times \mHurm\ar[r,"-\cdot-"] &\mHurm,
 \end{tikzcd} 
\]
sending $((t,\fc),(t',\fc'))\mapsto \pa{t,\fc^{\totmon(\fc')}}\cdot \pa{t',(\fc')^{\totmon(\fc)^{\totmon(\fc')}}}$
is homotopic to  $-\cdot-$. The same properties hold for the restricted map $\braiding\colon\mHurm_+\times\mHurm_+\to\mHurm_+\times\mHurm_+$.

We can now set $(t,\fc)=\pa{1,\fc_{(j+1,j+2)}}\cdot\pa{1,\fc_{(j+1,j+2)}}$; using that the element
\[
 \totmon\big((1,\fc_{(j+1,j+2)})(1,\fc_{(j+1,j+2)})\big)=(2,\one)\in\tilde\fS_d=\cG(\fS_d\geo)
\]
conjugates trivially all elements in $\fS_d\geo$,
and using instead that
\[
(j+1,j+2)^{(d-1+2g,\lc_d)}=(j+1,j+2)^{\lc_d}=(j,j+1)\in\fS_d\geo,
\]
we obtain precisely that
the map $\bl^2_{(j,j+1)}=\pa{1,\fc_{(j,j+1)}}\cdot\pa{1,\fc_{(j,j+1)}}\cdot-$ is homotopic to the map
$\bl^2_{(j+1,j+2)}=\pa{1,\fc_{(j+1,j+2)}}\cdot\pa{1,\fc_{(j+1,j+2)}}\cdot-$, when considered as maps
$\mHurm_{+,\kld{d}_g}\to\mHurm_{+,\kld{d}_{g+2}}$.
\end{proof}

We now notice that both $\Xi_{\mHurm_+,a;e'}\colon H_*(\Tel(\mHurm_+,a;e')) \to H_*(\Omega_0B\mHurm_+)$ and $\Xi_{\mHurm_+,a;e}\colon H_*(\Tel(\mHurm_+,a;e)) \to H_*(\Omega_0B\mHurm_+)$ can be described, by Lemma \ref{lem:braiding}, as the map induced on colimits of rows by the following commutative diagram of homology groups, where we use $\hat\totmon(a)=\kld{d}_0$:
\[
\begin{tikzcd}[column sep =30pt, row sep=40pt]
H_*(\mHurm_{+,\kld{d}_0})\ar[r,"e\cdot-=e'\cdot-"]\ar[d,"{\bs(-)\cdot\bs(a)^{-1}}"] & H_*(\mHurm_{+,\kld{d}_{d-1}})\ar[r,"e\cdot-=e'\cdot-"]
\ar[d,"{\bs(-)\cdot\bs(e\cdot a)^{-1}}=", near start]
\ar[d,"{\bs(-)\cdot\bs(e'\cdot a)^{-1}}", near end] &
H_*(\mHurm_{+,\kld{d}_{2d-2}})\ar[r,"e\cdot-=e'\cdot-"]
\ar[d,"{\bs(-)\cdot\bs(e\cdot e\cdot a)^{-1}}=", near start]
\ar[d,"{\bs(-)\cdot\bs(e'\cdot e'\cdot a)^{-1}}", near end]
&\dots \\
H_*(\Omega_0B\mHurm_+)\ar[r,equal] &H_*(\Omega_0B\mHurm_+)\ar[r,equal]  &H_*(\Omega_0B\mHurm_+)\ar[r,equal] &\dots
\end{tikzcd}
\]
Since $\Xi_{\mHurm_+,a;e'}$ induces a homology isomorphism, also
$\Xi_{\mHurm_+,a;e}$ induces a homology isomorphism.
To conclude, let $\bar m= \pa{1,\fc_{(1,2)}}^2$ as in the previous subsection: we note that there is a weak equivalence $\Tel(\mHurm_+,a;e)\overset{\simeq}{\to}\Tel\pa{\mHurm_+,a;\bar m}$,
given by the fact that $e=\bar m^{d-1}$ and hence the first telescope is defined as a homotopy colimit over a cofinal subsequence of the sequence of spaces yielding the second telescope as homotopy colimit. Moreover the following triangle commutes up to homotopy:
\[
 \begin{tikzcd}
  \Tel(\mHurm_+,a;e)\ar[r,"\simeq"] \ar[dr,"{\Xi_{\mHurm_+,a;e}}"']& \Tel\pa{\mHurm_+,a;\bar m}\ar[d,"{\Xi_{\mHurm_+,a;\bar m}}"]\\
  & \Omega_0B\mHurm_+.
 \end{tikzcd}
\]
Since the horizontal and the diagonal arrows are homology equivalences, also the vertical arrow is a homology equivalence.
This concludes the proof of Proposition \ref{prop:Xid}, and hence of Theorem \ref{thm:main4}.

\section{A Hurwitz-Ran model for \texorpdfstring{$\MTSO(2)$}{MTSO(2)}}
\label{sec:MTSO}
In this section we compare Theorem \ref{thm:main4}, asserting that 
$H_*(\fM_{\infty,1})\cong H_*\pa{\Omega^2_0 \bB_\infty}$, with the isomorphism
$H_*(\fM_{\infty,1})\cong H_*\pa{\Omega^\infty_0\MTSO(2)}$ proved
by Madsen and Weiss \cite{MadsenWeiss}. The following is the main result of the section.
\begin{thm}
 \label{thm:main5}
 Let $\bB_\infty$ be the space occurring in the statement of Theorem \ref{thm:main4}.
 There is a weak homotopy equivalence of spaces
 \[
 \bB_\infty\simeq\Omega^{\infty-2}\MTSO(2).\]
\end{thm}
We note that both spaces are simply connected:
in particular, by \cite[Theorem 4.19]{Bianchi:Hur3}, $\bB_\infty$ is a sequential homotopy colimit of simply connected spaces.
Therefore it suffices to construct a homology equivalence between the two spaces.
\subsection{\texorpdfstring{$E_2$}{E2}-algebras from Hurwitz-Ran spaces}
For $d\ge2$, the space $\Hur(\mcR;\fS_d\geo)$ can be considered as an $E_1$-algebra,
i.e. an algebra over the operad of little $1$-cubes,
by squeezing horizontally configurations supported in $\mcR=(0,1)^2$ and by embedding them into disjoint sub-rectangles of $\mcR$ of the form
$(t_1,t_2)\times(0,1)$. The Hurwitz-Moore topological monoid $\mHurm(\fS_d\geo)$ is a strictification of the $E_1$-algebra
$\Hur(\mcR;\fS_d\geo)$.

Unfortunately, for $d\ge3$, the monoid $\pi_0(\Hur(\mcR;\fS_d\geo))\cong \pi_0(\mHurm(\fS_d\geo))$ is not commutative, and therefore we cannot upgrade $\Hur(\mcR;\fS_d\geo)$ to an $E_2$-algebra.
\begin{nota}
 For $r\ge1$ we denote by $E_2(r)$ the space of $r$-tuples $(\iota_1,\dots,\iota_r)$ of self-embeddings $\iota_i\colon\mcR\to\mcR$ satisfying the following properties:
 \begin{itemize}
  \item the embeddings are \emph{rectilinear} in the following strong sense: for each $1\le i\le r$ there is $\hat z_i\in\C$ and $\lambda_i>0$ such that $\iota_i(z)=\lambda z+\hat z_i$ for all $z\in\mcR$;
  \item the closures in $\C$ of the open squares $\iota_i(\mcR)$ are pairwise disjoint and are contained in $\mcR$.
 \end{itemize}
\end{nota}
The spaces $E_2(r)$ assemble into a version of the little $2$-cubes operad; however, we insist that a rectilinear embedding squeezes $\mcR$ as much horizontally as vertically, and that boundaries of little $2$-squares and of their ambient square are disjoint.

Recall that the enveloping group $\cG(\fS_d\geo)$ of $\fS_d\geo$ is the index-2 subgroup $\tilde\fS_d\subset\Z\times\fS_d$ containing pairs $(r,\sigma)$ with $r$ and $\sigma$ of the same parity \cite[Lemma 7.2]{Bianchi:Hur1}.
Recall also that the $\widehat{\fS_d\geo}$-valued total monodromy $\hat\totmon\colon\Hur(\mcR;\fS_d\geo)\to \widehat{\fS_d\geo}$
can be postcomposed with the natural map of PMQs $\widehat{\fS_d\geo}\to\cG(\fS_d\geo)$, yielding the 
$\tilde\fS_d$-valued total monodromy $\omega\colon\Hur(\mcR;\fS_d\geo)\to\tilde\fS_d$.

\begin{defn}
We denote by $\Hur(\mcR;\fS_d\geo)_{(\bullet,\one)}\subset\Hur(\mcR;\fS_d\geo)$ the subspace of configurations $(P,\psi)$
whose $\tilde\fS_d\geo$-valued total monodromy has the form $(h,\one)$ for some even integer $h\in2\Z$.
\end{defn}
Note that $\Hur(\mcR;\fS_d\geo)_{(\bullet,\one)}$ is a union of connected components of $\Hur(\mcR;\fS_d\geo)$,
and it is a sub-$E_1$-algebra of $\Hur(\mcR;\fS_d\geo)$.
\begin{lem}
\label{lem:E2structure}
 The $E_1$-algebra structure on $\Hur(\mcR;\fS_d\geo)_{(\bullet,\one)}$ can be extended to an $E_2$-algebra structure.
\end{lem}
\begin{proof}
 Let $(\iota_1,\dots,\iota_r)\in E_2(r)$ be a configuration of little $2$-cubes, i.e. rectilinear embeddings $\iota_i\colon\mcR\hookrightarrow\mcR$, and choose an extension of each $\iota_i$ to a semi-algebraic
 homeomorphism $\tilde\iota_i\colon\C\to\C$ fixing $*=-\sqrt{-1}\in\C$.

 Let $\fc_1,\dots,\fc_r\in\Hur(\mcR;\fS_d\geo)_{(\bullet,\one)}$
 be configurations, and write $\fc_i=(P_i,\psi_i)$. We let $P=\iota_1(P_1)\sqcup\dots\sqcup \iota_r(P_r)$, and write
 $P=\set{z_{i,j}}_{1\le i\le r,1\le j\le k_i}$, with $z_{i,j}\in\iota_i(P_i)$, for suitable $k_i=|P_i|\ge0$.

 Fix an admissible generating set of $\pi_1(\CmP,*)$ represented by loops $\alpha_{i,j}$ for $1\le i\le r$ and $1\le j\le k_i$, such that $\alpha_{i,j}$ spins clockwise around $z_{i,j}$, and such that for each $1\le i\le r$
 the product $\alpha_{i,1}\dots\alpha_{i,k_i}$ is homotopic in $\CmP$ to a simple loop spinning clockwise around $\iota_i(\mcR)\subset\C$ and disjoint from $\iota_1(\mcR)\sqcup\dots\sqcup\iota_r(\mcR)$.
 
 We define a map of PMQs $\psi\colon\fQ(P)\to\fS_d\geo$ by sending $\alpha_{i,j}\mapsto \psi_i(\tilde\iota_i^{-1}(\alpha_{i,j}))$.
 The map $\psi$ does not depend on the choice of admissible generating set $\set{\alpha_{i,j}}_{1\le i\le r,1\le j\le k_i}$ with the properties above, nor on the choice of extensions $\tilde\iota_i\colon \C\to \C$ of $\iota_i\colon\mcR\to \C$ with the properties above,
 as a consequence of the fact that each configuration $\fc_i$ has a $\tilde\fS_d$-valued total monodromy $\totmon(\fc_i)$ that acts trivially by conjugation on $\fS_d\geo$, and hence $\totmon(\fc_i)$ also acts trivially by global conjugation on
 the space $\Hur(\fS_d\geo)$, and in particular on its subspace $\Hur_{+,(\bullet,\one)}$.
 
 We then let the operation $(\iota_1,\dots,\iota_r)\in E_2(r)$ act on the $r$-tuple of inputs
 $(\fc_1,\dots,\fc_r)$ by giving the constructed output $(P,\psi)$. This defines an action of the operad $E_2$ on $\Hur(\mcR;\fS_d\geo)_{(\bullet,\one)}$.
See also Figure \ref{fig:action_7}.
\end{proof}

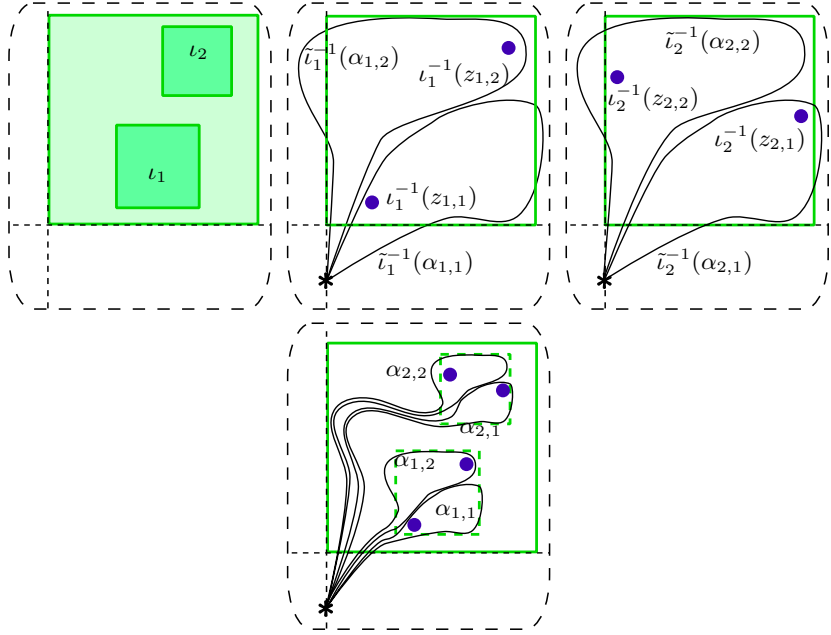
\begin{figure}[ht]
\centering
\begin{overpic}[scale=0.7]{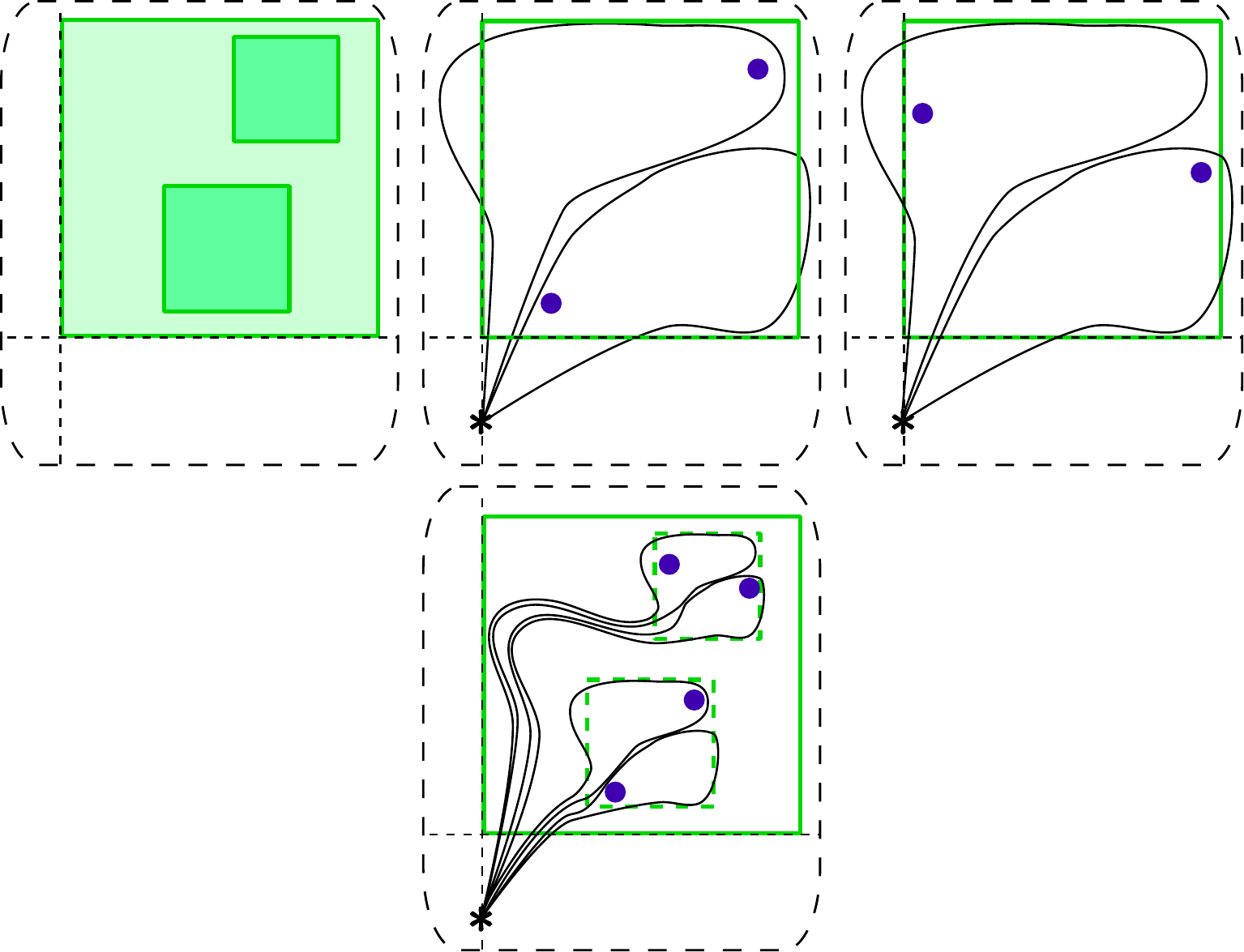}
\put(17,55){\small$\iota_1$}
\put(22,70){\small$\iota_2$}

\put(45,44){\small$\tilde\iota^{-1}_1(\alpha_{1,1})$}
\put(36,69){\small$\tilde\iota^{-1}_1(\alpha_{1,2})$}
\put(46,52){\small$\iota_1^{-1}(z_{1,1})$}
\put(50,67){\small$\iota_1^{-1}(z_{1,2})$}

\put(79,44){\small$\tilde\iota^{-1}_2(\alpha_{2,1})$}
\put(80,71){\small$\tilde\iota^{-1}_2(\alpha_{2,2})$}
\put(86,59){\small$\iota_2^{-1}(z_{2,1})$}
\put(73,64){\small$\iota_2^{-1}(z_{2,2})$}

\put(52,14){\small$\alpha_{1,1}$}
\put(47,20){\small$\alpha_{1,2}$}
\put(55,24){\small$\alpha_{2,1}$}
\put(46,31){\small$\alpha_{2,2}$}

\end{overpic}
\vspace{0.4cm}
\caption{
On top, a configuration $(\iota_1,\iota_2)\in E_2(2)$
and two configurations in $\Hur(\mcR;\fS_d\geo)_{(\bullet,\one)}$, whose $\fS_d\geo$-valued local monodromies
are not represented.
On bottom, the configuration in $\Hur(\mcR;\fS_d\geo)_{(\bullet,\one)}$ arising as $E_2$-composition.
}
\label{fig:action_7}
\end{figure}

We can consider the subspace $\Hur(\mcR;(\fS_d\geo)_+)$, containing configurations $(P,\psi)$ with $\psi\colon\fQ(P)\to\fS_d\geo$ being an augmented map of PMQs; the $E_2$-algebra structure from
Lemma \ref{lem:E2structure} restricts to an $E_2$-algebra structure on the corresponding subspace
$\Hur(\mcR;(\fS_d\geo)_+)_{(\bullet,\one)}$.

The standard inclusion of augmented PMQs $\fS_d\geo\subset\fS_{d+1}\geo$ gives rise to an inclusion of $E_2$-algebras
\[
\Hur(\mcR;(\fS_d\geo)_+)_{(\bullet,\one)}\subset \Hur(\mcR;(\fS_{d+1}\geo)_+)_{(\bullet,\one)},
\]
with image a union of connected components.
\begin{nota}
 We denote by $\Hur(\mcR;(\fS_\infty\geo)_+)_{(\bullet,\one)}$ the $E_2$-algebra
 \[
\Hur(\mcR;(\fS_\infty\geo)_+)_{(\bullet,\one)}:= \bigcup_{d\ge2}\Hur(\mcR;(\fS_d\geo)_+)_{(\bullet,\one)}.
 \]
\end{nota}

Our next goal is to identify the homology of the group completion
\[
H_*(\Omega B \Hur(\mcR;(\fS_\infty\geo)_+)_{(\bullet,\one)}),
\]
where we consider $\Hur(\mcR;(\fS_\infty\geo)_+)_{(\bullet,\one)}$ as an $E_1$-algebra.
We will first compute
$H_*(\Omega B \Hur(\mcR;(\fS_d\geo)_+)_{(\bullet,\one)})$ for all $d\ge2$, and then take the colimit for $d\to\infty$.

\subsection{Comparison between mapping telescopes}
For each $d\ge2$, we replace the $E_2$-algebra $\Hur(\mcR;(\fS_d\geo)_+)_{(\bullet,\one)}$ by the corresponding sub-topological monoid
$\mHurm((\fS_d\geo)_+)_{(\bullet,\one)}$ of $\mHurm((\fS_d\geo)_+)$, in order to be able to use again part of the arguments of Section \ref{sec:mumford}. We fix $d \ge2$ and abbreviate $\mHurm((\fS_d\geo)_+)$ by $\mHurm_+$ throughout the subsection, to simplify formulas.

Let $e$ and $e'$ be as in Notation \ref{nota:propagators}. Let $a=\pa{1,\fc_{(1,2)}}\pa{1,\fc_{(2,3)}}\dots\pa{1,\fc_{(d-1,d)}}$ as in Subection \ref{subsec:propagators}, and let similarly $\dot a:=\pa{1,\fc_{(d-1,d)}}\pa{(1,\fc_{(d-2,d-1)}}\dots \pa{1,\fc_{(1,2)}}$. Let the neutral element $\pa{0,(\emptyset,\one)}\in\mHurm_+$ be abbreviated by $0$. Note that $\pi_0(a\dot a)=\pi_0(\dot a a)=\pi_0(e')\in\pi_0(\mHurm_+)$, as follows from the equalities
\[
\begin{split}
 \hat\totmon(a\dot a)&=\hat\totmon(a)\hat\totmon(\dot a)=\hat{\lc_d}\,\widehat{\lc_d^{-1}}= (\one,\set{1,\dots,d},2d-2)\\
 &=\widehat{\lc_d^{-1}}\,\hat{\lc_d}=\hat\totmon(\dot a)\hat\totmon(a)=\hat\totmon(\dot aa)=\hat\totmon(e')\in\widehat{\fS_d\geo}.
\end{split}
\]
Passing to $\tilde\fS_d$-valued total monodromies, we obtain the equality $\totmon(a\dot a)=\totmon(\dot aa)=\totmon(e')=(2d-2,\one)\in\tilde\fS_d$. In particular, using that the element $(2d-2,\one)$ acts trivially by conjugation on $\fS_d\geo$, and thanks to the weak braiding of $\mHurm_+$ already recalled in the proof of Lemma \ref{lem:braiding}, we obtain that the six maps $e'\cdot-$, $-\cdot e'$, $a\dot a\cdot -$, $-\cdot a\dot a$, $\dot a a\cdot -$ and $-\cdot\dot aa$ are all homotopic to each other when considered as maps $\mHurm_+\to\mHurm_+$.

The mapping telescopes
$\Tel(\mHurm_{+,(\bullet,\one)},0;e')$ and $\Tel(\mHurm_+,a;e')$
can be compared as follows:
\begin{itemize}
 \item there is a map, as in Subsection \ref{subsec:maptel},
 \[
  \br^{\Tel}_a\colon
  \Tel(\mHurm_{+,(\bullet,\one)},0;e')\to\Tel(\mHurm_+,a;e');
 \]
 \item there is a map
 \[
  \br^{\Tel}_{\dot a}\colon
  \Tel(\mHurm_+,a;e')\to\Tel(\mHurm_{+,(\bullet,\one)},a\dot a;e');
 \]
we may then consider $\Tel(\mHurm_{+,(\bullet,\one)},a\dot a;e')$ as a subspace of
$\Tel(\mHurm_{+,(\bullet,\one)},0;e')$, since both spaces are homotopy colimits of sequences of spaces, and the first sequence is obtained from the second by deleting the very first term.
\end{itemize}
We obtain a commutative diagram of homology groups as follows, whose rows have
$H_*(\Tel(\mHurm_{+,(\bullet,\one)},0;e'))$ and $H_*(\Tel(\mHurm_+,a;e'))$
as colimits, respectively
\[
\begin{tikzcd}[row sep=30pt]
 H_*(\mHurm_{+,\one})\ar[r,"e'\cdot-"]\ar[d,"-\cdot a"] & H_*(\mHurm_{+,\hat\totmon(e')})\ar[r,"e'\cdot-"] \ar[d,"-\cdot a"] &
 H_*(\mHurm_{+,\hat\totmon(e'\cdot e')})\ar[r,"e'\cdot-"] \ar[d,"-\cdot a"] &\dots \\
 H_*(\mHurm_{+,\hat\lc_d})\ar[r,"e'\cdot-"] \ar[ur,"-\cdot \dot a"] & H_*(\mHurm_{+,\hat\lc_d\cdot\hat\totmon(e')})\ar[r,"e'\cdot-"]  \ar[ur,"-\cdot \dot a"]&
 H_*(\mHurm_{+,\hat\lc_d\cdot \hat\totmon(e'\cdot e')})\ar[r,"e'\cdot-"] \ar[ur,"-\cdot \dot a"] &\dots \\
\end{tikzcd}
\]
It follows that the maps $\br^{\Tel}_{a}$ and $\br^{\Tel}_{\dot a}$ induce homology isomorphisms between $H_*(\Tel(\mHurm_{+,(\bullet,\one)},0;e'))$ and $H_*(\Tel(\mHurm_+,a;e'))$.

The group completion theorem can be applied to the (weakly) braided topological monoid $\mHurm_{+,(\bullet,\one)}$, for which $e'$ is a propagator:
indeed $\pi_0(\mHurm_{+,(\bullet,\one)})$ is generated by the elements $\hat\tr\cdot\hat\tr\in\widehat{\fS_d\geo}$ for $\tr$ ranging among transpositions in $\fS_d\geo$,
and again by \cite[Proposition 7.11]{Bianchi:Hur1}, for each transposition $\tr$ we can factor $e'$ inside $\widehat{\fS_d\geo}$ as a product of $\hat\tr\cdot\hat\tr$ and some other element, whose image in $\tilde\fS_d$ along the natural map $\widehat{\fS_d\geo}\to\tilde\fS_d$ is $(2d-4,\one)$.
We obtain a commutative diagram
\[
 \begin{tikzcd}
 H_*(\Tel(\mHurm_{+,(\bullet,\one)},0;e'))\ar[r,"\br^{\Tel}_{a}","\cong"']\ar[d,"\Xi_{\mHurm_{+,(\bullet,\one)},0;e'}","\cong"'] & H_*(\Tel(\mHurm_+,a;e')) \ar[d,"\Xi_{\mHurm_+,a;e'}","\cong"']\\
 H_*(\Omega_0 B\mHurm_{+,(\bullet,\one)})\ar[r]& H_*(\Omega_0B\mHurm_+),
 \end{tikzcd}
\]
where the bottom horizontal map is the map induced by the inclusion of monoids $\mHurm_{+,(\bullet,\one)}\hookrightarrow\mHurm_+$. We conclude therefore the following lemma.
\begin{lem}
\label{lem:Hiso0mHurmonetomHurm}
 The inclusion of monoids $\mHurm((\fS_d\geo)_+)_{(\bullet,\one)}\hookrightarrow\mHurm(\fS_d\geo)_+$ induces a homology isomorphism
 $H_*(\Omega_0 B\mHurm((\fS_d\geo)_+)_{(\bullet,\one)})\cong H_*(\Omega_0B\mHurm(\fS_d\geo)_+)$.
\end{lem}

\subsection{Enveloping groups of \texorpdfstring{$\pi_0$}{pi0}}
We fix $d\ge2$ also in this subsection, and abbreviate $\mHurm((\fS_d\geo)_+)$ by $\mHurm_+$.
We want now to determine the enveloping group of
the abelian monoid $\pi_0(\mHurm_{+,(\bullet,\one)}))$; in other words, we want to compute the fundamental group
$\pi_1(B \mHurm_{+,(\bullet,\one)})$.

\begin{lem}
\label{lem:cGmHurmone}
The map $\pi_1(B \mHurm_{+,(\bullet,\one)})\to\pi_1(B\mHurm_+)\cong\tilde\fS_d$ induced by the inclusion
of monoids $\mHurm_{+,(\bullet,\one)}\hookrightarrow\mHurm_+$ is injective, with image the subgroup
$2\Z\subset\tilde\fS_d$.
\end{lem}
\begin{proof}
Using \cite[Propositions 7.11 and 7.13]{Bianchi:Hur1} we can
identify $\pi_0(\mHurm_{+,(\bullet,\one)}))$ with the subset of
$\widehat{\fS_d\geo}=\pi_0(\mHurm_+)$ containing tuples of the following form,
see the notation of \cite[Proposition 7.13]{Bianchi:Hur1}:
\[
(\one;\fP_1,\dots,\fP_\ell;r_1,\dots,r_\ell).
\]
Note that all numbers $r_i$ are automatically even, by the conditions
imposed by \cite[Proposition 7.13]{Bianchi:Hur1}.
Multiplication by $\pi_0(e')$ transforms the above element of $\widehat{\fS_d\geo}$ into the element
\[
 (\one;\set{1,\dots,d};r_1+\dots+r_\ell+2d-2)\in\widehat{\fS_d\geo};
\]
it follows that $\cG(\pi_0(\mHurm_{+,(\bullet,\one)}))$ coincides with the enveloping group of the submonoid of $\widehat{\fS_d\geo}$
spanned by elements of the form $(\one,\set{1,\dots,d},h)$ for $h\ge2$ even; on this submonoid the composition
is just given by summing the third components, and the enveloping group of this submonoid can be identified
with the subgroup of $\tilde\fS_d$ of elements of the form $(h,\one)$, with $h\in2\Z$
\end{proof}

In order to put together Lemmas \ref{lem:Hiso0mHurmonetomHurm} and \ref{lem:cGmHurmone}, we need to replace
$\Omega B\mHurm_+$ by its sub-$E_1$-algebra of connected components mapping to $2\Z$ under the map
$\pi_0\colon \Omega B\mHurm_+\to \pi_0(\Omega B\mHurm_+)\cong\tilde\fS_d$:
roughly speaking, this is a sub-$E_1$-algebra of ``index $\fS_d$'' in $\Omega B\mHurm_+$.
\begin{nota}
 We denote by $\Omega_{(\bullet,\one)} B\mHurm_+\subset \Omega B\mHurm_+$ the aforementioned sub-$E_1$-algebra of $\Omega B\mHurm_+$.
\end{nota}
The $E_1$-algebra $\Omega_{(\bullet,\one)}B\mHurm_+$ is the loop space of some
covering space of $B\mHurm_+$ with $\fS_d$ as group of deck transformations,
and this covering space can be identified, up to homotopy,
with the space $\bHurm_+(\fS_d\geo,\fS_d)_{\one}$ from \cite[Definition 2.4]{Bianchi:Hur3}:
\begin{itemize}
 \item \cite[Theorem 4.1]{Bianchi:Hur3} provides a weak homotopy equivalence
\[
 B\mHurm((\fS_d\geo)_+)\simeq \Hur(\bdiamolr,\bdel;\fS_d\geo,\fS_d)_{\fS_d,\fS_d^{op}};
\]
\item the space $\Hur(\bdiamolr,\bdel;\fS_d\geo,\fS_d)_{\fS_d,\fS_d^{op}}$ is by definition the quotient of another space, denoted
$\Hur(\bdiamolr,\bdel;\fS_d\geo,\fS_d)_{\zdiamleft,\zdiamright}$, by a free, properly discontinuous action of the group $\fS_d\times\fS_d^{op}$;
\item the action of the subgroup $\fS_d^{op}\subset\fS_d\times\fS_d^{op}$ is free and transitive on connected
components of $\Hur(\bdiamolr,\bdel;\fS_d\geo,\fS_d)_{\zdiamleft,\zdiamright}$;
\item the quotient
\[
\pa{\Hur(\bdiamolr,\bdel;\fS_d\geo,\fS_d)_{\zdiamleft,\zdiamright}}/\fS_d^{op}
\]
is thus connected and carries a residual action of $\fS_d$; we can identify this partial quotient with
the component $\Hur(\bdiamolr,\bdel;\fS_d\geo,\fS_d)_{\zdiamleft,\zdiamright;\one}$ of configurations in
$\Hur(\bdiamolr,\bdel;\fS_d\geo,\fS_d)_{\zdiamleft,\zdiamright}$ with $\fS_d$-valued total monodromy equal to $\one\in\fS_d$;
the corresponding action of $\fS_d$ can be identified with the action by global conjugation of $\fS_d$ on
$\Hur(\bdiamolr,\bdel;\fS_d\geo,\fS_d)_{\zdiamleft,\zdiamright;\one}$;
\item the composite map of monoids
\[
\begin{tikzcd}[column sep=10pt]
\widehat{\fS_d\geo}\ar[r,"\cong"] & \pi_0(\mHurm((\fS_d\geo)_+))\ar[r,"\bs"]& \pi_0(\Omega B\mHurm((\fS_d\geo)_+))
\ar[d,"\cong"] & \\
 & & \pi_1\pa{\Hur(\bdiamolr,\bdel;\fS_d\geo,\fS_d)_{\fS_d,\fS_d^{op}}}\ar[r,"\mathrm{deck}"] & \fS_d
\end{tikzcd}
\]
coincides with the canonical map of PMQs $\widehat{\fS_d\geo}\to\fS_d$; here
``$\mathrm{deck}$'' denotes the map of groups corresponding to the action of $\pi_1(\Hur(\bdiamolr,\bdel;\fS_d\geo,\fS_d)_{\fS_d,\fS_d^{op}})$ by deck transformations on the covering space
\[
\begin{tikzcd}
\Hur(\bdiamolr,\bdel;\fS_d\geo,\fS_d)_{\zdiamleft,\zdiamright;\one}\ar[d] & \\
\pa{\Hur(\bdiamolr,\bdel;\fS_d\geo,\fS_d)_{\zdiamleft,\zdiamright;\one}}/\fS_d\ar[r,equal]&\Hur(\bdiamolr,\bdel;\fS_d\geo,\fS_d)_{\fS_d,\fS_d^{op}};
\end{tikzcd}
\]
this fact can be checked by analysing the behaviour of the generators $\hat\tr\in\widehat{\fS_d\geo}$ along the composition, where $\tr\in\fS_d\geo$ is a transposition;
\item it follows that the weak homotopy equivalence obtained by looping  \cite[Theorem 4.1]{Bianchi:Hur3}
 \[
\Omega B\mHurm((\fS_d\geo)_+)\simeq \Omega\Hur(\bdiamolr,\bdel;\fS_d\geo,\fS_d)_{\fS_d,\fS_d^{op}}
 \]
 restricts to a weak homotopy equivalence
 \[
 \Omega_{(\bullet,\one)} B\mHurm((\fS_d\geo)_+)\simeq \Omega\Hur(\bdiamolr,\bdel;\fS_d\geo,\fS_d)_{\zdiamleft,\zdiamright;\one}
 \]
\item the space $\Hur(\bdiamolr,\bdel;\fS_d\geo,\fS_d)_{\zdiamleft,\zdiamright;\one}$ is weakly equivalent to the topological mo\-noid
$\bHurm_+(\fS_d\geo,\fS_d)_{\one}$, as already observed in \cite[Subsection 4.4]{Bianchi:Hur3}.
\end{itemize}

We obtain the following lemma.
\begin{lem}
\label{lem:firstHequiv}
 The map $\Omega B\mHurm((\fS_d\geo)_+)_{(\bullet,\one)}\to\Omega B\mHurm((\fS_d\geo)_+)$ induced by the inclusion
 of monoids $\mHurm((\fS_d\geo)_+)_{(\bullet,\one)}\hookrightarrow\mHurm((\fS_d\geo)_+)$ has image in the subspace
 $\Omega_{(\bullet,\one)}B\mHurm((\fS_d\geo)_+)\simeq \Omega\bHurm_+(\fS_d\geo,\fS_d)_{\one}$,
 and induces an isomorphism on homology groups
 \[
  H_*(\Omega B\mHurm((\fS_d\geo)_+)_{(\bullet,\one)})\overset{\cong}{\to}
 H_*(\Omega_{(\bullet,\one)}B\mHurm((\fS_d\geo)_+))\cong H_*(\Omega\bHurm_+(\fS_d\geo,\fS_d)_{\one}).
 \]
\end{lem}
In fact both loop spaces $\Omega B\mHurm((\fS_d\geo)_+)_{(\bullet,\one)}$
and $\Omega\bHurm_+(\fS_d\geo,\fS_d)_{\one}$
are double loop spaces:
\begin{itemize}
 \item on the one hand we have $\Omega B\mHurm((\fS_d\geo)_+)_{(\bullet,\one)}\simeq \Omega^2B^2(\Hur(\mcR;(\fS_d\geo)_+)_{(\bullet,\one)})$,
where $B^2(\Hur(\mcR;(\fS_d\geo)_+)_{(\bullet,\one)})$ denotes the double bar construction
of the $E_2$-algebra $\Hur(\mcR;(\fS_d\geo)_+)_{(\bullet,\one)}$;
\item on the other hand we have
$\Omega\bHurm_+(\fS_d\geo,\fS_d)_{\one}\simeq \Omega^2 B\bHurm_+(\fS_d\geo,\fS_d)_{\one}$,
since the topological monoid $\bHurm_+(\fS_d\geo,\fS_d)_{\one}$
is connected, hence group-like.
\end{itemize}
The map $\Omega B\mHurm((\fS_d\geo)_+)_{(\bullet,\one)}\to
\Omega\bHurm_+(\fS_d\geo,\fS_d)_{\one}$ from Lemma \ref{lem:firstHequiv} is in fact a map of double loop spaces.
We can now use again \cite[Theorem 4.1]{Bianchi:Hur3} and write a sequence of weak equivalences,
homeomorphisms and equalities of double loop spaces
\[
\begin{split}
\Omega\bHurm_+(\fS_d\geo,\fS_d)_{\one}&=\Omega\bHurm_+(\fS_d\geo,\fS_d)\simeq
 \Omega(\Omega B\bHurm_+(\fS_d\geo,\fS_d))\\
 &=\Omega^2 (B\bHurm_+(\fS_d\geo,\fS_d))
 \simeq
 \Omega^2(\Hur(\diamo,\del\diamo;\fS_d\geo,\fS_d)_{\fS_d,\fS_d^{op}})\\
 &\cong
 \Omega^2(\Hur(\diamo,\del\diamo;\fS_d\geo,\fS_d)_{\zdiamleft,\zdiamright;\one})\simeq
 \Omega^2\Hur(\cR,\del\cR;\fS_d\geo,\fS_d)_{\one}.
\end{split}
\]
We have used that the monoid $\bHurm_+(\fS_d\geo,\fS_d)$ is group-like \cite[Theorem 2.19]{Bianchi:Hur3},
that a loop space does not change by adding connected components to a space, and that a double loop space
does not change up to homeomorphism by replacing a space with a covering of it: in fact
$\Hur(\diamo,\del\diamo;\fS_d\geo,\fS_d)_{\zdiamleft,\zdiamright;\one}$ is a covering space
of $\Hur(\diamo,\del\diamo;\fS_d\geo,\fS_d)_{\fS_d,\fS_d^{op}}$, with $\fS_d$ as group of deck transformations.

Using Lemma \ref{lem:firstHequiv} we thus obtain a map of double loops spaces which is a homology equivalence
\[
 \Omega B\mHurm((\fS_d\geo)_+)_{(\bullet,\one)}\to \Omega^2\Hur(\cR,\del\cR;\fS_d\geo,\fS_d)_{\one}.
\]
The arguments of this subsection (in particular those relying on \cite{Bianchi:Hur3}) are natural with respect
to the inclusions of PMQ-group pairs $(\fS_d\geo,\fS_d)\hookrightarrow(\fS_{d+1}\geo,\fS_{d+1})$. Letting $d$ tend
to infinity, we thus obtain a map of double loops spaces which is a homology equivalence
\[
 \Omega B\mHurm((\fS_\infty\geo)_+)_{(\bullet,\one)}\to \Omega^2\bB_\infty.
\]
In the last formula $\Omega B\mHurm((\fS_\infty\geo)_+)_{(\bullet,\one)}$ denotes the homotopy colimit
of the sequence $\Omega B\mHurm((\fS_d\geo)_+)_{(\bullet,\one)}$, for $d\to\infty$;
we can compute the same homotopy colimit using the double loop spaces
$\Omega^2 B^2(\Hur(\mcR;(\fS_d\geo)_+)_{(\bullet,\one)})$, thus obtaining a map of double loop spaces
which is a homology equivalence
\[
 \Omega^2 B^2(\Hur(\mcR;(\fS_\infty\geo)_+)_{(\bullet,\one)})\to \Omega^2\bB_\infty.
\]
Here we use that both $\Omega^2$ and $B^2$ are functors that preserve sequential homotopy colimits.
We can then deloop twice the previous, i.e. apply again the functor $B^2$ from $E_2$-algebras to simply connected spaces, which sends homology equivalences
to homology equivalences; we obtain a homology equivalence of simply connected spaces
\[
B^2(\Hur(\mcR;(\fS_\infty\geo)_+)_{(\bullet,\one)})\to \bB_\infty,
\]
which is therefore a weak homotopy equivalence;
a posteriori we understand that we have been working with weak equivalences, instead of just homology equivalences, essentially in the entire subsection!

\subsection{\texorpdfstring{$E_2$}{E2}-algebras from moduli spaces}
\begin{defn}
\label{defn:scM}
 Let $d\ge1$. A \emph{$d$-little Riemann surface} is a compact, possibly disconnected Riemann surface $\cS$ with $d$ ordered boundary components $\del_1\cS,\dots,\del_d\cS$, such that each connected component of $\cS$ touches $\del\cS$, and such that 
 $\del_i$ is endowed with a \emph{little parametrisation}: by this we mean that there is a collar neighbourhood $\del_i\cS\subset U_i\subset\cS$ and a continuous map $\mathfrak{u}_i\colon U_i\to\cR=[0,1]^2\subset\C$ which is a homeomorphism onto its image, sends $\del_i\cS$ homeomorphically to $\del\cR$,
 and restricts to a holomorphic map on $U_i\setminus\cS$; moreover, we are only considering $[U_i,\mathfrak{u}_i]$ as a germ of parametrised collar neighbourhood of $\del_i\cS$, i.e. as an equivalence class of parametrised collar neighbourhoods of $\del_i\cS$ up to the equivalence relation given by agreement of parametrisations on some smaller collar neighbourhood: such a germ is what we mean by \emph{little parametrisation} of a boundary component of $\cS$.
 
 We denote by $\scM_d$ the moduli space of $d$-little Riemann surfaces. Two $d$-little Riemann surfaces are considered equivalent if there is a biholomorphism between them which is compatible with the ordering of boundary components and with the little parametrisations of boundary components.
\end{defn}
The space $\scM_d$ is disconnected, and there is a connected component for every tuple $(\fP_1,\dots,\fP_\ell;2g_1,\dots,2g_\ell)$,
where $\fP_1,\dots,\fP_\ell$ is a partition of $\set{1,\dots,d}$ into $1\le \ell\le d$ non-empty subsets, and $2g_1,\dots,2g_\ell$ are even numbers $\ge0$: this connected component contains $d$-little Riemann surfaces $\cS$ having precisely $\ell$ components $\cS(1),\dots,\cS(\ell)$ of genera $g_1,\dots,g_\ell$, and such that $\del_i\cS\subset\cS(j)$ if and only if $i\in \fP_j$. The use of even numbers instead of just natural numbers will become clear later.

Given an operation $(\iota_1,\dots,\iota_r)\in E_2(r)$ and given $d$-little surfaces $\cS_1,\dots,\cS_r$, we can create a new $d$-little surface $\cS$ by sewing
$\cS_1\sqcup\dots\sqcup\cS_r$ together with $d$ ordered copies of $\cR\setminus\pa{\iota_1(\mcR)\sqcup\dots\sqcup\iota_r(\mcR)}$: we sew $\del_i\cS_j$ with the $i$\textsuperscript{th} copy of $\del(\iota_j(\mcR))$, leveraging on the fact that both curves are identified with $\del\cR$, either by translation and dilation, or by the little parametrisation. The Riemann structure on the sewing locus is defined in the evident way. The outer boundary $\del\cR$
of each copy of $\cR\setminus\pa{\iota_1(\mcR)\sqcup\dots\iota_r(\mcR)}$ is endowed with a tautological little parametrisation. The previous construction leads to the following, which we state as a lemma.
\begin{lem}
 For all $d\ge1$ the space $\scM_d$ is an $E_2$-algebra.
\end{lem}
Adjoining a copy of $\cR$, whose boundary is labeled $d+1$, gives an injective map of $E_2$-algebras $\scM_d\hookrightarrow\scM_{d+1}$.
\begin{nota}
 We denote by $\scM_\infty$ a homotopy colimit in the category of $E_2$-algebras of the
 sequence $\scM_1\hookrightarrow\scM_2\hookrightarrow\scM_3\hookrightarrow\dots$.
\end{nota}

For all $d\ge2$ we can construct a map of $E_2$-algebras
\[
\bc_d\colon \Hur(\mcR;\fS_d\geo)_{(\bullet,\one)}\to \scM_d
\]
as follows: given a configuration $(P,\psi)\in\Hur(\mcR;\fS_d\geo)_{(\bullet,\one)}$, we consider the $d$-fold, possibly disconnected cover $f\colon\mathring{\cS}\to\CmP$ with monodromy induced by $\psi$, we compactify $\mathring{\cS}$
to a closed, possibly disconnected, smooth Riemann surface $\bar\cS$ admitting a $d$-fold branch cover map  $f\colon\bar\cS\to\C P^1$,
and finally we take $\cS=f^{-1}(\cR)$: the fact that $\totmon(P,\psi)\in\tilde\fS_d$ has the form $(h,\one)$, for some $h\in2\Z$,
implies that $\cS$ has $d$ boundary curves; these boundary curves have a canonical order, and the very projection $f$ endowes each of them with a little parametrisation.

The explicit description of $\pi_0(\scM_d)$ given above can be compared with the description of $\pi_0(\Hur(\mcR;(\fS_d\geo)_+)_{(\bullet,\one)}$ derived from \cite[Proposition 7.13]{Bianchi:Hur1}: the conclusion is that, for all $d\ge2$, the map $\bc_d$ induces a bijection on path components
$\pi_0(\Hur(\mcR;\fS_d\geo)_{(\bullet,\one)})\cong\pi_0(\scM_d)$.

For all $d\ge2$ the following diagram of $E_2$-algebras commutes strictly
\[
 \begin{tikzcd}
  \Hur(\mcR;\fS_d\geo)_{(\bullet,\one)}\ar[r,hook]\ar[d,"\bc_d"] & \Hur(\mcR;\fS_{d+1}\geo)_{(\bullet,\one)}\ar[d,"\bc_{d+1}"]\\
  \scM_d\ar[r,hook] &\scM_{d+1}
 \end{tikzcd}
\]
\begin{nota}
 We denote by $\bc_\infty\colon \Hur(\mcR;\fS_\infty\geo)_{(\bullet,\one)}\to\scM_\infty$ the map of $E_2$-algebras between the homotopy colimits induced by the sequence of maps $\bc_d$ for $d\ge2$.
\end{nota}

The theorem of Madsen and Weiss identifies $B^2\scM_d$ up to weak equivalence with
$\Omega^{\infty-2}\MTSO(2)$ for all $d\ge2$,
in such a way that the map of spaces $B^2\scM_d\to B^2\scM_{d+1}$ induced by the inclusion of $E_2$-algebras
$\scM_d\hookrightarrow\scM_{d+1}$ can be identified up to homotopy with the identity of $\Omega^{\infty-2}\MTSO(2)$.
In particular we also have $B^2\scM_\infty\simeq\Omega^{\infty-2}\MTSO(2)$, and looping twice we have
a weak equivalence of double loop spaces
\[
\Omega^2B^2\scM_\infty\simeq \Omega^\infty\MTSO(2).
\]
Our next aim is to prove that the map of double loop spaces
\[
 \Omega^2B^2\bc_\infty \colon \Omega^2B^2 \Hur(\mcR;\fS_\infty\geo)_{(\bullet,\one)}\to\Omega^2B^2\scM_\infty
\]
is a weak equivalence. The fact that each map $\bc_d$ is a bijection on path components
implies that also $\bc_\infty$ is a bijection on path components, and hence also $\Omega^2B^2\bc_\infty$
is a bijection on path components.
We thus only have to prove that 
$\Omega^2_0B^2\bc_\infty \colon \Omega^2_0B^2 \Hur(\mcR;\fS_\infty\geo)_{(\bullet,\one)}\to\Omega^2_0B^2\scM_\infty$
is a weak equivalence.

We can identify up to homotopy the previous map with the composition of weak homotopy equivalences
\[
\begin{split}
\Omega^2_0B^2 \Hur(\mcR;\fS_\infty\geo)_{(\bullet,\one)} &\simeq \Omega_0B\mHurm((\fS_\infty\geo)_+)_{(\bullet,\one)}\\
&\simeq \Omega_0B\mHurm((\fS_\infty\geo)_+)\\
&\simeq \pa{\mHurm((\fS_\infty\geo)_+)_{\kld{\infty}_\infty}}^+ \simeq \pa{\fM_{\infty,1}}^+\\
&\simeq \Omega^2_0B^2\scM_1\simeq \Omega^2_0B^2\scM_\infty.
\end{split}
\]
Here $(-)^+$ denotes the Quillen plus construction of a space, and in particular the weak equivalence
$\pa{\mHurm((\fS_\infty\geo)_+)_{\kld{\infty}_\infty}}^+ \simeq \pa{\fM_{\infty,1}}^+$ is induced
by the zig-zag of weak equivalences from Lemma \ref{lem:zigzag}.

Delooping twice we obtain a weak equivalence
\[
 B^2 \Hur(\mcR;\fS_\infty\geo)_{(\bullet,\one)} \simeq B^2\scM_\infty\simeq \Omega^{\infty-2}\MTSO(2).
\]
Finally, using the weak equivalence $B^2(\Hur(\mcR;(\fS_\infty\geo)_+)_{(\bullet,\one)})\to \bB_\infty$
from the previous subsection, we have completed the proof of Theorem \ref{thm:main5}.

\appendix
\section{Deferred proofs}
\subsection{Continuity of \texorpdfstring{$\ccv$}{ccv}}
\label{subsec:ccvcontinuous}
Recall Definition \ref{defn:ccOud}. To prove continuity of $\check\cv$, it suffices to prove continuity of the composite
\[
 \tilde{\ccv}\colon \tccOud\to\ccOud\overset{\check\cv}{\to} \Hur(\C,(\fS_d\geo)_+)_{\klud_g}
\]
where the first map is the quotient map $(\fr,f)\mapsto [\fr,f]$.

Recall from Subsection \ref{subsec:classifyingspaces} that there is a surface bundle $\pr_{\ccO}\colon\scF^{\ccO}_{g,n}\to\ccOud$
with fibres Riemann surfaces of type $\sgn$ (i.e., each fibre has genus $g$ and is endowed with $n$ directed marked points), and there is a continuous map $f_{\ccO}\colon\scF^{\ccO}_{g,n}\to\C P^1$ restricting to $\ud$-directed meromorphic functions on fibres of $\pr_{\ccO}$. In fact, the surface bundle $p_{\ccO}$ is obtained from the surface bundle
\[
 \tilde\pr_{\ccO}\colon \tilde\scF^{\ccO}_{g,n}:=\sgn\times \tccOud\to\tccOud,
\]
where the bundle map is projection on the second factor: more precisely,
one quotients the total space by the diagonal action of $\Diff_{g,n}$, and the base space by the action of $\Diff_{g,n}$.
Each fibre of $\tilde\pr_{\ccO}$ is a Riemann surface of type $\sgn$ (in fact, it is a copy of $\sgn$ with a Riemann structure $\fr$ dictated by the selected point $(\fr,f)$ on the base space).
The composite $\tilde\scF^{\ccO}_{g,n}\to\scF^{\ccO}_{g,n}\overset{f_{\ccO}}{\to}\C P^1$ is a continuous map $\tilde f_{\ccO}$ restricting to a $\ud$-directed meromorphic function on fibres of $\tilde\pr_{\ccO}$: more precisely, $\tilde f_{\ccO}$ restricts to the map $f\colon\sgn\to\C P^1$ on the fibre $\sgn\times(\fr,f)=\tilde\pr_{\ccO}^{-1}(\fr,f)$.

Let now $(\fr,f)$ be a point in $\tccOud$, and let $(P,\psi)=\tilde{\ccv}(\fr,f)$.
We want to show that the preimage along $\tilde{\ccv}$ of a sufficiently small neighbourhood of $(P,\psi)$ contains a neighbourhood of $(\fr,f)$.
Let $V\subset\bH_{>\Im(*_P)}$ be a convex, compact set containing all points of $P\cup\set{0}$ in its interior.
Write $P=\set{z_1,\dots,z_k}$ and let $\uU=(U_1,\dots,U_k)$ be an adapted covering of $P$: the open sets $U_i\subset\C$ are convex, relatively compact and have disjoint closures; we further assume $U_i\subset V$ for all $1\le i\le k$. Let
\[
\uzeta=\set{\zeta_{i,j}\,|\,1\le i\le k,1\le j\le\lambda_i}\subset\sgn\setminus\uQ
\]
be the set of critical points of $f$ lying over $\C$, with $\zeta_{i,j}$ lying over $z_i$, and let $\cS\subset\sgn$ be a compact subsurface of $\sgn$ with the following property: the image along $f$ of $\sgn\setminus\cS$ is a relatively compact subset of $\uU\cup \C P^1\setminus V$. Such $\cS$ can be taken as the complement in $\sgn$ of a small neighbourhood of $\uQ\cup\uzeta$; in particular, let us assume that $\cS$ is a surface of genus $g$ with $n+\sum_{i=1}^k\lambda_i$ boundary curves. We denote by $\beta_{i,j}\subset\del\cS$ the curve bounding a disc $D_{i,j}\subset\sgn$ that contains $\zeta_{i,j}$, for $1\le i\le k$ and $1\le j\le \lambda_i$; and we denote by $\beta_{\infty,i}\subset\del\cS$ the curve bounding a disc $D_i\subset\sgn$ that contains $Q_i$, for $1\le i\le n$.

Let $\tilde\scF^{\ccO}_{\cS}:=\cS\times\tccOud  \subset \tilde\scF^{\ccO}_{g,n}$, and consider
$df_{\ccO}$ as a section of the vertical cotangent bundle of $\tilde\pr_{\ccO}$, defined only over the subspace
$\tilde\scF^{\ccO}_{\cS}$: the restriction of $df_{\ccO}$ over $\cS\times(\fr',f')$
is defined to be the differential $df'$, for all $(f',\fr')\in\tccOud$.
Note that both $f'$ and $df'$ have singularities near the points $Q_i$, and that is one reason why we restrict $df_{\ccO}$ to $\cS$. More precisely, for $p\in\sgn$, $df'_p$ is really a $\C$-linear map $T_p\sgn\to T_{f'(p)}\C P^1$, and only assuming that $f'(p)\neq \infty$ we can identify $T_{f'(p)}\C P^1= T_{f'(p)}\C\cong \C$, and thus treat $df'_p$ as an element of the dual of $T_p\sgn$.

The section $df_{\ccO}$ does not vanish on any point of $\cS\times(\fr,f)$; similarly, the function $f_{\ccO}$ admits no zero on any point of $D_i\times(\fr,f)$, for all $1\le i\le n$.
Using compactness of $\cS$ and the discs $D_i$, we can find a neighbourhood
$(\fr,f)\in \scU\subset\tccOud$ such that, for all $(\fr',f')\in\scU$, the following hold:
\begin{itemize}
 \item $df_{\ccO}$ does not vanish on any point of $\cS\times(\fr',f')$;
 \item for all $1\le i\le n$, $f_{\ccO}$ does not vanish on any point of $D_i\times (\fr',f')$.
\end{itemize}
This implies that, for $(\fr',f')\in\scU$, the function $f'\colon(\sgn,\fr')\to\C P^1$ admits no critical points in $\cS\subset\sgn$.
Up to shrinking $\scU$, we can also assume the following, again using the compactness of $\cS$: for 
$(\fr',f')\in\scU$, the image of $\cS$ along $f'$ is relatively compactly contained in $\uU\cup \C P^1\setminus V$. It follows
that for $(\fr',f')\in\scU$, the critical values of $f'\colon(\sgn,\fr')\to\C P^1$ are contained in $\uU\cup \C P^1\setminus V$.

The next step is to prove that for $(\fr',f')\in\scU$ and for all $1\le i\le n$, the function $f'\colon(\sgn,\fr')\to\C P^1$ admits no critical point $p\in D_i$, except, possibly, the point $Q_i$. Suppose instead that $\set{p_1,\dots,p_l}\subset D_i\setminus\set{Q_i}\subset\sgn\setminus\set{Q_i}$ are the critical points of $f'$ contained in $D_i$ and different from $Q_i$. Then $1/f'$ is a holomorphic function on $D_i$ (here we use that $f'$ admits no zero on $D_i$), and we can apply the Riemann-Hurwitz formula to compute
\[
 1=\chi(D_i)=\mathrm{wind}_{d(1/f')}(\beta_{\infty,i})+\mathrm{ind}_{d(1/f')}(Q_i)+\sum_{j=1}^l\mathrm{ind}_{d(1/f')}(p_j).
\]
Here $\mathrm{wind}_{d(1/f')}$ is the winding number of $d(1/f')=-df'/(f')^2$ along the curve $\beta_{\infty,i}$, which is well-defined because $d(1/f')$ does not vanish along $\beta_{\infty,i}$; similarly, we denote by $\mathrm{ind}_{d(1/f')}(p_j)$ and $\mathrm{ind}_{d(1/f')}(Q_i)$ the indices of $d(1/f')$ at the critical points $p_j$ and at $Q_i$. We can compare the above computation with the similar one obtained using $1/f$, recalling that the only critical point of $f$ in $D_i$ is, possibly, $Q_i$:
\[
 1=\chi(D_i)=\mathrm{wind}_{d(1/f)}(\beta_{\infty,i})+\mathrm{ind}_{d(1/f)}(Q_i)=-(d_i-2)+(d_i-1)
\]
We can now use the continuity of the winding number along the curve $\beta_{\infty,i}\subset\sgn$ with respect to no-where-vanishing sections of the cotangent bundle of $\sgn$ along $\beta_{\infty,i}$: up to shrinking $\scU$, we can assume the equality
\[
\mathrm{wind}_{d(1/f')}(\beta_{\infty,i})=\mathrm{wind}_{d(1/f)}(\beta_{\infty,i})=-(d_i-2).
\]
Moreover, since $f'$ is $\ud$-adapted, we have $\mathrm{ind}_{d(1/f')}(Q_i)=d_i-1$. It follows that
$\sum_{j=1}^l\mathrm{ind}_{d(1/f')}(p_j)=0$, and since each summand is strictly positive, we conclude that this is an empty sum, i.e. $l=0$.
This implies that for all $(\fr',f')\in\scU$, the set $P'$ of critical values of $f'$ in $\C$ is contained in $\uU\subset\C$.

Write $P'=\set{z'_{i,m}\,|\, 1\le i\le k,1\le m\le k'_i}$, for some integers $k'_i\ge0$, such that $z'_{i,m}$ is contained in $U_i$ (soon we will see that actually $k'_i\ge1$).
For all $1\le i\le k$, $1\le m\le k'_i$ and $1\le j\le \lambda_i$ we denote by
$\zeta'_{i,m,j,1},\dots,\zeta'_{i,m,j,\lambda'_{i,m,j}}$ the critical points of $f'$ lying in $D_{i,j}$ and lying over $z'_{i,m}$. We can then compute
\[
 1=\chi(D_{i,j})=\mathrm{wind}_{d(1/f')}(\beta_{i,j})+\sum_{m=1}^{k'_i}\sum_{r=1}^{\lambda'_{i,m,j}}\mathrm{ind}_{d(1/f')}(\zeta'_{i,m,j,r}).
\]
Similarly as above, up to shrinking $\scU$ we can assume the equality $\mathrm{wind}_{d(1/f')}(\beta_{i,j})=\mathrm{wind}_{d(1/f)}(\beta_{i,j})$ for all $1\le i\le k$ and $1\le j\le\lambda_i$; we thus obtain, for all $1\le i\le k$ and $1\le j\le \lambda_i$ the equality
\[
\sum_{m=1}^{k'_i}\sum_{r=1}^{\lambda'_{i,m,j}}\mathrm{ind}_{d(1/f')}(\zeta'_{i,j,m,r})=\mathrm{ind}_{d(1/f)}(\zeta_{i,j}).
\]
Summing over $j$, and setting $\lambda'_{i,m}=\sum_{j=1}^{\lambda_i}\lambda'_{i,m,j}$, we obtain
\[
 \sum_{m=1}^{k'_i}\pa{d-\lambda'_{i,m}}=d-\lambda_i.
\]
In particular, since for all $1\le i\le k$ we have $d-\lambda_i>0$, we must also have $k'_i>0$.

If $(P',\psi')=\tilde{\ccv}(\fr',f')$, then we can let $(\alpha'_{i,m})_{1\le i\le k,1\le m\le k'_i}$ be loops representing an admissible generating set of $\pi_1(\CmP',*_{P'})$. The inclusion $\C\setminus\uU\subset\CmP'$ induces an inclusion of fundamental groups $\pi_1(\C\setminus\uU,*_{\uU})\subseteq\pi_1(\CmP',*_{\uU})$. The former group can be identified with $\pi_1(\CmP,*_P)$, whereas the second can be identified with $\pi_1(\CmP',*_{P'})$: both identification come from inclusion of subspaces of $\C$ and translations of basepoints along straight segments. We thus get an inclusion
$\pi_1(\CmP,*_P)\subset\pi_1(\CmP,*_{P'})$. We can assume that for all $1\le i\le k$, the product
$\alpha'_{i,1}\dots\alpha'_{i,k'_i}\in\pi_1(\CmP',*_{P'})$ corresponds to the image, under the above inclusion, of an element $\alpha_i\in\pi_1(\CmP,*_P)$, such that $\alpha_1,\dots,\alpha_k$ form an admissible generating set of
$\pi_1(\CmP,*_P)$. The above formula becomes, for all $1\le i\le k$, the equality
\[
 \sum_{m=1}^{k'_i}N(\psi'(\alpha'_{i,m}))=N(\psi(\alpha_i)).
\]
If we prove that the product of permutations $\psi'(\alpha'_{i,1})\dots\psi'(\alpha'_{i,k'_i})$ equals $\psi(\alpha_i)$ in the symmetric group $\fS_d$, then the equality
\[
 \psi'(\alpha'_{i,1})\dots\psi'(\alpha'_{i,k'_i})=\psi(\alpha_i)
\]
also holds in $\fS_d\geo$, and we could conclude that $(P',\psi')$ lies in the neighbourhood $\fU(P,\psi;\uU)$ of $(P,\psi)$ in $\Hur(\C,(\fS_d\geo)_+)_{\klud_g}$;
we would then have that $\tilde{\ccv}$ sends the entire neighbourhood $\scU$ of $(\fr,f)$ inside $\fU(P,\psi;\uU)$,
thus concluding the proof of continuity of $\tilde{\ccv}$, and hence the proof of continuity of $\ccv$.

Use translations of basepoints along straight segments to identify, for all $P'\subset\uU$ (including $P$), $\pi_1(\CmP',*_{P'})\cong \pi_1(\CmP',*_{\uU})$. Denote by $\bar\uU$ the closure of $\uU$ in $\C$, and
suppose that $\alpha_i$ is a smooth, immersed simple loop $[0,1]\to\C\setminus\bar\uU$ based at $*_{\uU}$; denote by $\sigma_i=\psi(\alpha_i)\in\fS_d\geo$, and regard $\psi(\alpha_i)$ as a permutation in $\fS_d$. For all $1\le j\le d$ we can lift $\alpha_i$ along $f$ to a smooth, immersed path $\alpha_{i,(j)}\colon[0,1]\to\sgn$, starting at the point labeled $j$ in $f^{-1}(*_{\uU})$ and ending at the point labeled $\sigma_i(j)$: here we use that $\Im(*_{\uU})\le\Im(*_{P})$, so $f^{-1}(*_{\uU})$ also comes with an identification with $\set{1,\dots,d}$.

Let $I_{\uU}\subset\C P^1$ be the segment starting at $*_{\uU}$ and running horizontally towards right in $\C$ to $\infty$. For all $1\le j\le d$
denote by $I_{\uU,j}\subset\sgn$ the lift of $I_{\uU}$ along $f$ starting at the point $*_{\uU,j}\in f^{-1}(*_{\uU})$ labelled $j$. We parametrise $I_{\uU,j}$ as a smooth path $[0,1]\to\sgn$ exiting from $Q_{\ell(j)}$ with velocity $\mu(j)\cdot X_{\ell(j)}$,
for some $1\le\ell(j)\le n$ and some $0\le\mu(j)\le d_{\ell(j)}$: compare with Figure \ref{fig:moduli_3}.

We define a piecewise smooth path $\tilde\alpha_{i,(j)}\colon[0,3]\to\sgn\setminus f^{-1}(\bar\uU)$ by concatenating $I_{\uU,j}$, $\alpha_{i,(j)}$ and the inverse of $I_{\uU,\sigma_i(j)}$, on the segments $[0,1]$, $[1,2]$ and $[2,3]$ respectively.

Using that $\sgn\setminus f^{-1}(\bar\uU)$ is open in $\sgn$ and $[0,3]$ is compact, up to shrinking $\scU$ we may assume that
for all $(\fr',f')\in\scU$ the following hold:
\begin{itemize}
 \item $f'\circ\tilde\alpha_{i,(j)}\colon [0,3]\to\C P^1$ is a piecewise smooth loop based at $\infty$ and taking values in $\C P^1\setminus \bar\uU$.
\end{itemize}
Moreover, up to shrinking $\scU$, and using again compactness of $[0,3]$, we can assume that the smooth functions $f,f'\colon\sgn\to\C P^1$ are close enough in the $C^{\infty}$-topology, and that the Riemann structures $\fr,\fr'$ are also close enough with respect to the $C^{\infty}$-topology, that there is an oblique hyperplane of the form
 \[
\set{\Im-\Re<\upsilon}=\set{z\in\C\,|\,\Im(z)-\Re(z)<\upsilon}\subset\C,
 \]
containing $\bar\uU$ and such that the following holds:
 \begin{itemize}
 \item for all $1\le i\le k$ and $1\le j\le d$, the loops $f'\circ\tilde\alpha_{i,(j)})\colon[0,3]\to\C P^1\setminus\bar\uU$ and
 $f\circ\tilde\alpha_{i,(j)}\colon[0,3]\to\C P^1\setminus\bar\uU$, which are loops based at $\infty$, actually take value inside $\set{\Im-\Re<\upsilon}\setminus\bar\uU$,
 and they are based-homotopic loops inside $\set{\Im-\Re<\upsilon}\setminus\bar\uU$.
\end{itemize}
It follows that, along all identifications, $\psi'(\alpha_i)$ is a permutation sending $j\mapsto\sigma_i(j)$ for all $1\le j\le d$, just as $\psi(\alpha_i)$ does; in other words, $\psi'(\alpha_i)=\psi(\alpha_i)\in\fS_d\geo$. Since $\alpha_i=\alpha'_{i,1}\dots\alpha'_{i,k'_i}$ in $\pi_1(\CmP',*_{\uU})$, we also conclude the equality $\psi'(\alpha_i)=\psi'(\alpha'_{i,1})\dots\psi'(\alpha'_{i,k'_i})$ in $\fS_d$.
This concludes the proof of continuity of $\tilde\cv$.

\subsection{Continuity of \texorpdfstring{$\bc$}{bc}}
\label{subsec:bccontinuous}
We fix $(P,\psi)\in\Hur(\C;(\fS_d\geo)_+)_{\klud_g}$, we write $P=\set{z_1,\dots,z_k}$, we choose $\epsilon>0$ such that the open squares
\[
z_i+(-5\epsilon,5\epsilon)^2\subset\C
\]
centred at $z_i$ and of side $10\epsilon$ are pairwise disjoint for $1\le i\le k$, we set $U_i=z_i+(-\epsilon,\epsilon)^2\subset\C$ and thus obtain an adapted covering $\uU=(U_1,\dots,U_k)$ of $P$. Our aim is to prove that $\bc$ is continuous on the normal neighbourhood $\fU(P,\psi;\uU)\subset \Hur(\C;(\fS_d\geo)_+)_{\klud_g}$. We
consider $*_{\uU}$ as preferred basepoint for fundamental groups of subsets of $\C$ containing $\C\setminus\uU$: in all of the following arguments, basepoints for fundamental groups can always be translated to $*_{\uU}$ along a straight segment.

Fix simple loops $\alpha_1,\dots,\alpha_k\subseteq\C\setminus\uU$ that are disjoint away from $*_{\uU}$ and represent an admissible generating set of $\pi_1(\CmP,*_{\uU})$, and let $D_i\subset\C$ be the disc bounded by $\alpha_i$. Define also
$\sigma_i=\psi(\alpha_i)\in\fS_d\geo$, and consider $\sigma_i$ also as element in $\fS_d$. The permutation $\sigma_i$ decomposes as a product of cycles $c_{i,1}\dots c_{i,\lambda_i}$, and we denote by $1\le d_{i,j}\le d$ the length of $c_{i,j}$.

Let $\mathring{\cS}$ be the total space of the $d$-fold 
covering of $\C\setminus\uU$ associated with the monodromy $\phi\colon\pi_1(\CmP,*_{\uU})\to\fS_d$ induced by $\psi\colon\fQ(P)\to\fS_d\geo$, and compactify $\mathring{\cS}$ over $\infty$ by adjoining $n$ points $Q_1,\dots,Q_n$, in order to obtain a compact surface $\cS$ of genus $g$ with $kd-h$ boundary components. There is a branched cover map $f\colon\cS\to\C P^1\setminus\uU$ of degree $d$, branching, possibly, only over $\infty$. There are $\lambda_i$ boundary components of $\cS$ lying over $\del U_i$:
more precisely, $\partial\cS$ contains a component $\partial_{i,j}\cS$ covering $\del U_i$ with degree $d_{i,j}$, for all $1\le i\le k$ and $1\le j\le \lambda_i$: more precisely, $\partial_{i,j}\cS$ is the only component in $\del\cS$ that can be connected inside $f^{-1}(D_i\setminus U_i)$ to the $d_{i,j}$ points of $f^{-1}(*_{\uU}))$ labeled by the elements of the cycle $c_{i,j}$.

We denote by $5U_i$ the square $z_i+(-5\epsilon,5\epsilon)^2$, and we let $5U_i\setminus U_i$ be a collar neighbourhood of $\del U_i$ in $\C\setminus\uU$.
We denote by $V_{i,j}$ the connected component of $f^{-1}(5U_i\setminus U_i)\subset\cS$ containing $\del_{i,j}\cS$. Up to translations and dilations in $\C$, we can identify $5U_i\setminus U_i$ with $(-2,3)^2\setminus\mcR$, where we recall that $\mcR$ denotes the unit square $(0,1)^2\subset\C$. 
Similarly, after choosing an element in the cycle $c_{i,j}$, we can identify $V_{i,j}$ with the standard $d_{i,j}$-fold cyclic cover of $(-2,3)^2\setminus\mcR$, with trivialised fibre over $(0,-1)=-\sqrt{-1}$.

We can then consider $\cS\times\fU(P,\psi;\uU)$ as a trivial bundle over $\fU(P,\psi;\uU)$ with fibre a surface of genus $g$ with   $(\sum_{i=1}^k\lambda_i)$ boundary components; moreover each boundary component of each fibre has a collar neighbourhood parametrised by a suitable cyclic cover of $(-2,3)^2\setminus\mcR$.

Now, similarly as in \cite[Definition 3.15]{Bianchi:Hur2}, we define for all $1\le i\le k$ a continuous map $\fri^{\C}_{D_i}\colon \fU(P,\psi;\uU)\to\Hur(U_i;(\fS_d\geo)_+)_{\hat\sigma_i}$, where $\Hur(U_i;(\fS_d\geo)_+)_{\hat\sigma_i}$
denotes the subspace of $\Hur(\C;(\fS_d\geo)_+)_{\hat\sigma_i}$ of configurations supported on $U_i$, and $\hat\sigma_i\in\widehat{\fS_d\geo}$ is the image of $\sigma_i\in\fS_d\geo$ along the canonical inclusion $\fS_d\geo\hookrightarrow\widehat{\fS_d\geo}$.
Given $(P',\psi')\in\fU(P,\psi;\uU)$, we send it to the pair $(P'',\psi'')$, where
\begin{itemize}
 \item $P''=P'\cap U_i$;
 \item $\psi''\colon\fQ(P'')\to\fS_d\geo$ is the unique (augmented) map of PMQs that agrees with $\psi'$ on all loops in $D_i\setminus P''$ representing classes both in $\fQ(P'')$ and $\fQ(P')$.
\end{itemize}
We can then identify the spaces
\[
 \Hur(U_i;(\fS_d\geo)_+)_{\hat\sigma_i}\cong \prod_{j=1}^{\lambda_i}\Hur(U_i;(\fS_{c_{i,j}}\geo)_+)_{\hat c_{i,j}},
\]
and thus define maps $\fri^{\C}_{D_i,j}\colon \fU(P,\psi;\uU)\to\Hur(U_i;(\fS_{c_{i,j}}\geo)_+)_{\hat c_{i,j}}$ by composing $\fri^{\C}_{D_i}$ with the $j$\textsuperscript{th} projection, for all $1\le j\le \lambda_i$.

By rescaling and translating, we can identify $U_i$ with $\mcR$, and thus obtain maps  $\fri^{\C}_{\mcR,i,j}\colon \fU(P,\psi;\uU)\to\Hur(\mcR;(\fS_{c_{i,j}}\geo)_+)_{\hat c_{i,j}}$ for all $1\le j\le \lambda_i$. We can then use the argument from Subsection \ref{subsec:monic} and identify $\Hur(\mcR;(\fS_{c_{i,j}}\geo)_+)_{\hat c_{i,j}}$ with the space $\NMonPol_{d_{i,j}}(\mcR)$ of monic polynomials of degree $d_{i,j}$ whose critical values in $\C$ actually lie in $\mcR$.
Define
\[
\scD_d=\set{(f,z)\in\NMonPol_d(\mcR)\times\C\,|\, f(z)\in [-2,3]^2};
\]
then the projection $\pi_d\colon \scD_d\to \NMonPol_d$ given by $\pi_d\colon (f,z)\mapsto f$ is a holomorphic disc bundle over $\NMonPol_d$; moreover each fibre $\pi_d^{-1}(f)$ has a collar neighbourhood of its boundary, namely the subspace
$V_f=\set{z\in\C\,|\,f(z)\in [-2,3]^2\setminus\cR}$. This subspace is canonically identified with the standard $d$-fold cyclic
cover of $[-2,3]^2\setminus\cR$, for varying $f\in\NMonPol_d(\mcR)$.

For $(P',\psi')\in\fU(P,\psi;\uU)$ we can now retrieve $\bc(P',\psi')$ as follows:
\begin{itemize}
 \item we take a copy of $\cS$, which is a Riemann surface with $n$ directed marked points and with $(\sum_{i=1}^k\lambda_i)$ boundary curves, endowed with parametrised collar neighbourhoods;
 \item we consider, for each $1\le i\le k$ and $1\le j\le \lambda_i$, the polynomial $f_{i,j}(P',\psi'):=\fri^{\C}_{\mcR,i,j}(P',\psi')$, and take a copy of the disc $\pi_{d_{i,j}}^{-1}(f_{i,j}(P',\psi'))$, which is endowed with a parametrised collar neighbourhood $V_{f_{i,j}(P',\psi')}$ of its boundary;
 \item we use these collar neighbourhoods to glue $\cS$ with the $(\sum_{i=1}^k\lambda_i)$ discs: more precisely, we identify, for each $1\le i\le k$ and $1\le j\le \lambda_i$, the points of $V_{i,j}$ and $V_{f_{i,j}(P',\psi')}$ with an equal image in the standard $d_{i,j}$-fold cyclic cover of $(-2,3)^2\setminus\cR$ (note that we glue along an annulus which is open on both sides);
 \item we obtain a Riemann surface of type $\sgn$, represented by a Riemann structure $\fr'\in\Riem(\sgn)$; moreover we have an $\ud$-directed meromorphic function $f'\colon(\sgn,\fr')\to \C P^1$ which is defined by $f\colon\cS\to\C P^1\setminus\uU$ on $\cS$, and is defined by suitably translating and rescaling the polynomial $f'_{i,j}(P',\psi')$ on $\pi_{d_{i,j}}^{-1}(f_{i,j}(P',\psi'))$.
\end{itemize}
We note that all operations are continuous, in particular glueing a family of discs to a fixed surface with boundary along parametrised collar neighbourhoods of the boundary curves yields a family of closed Riemann surfaces. This concludes the proof that $\bc$ is continuous.

\bibliography{Bibliography4.bib}
\bibliographystyle{alpha}

\end{document}